\documentclass[nospthms, 10pt, final]{svjour3_arXiv}%

\usepackage[a4paper, headheight = 2.0cm, tmargin=1.5cm, bmargin = 1.5cm, hmargin={2.25cm,2.25cm} ]{geometry}

\smartqed

\usepackage{enumitem}
\usepackage{amsmath}
\usepackage{amsfonts}
\usepackage{amssymb}
\usepackage{graphicx}
\usepackage{lineno}
\usepackage[unicode]{hyperref}
\usepackage{color}%
\usepackage[usenames,dvipsnames,table]{xcolor}

\newtheorem{theorem}{Theorem}[section]
\newtheorem{corollary}{Corollary}[section]

\newtheorem{example}{Example}[section]
\newtheorem{remark}{Remark}[section]
\newtheorem{lemma}{Lemma}[section]
\newtheorem{algorithm}{Algorithm}[section]
\newtheorem{proposition}{Proposition}[section]

\providecommand{\U}[1]{\protect\rule{.1in}{.1in}}
\newenvironment{proof}[1][Proof]{\noindent\textbf{#1.} }{\ \rule{0.5em}{0.5em}}

\newcommand{\DPurple}[1]{{\color[rgb]{0.6, 0.0, 0.8}{#1}}}

\hypersetup{
	colorlinks = true,
	citecolor = blue,
	filecolor=black,
	linkcolor = red,
	anchorcolor = red,
	urlcolor= red
}

\usepackage{caption, listings}

\DeclareCaptionFont{white}{\color{white}}
\DeclareCaptionFormat{listingcaptionformat}{\colorbox[rgb]{0.15, 0.38, 0.0}{\parbox{0.985\textwidth}{\hspace{0pt}#1#2#3}}}
\DeclareCaptionLabelFormat{listingcaptionlabelformat}{#1 #2}
\makeatletter
\lst@Key{countblanklines}{true}[t]%
{\lstKV@SetIf{#1}\lst@ifcountblanklines}

\lst@AddToHook{OnEmptyLine}{%
	\lst@ifnumberblanklines\else%
	\lst@ifcountblanklines\else%
	\advance\c@lstnumber-\@ne\relax%
	\fi%
	\fi}
\makeatother

\newcommand{\mnumberedcode}[4]
{
	\lstset
	{
		mathescape = true,
		escapeinside={(*@}{@*)},
		firstnumber = {#1},
		language = Matlab,
		showstringspaces = false,
		numberblanklines = false,
		countblanklines = false,
		keywordstyle = {\color[rgb]{0.5, 0.0, 0.0}},
		commentstyle = {\color[rgb]{0.0, 0.5, 0.0}},
		basicstyle = \ttfamily\scriptsize,
		numbers = left, numberstyle = \tiny, numbersep = 5 pt,
		classoffset = 0,
		morekeywords = {syms, matlabFunction, ones, cell},
		classoffset = 1,
		morekeywords = {M, Doolittle, dim, phi, b, v, T, rho, L, U, LI, UI, mu, lambda, W_reversed_beta, beta, W, c, gamma, B_c, arc, hodograph, tangent, p, psi, omega},
		keywordstyle = {\color[rgb]{0.0, 0.0, 0.70}},
		classoffset = 0
	}
	\lstinputlisting[label = #3, caption = #4]{#2}
}

\usepackage{wallpaper}
\journalname{}

\begin{document}
	
	\LRCornerWallPaper{0.225}{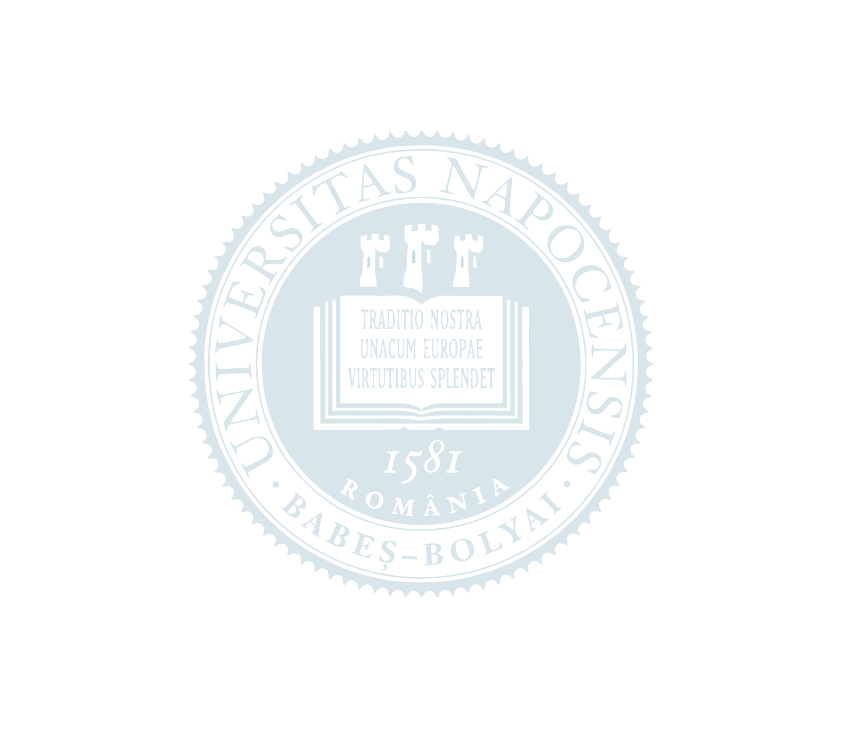}
	
\setcounter{tocdepth}{3}
\allowdisplaybreaks
\renewcommand{\thelstlisting}{\thesection.\arabic{lstlisting}}

\title{Control point based exact description of curves and surfaces\\in extended Chebyshev spaces}

\author{\'{A}goston R\'{o}th}

\institute{
\'A. R\'oth \at Department of Mathematics and Computer Science, Babe\c{s}--Bolyai University, RO--400084 Cluj-Napoca, Romania \\
Tel.: +40-264-405300\\
Fax:  +40-264-591906\\
\email{agoston\_roth@yahoo.com}
}

\date{Submitted to arXiv on \today}

\maketitle

\begin{abstract}
Extended Chebyshev spaces that also comprise the constants represent large families of functions that can be used in real-life modeling or engineering applications that also involve important (e.g.\ transcendental) integral or rational curves and surfaces. Concerning computer aided geometric design, the unique normalized B-bases of such vector spaces ensure optimal shape preserving properties, important evaluation or subdivision algorithms and useful shape parameters. Therefore, we propose global explicit formulas for the entries of those transformation matrices that map these normalized B-bases to the traditional (or ordinary) bases of the underlying vector spaces. Then, we also describe general and ready to use control point configurations for the exact representation of those traditional integral parametric curves and (hybrid) surfaces that are specified by coordinate functions given as (products of separable) linear combinations of ordinary basis functions. The obtained results are also extended to the control point and weight based exact description of the rational counterpart of these integral parametric curves and surfaces. The universal applicability of our methods is presented through polynomial, trigonometric, hyperbolic or mixed extended Chebyshev vector spaces.

\keywords{Extended Chebyshev vector spaces \and Curves and surfaces \and Normalized B-basis functions \and Basis transformation \and Control point based exact description}
\subclass{65D17 \and 68U07}
\end{abstract}

\section{Introduction\label{sec:introduction}}

Normalized B-bases (a comprehensive study of which can be found in
\cite{Pena1999} and references therein) are normalized totally positive bases
that imply optimal shape preserving properties for the representation of
curves described as convex combinations of control points and basis functions.
Similarly to the classical Bernstein polynomials of degree $n\in\mathbb{N}$ -- that in fact form the normalized B-basis of the vector space of polynomials of degree at most $n$ on the interval $\left[  0,1\right]  $, cf.
\cite{Carnicer1993} -- normalized B-bases provide shape preserving properties
like closure for the affine transformations of the control points, convex
hull, variation diminishing (which also implies the preservation of convexity
of plane control polygons), endpoint interpolation, monotonicity preserving,
hodograph and length diminishing, and a recursive corner cutting algorithm
(also called B-algorithm) that is the analogue of the de Casteljau algorithm
of B\'{e}zier curves. Among all normalized totally positive bases of a given
vector space of functions the normalized B-basis is the least variation
diminishing and the shape of the generated curve more mimics that of its control
polygon. Important curve design algorithms like evaluation, subdivision,
degree elevation or knot insertion are in fact corner cutting algorithms that
can be treated in a unified way by means of B-algorithms induced by B-bases.

Curve and surface modeling tools based on non-polynomial normalized B-bases also ensure further advantages like: possible shape or design parameters; singularity free exact parametrization (e.g.\ parametrization of
conic sections may correspond to natural arc-length parametrization); higher or even infinite order of precision concerning (partial) derivatives; ordinary (i.e., traditionally parametrized) integral curves and surfaces can exactly be described by means of control points without any additional weights (the calculation of which, apart of some simple cases, is cumbersome for the designer); important transcendental curves and surfaces which are of interest in real-life applications can also be exactly represented (the standard rational B\'{e}zier or NURBS models cannot encompass these
geometric objects). Moreover, concerning condition numbers and stability, a normalized B-basis is the unique normalized totally positive basis that is optimally stable among all non-negative bases of a given vector space of functions, cf. \cite[Corollary 3.4, p. 89]{Pena1999}. These advantageous properties make normalized B-bases ideal blending function
system candidates for curve (and surface) modeling.

Besides their interest in the classical contexts of CAGD and approximation theory, normalized B-bases and their spline counterparts have also been used in isogeometric analysis recently (consider e.g.\ \cite{ManniPelosiSampoli2011} and references therein). Compared with classical final element methods, isogeometric analysis provides several advantages when one describes the geometry by generalized B-splines and invokes an isoparametric approach in order to approximate the unknown solutions of differential equations (e.g.\ of Poisson type problems) or Dirichlet boundary conditions by the same type of functions.

Let $n\geq1$ be a fixed integer and consider the extended Chebyshev (EC) system%
\begin{equation}
\mathcal{F}_{n}^{\alpha,\beta}=\left\{  \varphi
_{n,i}\left(  u\right)  :u\in\left[  \alpha,\beta\right]  \right\}  _{i=0}%
^{n},~\varphi_{n,0}	\equiv 1,~-\infty<\alpha<\beta<\infty\label{eq:ordinary_basis}%
\end{equation}
of basis functions in $C^{n}\left(  \left[  \alpha,\beta\right]  \right)  $,
i.e., by definition \cite{KarlinStudden1966}, for any integer $0\leq r\leq n$, any strictly increasing
sequence of knot values $\alpha\leq u_{0}<u_{1}<\ldots<u_{r}\leq\beta$, any
positive integers (or multiplicities) $\left\{  m_{k}\right\}  _{k=0}^{r}$
such that $\sum_{k=0}^{r}m_{k}=n+1$, and any real numbers $\left\{
\xi_{k,\ell}\right\}  _{k=0,~\ell=0}^{r,~m_{k}-1}$ there always exists a
unique function
\begin{equation}
f:=\sum_{i=0}^{n}\lambda_{n,i}\varphi_{n,i}\in\mathbb{S}_{n}^{\alpha,\beta}:=\left\langle\mathcal{F}_{n}^{\alpha,\beta
}\right\rangle:=\operatorname{span}\mathcal{F}_{n}^{\alpha,\beta},~\lambda_{n,i}\in\mathbb{R},~i=0,1,\ldots,n \label{eq:unique_solution}%
\end{equation}
that satisfies the conditions of the Hermite interpolation problem%
\begin{equation}
f^{\left(  \ell\right)  }\left(  u_{k}\right)  =\xi_{k,\ell},~\ell
=0,1,\ldots,m_{k}-1,~k=0,1,\ldots,r. \label{eq:Hermite_interpolation_problem}%
\end{equation}
In what follows, we assume that the sign-regular determinant of the coefficient
matrix of the linear system (\ref{eq:Hermite_interpolation_problem}) of
equations is strictly positive for any permissible parameter settings introduced above.
Under these circumstances, the vector space $\mathbb{S}_{n}^{\alpha,\beta}$ of functions is
called an EC space of dimension $n+1$. In terms of zeros, this
definition means that any non-zero element of $\mathbb{S}_{n}^{\alpha,\beta}$
vanishes at most $n$ times in the interval $\left[  \alpha,\beta\right]  $. Such spaces and their corresponding spline counterparts have been widely studied, consider e.g.\ articles \cite{Lyche1985,CarnicerPena1994,Mazure1999,Mazure2001,MainarPenaSanchez2001,LuWangYang2002,CarnicerMainarPena2004,MainarPena2004,CostantiniLycheManni2005,CarnicerMainarPena2007,MainarPena2010} and many other references therein.

Hereafter we will also refer to $\mathcal{F}_{n}^{\alpha,\beta}$ as the
ordinary basis of $\mathbb{S}_{n}^{\alpha,\beta}$. Using \cite[Theorem
5.1]{CarnicerPena1995} it follows that the vector space $\mathbb{S}%
_{n}^{\alpha,\beta}$ also has a strictly totally positive basis, i.e., a basis
such that all minors of all its collocation matrices are strictly positive.
Since the constant function $1\equiv\varphi_{n,0}\in\mathbb{S}_{n}%
^{\alpha,\beta}$, the aforementioned strictly positive basis is normalizable,
therefore the vector space $\mathbb{S}_{n}^{\alpha,\beta}$ also has a unique
non-negative normalized B-basis%
\begin{equation}
\mathcal{B}_{n}^{\alpha,\beta}=\left\{  b_{n,i}\left(  u\right)  :u\in\left[
\alpha,\beta\right]  \right\}  _{i=0}^{n} \label{eq:B-basis}%
\end{equation}
that besides the identity%
\begin{equation}
\sum_{i=0}^{n}b_{n,i}\left(  u\right)  \equiv1,~\forall u\in\left[
\alpha,\beta\right]  \label{eq:partition_of_unity}%
\end{equation}
also fulfills the properties%
\begin{align}
b_{n,0}\left(  \alpha\right)   &  =b_{n,n}\left(  \beta\right)
=1,\label{eq:endpoint_interpolation}\\
b_{n,i}^{\left(  j\right)  }\left(  \alpha\right)   &  =0,~j=0,\ldots
,i-1,~b_{n,i}^{\left(  i\right)  }\left(  \alpha\right)
>0,\label{eq:Hermite_conditions_0}\\
b_{n,i}^{\left(  j\right)  }\left(  \beta\right)   &  =0,~j=0,1,\ldots
,n-1-i,~\left(  -1\right)  ^{n-i}b_{n,i}^{\left(  n-i\right)  }\left(
\beta\right)  >0 \label{eq:Hermite_conditions_alpha}%
\end{align}
conform \cite[Theorem 5.1]{CarnicerPena1995} and \cite[Equation (3.6)]%
{Mazure1999}.

Using the normalized B-basis $\mathcal{B}_{n}^{\alpha,\beta}$ of
$\mathbb{S}_{n}^{\alpha,\beta}$, one of our objectives is to provide explicit closed formulas for the control point based exact description of integral curves that are
specified with coordinate functions given in traditional parametric form in
the ordinary basis $\mathcal{F}_{n}^{\alpha,\beta}$ of the same vector space. Based on homogeneous coordinates and central projection, we also propose an
algorithm for the control point (and weight) based exact description of the
rational counterpart of these ordinary integral curves. Results
will also be extended to the exact representation of families of (hybrid)
integral and rational surfaces that are exclusively given in each
of their variables by using ordinary EC basis
functions of the type (\ref{eq:ordinary_basis}).

To the best of the author's knowledge, the coefficient based exact representation of ordinary (rational) functions, curves and surfaces by means of the (rational or spline counterpart) of the normalized B-basis of an arbitrary EC space (that also comprises the constant functions) was not considered in such a general unified context.  Without providing an exhaustive survey, so far the presented problem appears in the literature for example
in case of conversion algorithms related to Bernstein polynomials, monomials and the classical families of orthogonal Jacobi, Gegenbauer, Legendre, Chebyshev, Laguerre and Hermite polynomials \cite{CargoShisha1966,BarrioPena2004}
in special lower dimensional vector spaces (e.g.\ in \cite{Zhang1996,MainarPenaSanchez2001,CarnicerMainarPena2003,CarnicerMainarPena2006,RomaniSainiAlbrecht2014});
in case of conical and helical arcs, of catenaries, of patches on all types of quadrics and of helicoidal surfaces (e.g.\ in \cite{PottmannWagner1994,LuWangYang2002});
of certain (rational) trigonometric curves of arbitrarily finite order like epi- and hypotrochoidal arcs \cite{Sanchez1999}, or segments of offset-rational sinusoidal spirals, arachnidas and epi spirals \cite{Sanchez2002};
or more recently, in case of arbitrary trigonometric and hyperbolic (rational) polynomials, curves, (hybrid) surfaces and volumes of finite order \cite{Roth2015}.

The rest of the paper is organized as follows. Section
\ref{sec:main_results_and_remarks} lists our main results, namely it describes
closed formulas for the basis transformation that maps the normalized B-basis
$\mathcal{B}_{n}^{\alpha,\beta}$ of the vector space $\mathbb{S}_{n}%
^{\alpha,\beta}$ to its ordinary basis $\mathcal{F}_{n}^{\alpha,\beta}$ and also specifies control point configurations for the exact
representation of certain large classes of integral and rational curves and
surfaces that are specified in traditional parametric form by
means of ordinary bases like $\mathcal{F}_{n}^{\alpha,\beta}$. Although the presented results are mainly of theoretical interest, Section \ref{sec:main_results_and_remarks} also studies the computational complexity of the proposed basis conversion formulas and -- compared with alternative cubic time numerical methods like curve interpolation or least square approximation -- points out that these can more efficiently be implemented up to $n=15$. Section
\ref{sec:examples} emphasizes the universal applicability of the general basis
transformation described in Section \ref{sec:main_results_and_remarks} with
examples that can be compared to presumably already existing results in the
literature. This section considers EC vector spaces of
functions that may be important in computer aided geometric design, in
engineering, in (projective) geometry, in (numerical) analysis or in
approximation theory. The proofs of all theoretical results stated in Sections
\ref{sec:main_results_and_remarks}--\ref{sec:examples} can be found in Section
\ref{sec:proofs}. In the end, Section \ref{sec:final_remarks} closes the paper
with our final remarks. Based on the general context of the manuscript, Appendix \ref{sec:Bernstein_to_monomials} recalls the classic transformation matrix that maps the Bernstein polynomials of degree $n$ to the corresponding ordinary power basis of the vector space of traditional polynomials, while Appendix \ref{sec:implementation_details} provides implementation details by means of a simple Matlab example.

\section{Main results and remarks\label{sec:main_results_and_remarks}}

At first, we provide explicit formulas for the transformation of the
normalized B-basis $\mathcal{B}_{n}^{\alpha,\beta}$ of the vector space
$\mathbb{S}_{n}^{\alpha,\beta}$ to its ordinary basis $\mathcal{F}_{n}^{\alpha,\beta}$.

\begin{theorem}
	[General basis transformation]\label{thm:basis_transformation}The matrix form of the
	linear transformation that maps the normalized B-basis $\mathcal{B}%
	_{n}^{\alpha,\beta}$ to the ordinary basis $\mathcal{F}_{n}^{\alpha,\beta}$ is%
	\begin{equation}
	\left[
	\begin{array}
	[c]{c}%
	\varphi_{n,i}\left(  u\right)
	\end{array}
	\right]  _{i=0}^{n}=\left[  t_{i,j}^{n}\right]  _{i=0,~j=0}^{n,~n}\cdot\left[
	\begin{array}
	[c]{c}%
	b_{n,i}\left(  u\right)
	\end{array}
	\right]  _{i=0}^{n},~\forall u\in\left[  \alpha,\beta\right]  ,
	\label{eq:basis_transformation}%
	\end{equation}
	where $t_{0,j}^{n}=1,~j=0,1,\ldots,n$ and $t_{i,0}^{n}= \varphi_{n,i}\left(  \alpha\right),~
	t_{i,n}^{n}=  ~\varphi_{n,i}\left(  \beta\right)  ,~i=0,1,\ldots
	,n$, while%
	\begin{align}
	t_{i,j}^{n}=  &  ~\varphi_{n,i}\left(  \alpha\right)  -\frac{1}{b_{n,j}%
		^{\left(  j\right)  }\left(  \alpha\right)  }\cdot\left.  \sum_{r=1}%
	^{j-1}\frac{\varphi_{n,i}^{\left(  r\right)  }\left(  \alpha\right)  }%
	{b_{n,r}^{\left(  r\right)  }\left(  \alpha\right)  }\right(  b_{n,r}^{\left(
		j\right)  }\left(  \alpha\right)  +\label{eq:first_half}\\
	&  ~\left.  +\sum_{\ell=1}^{j-r-1}\left(  -1\right)  ^{\ell}\sum
	_{r<k_{1}<k_{2}<\ldots<k_{\ell}<j}\frac{b_{n,r}^{\left(  k_{1}\right)
		}\left(  \alpha\right)  b_{n,k_{1}}^{\left(  k_{2}\right)  }\left(
		\alpha\right)  b_{n,k_{2}}^{\left(  k_{3}\right)  }\left(  \alpha\right)
		\ldots b_{n,k_{\ell-1}}^{\left(  k_{\ell}\right)  }\left(  \alpha\right)
		b_{n,k_{\ell}}^{\left(  j\right)  }\left(  \alpha\right)  }{b_{n,k_{1}%
		}^{\left(  k_{1}\right)  }\left(  \alpha\right)  b_{n,k_{2}}^{\left(
		k_{2}\right)  }\left(  \alpha\right)  \ldots b_{n,k_{\ell}}^{\left(  k_{\ell
		}\right)  }\left(  \alpha\right)  }\right)  +\frac{\varphi_{n,i}^{\left(
		j\right)  }\left(  \alpha\right)  }{b_{n,j}^{\left(  j\right)  }\left(
	\alpha\right)  },\nonumber\\ i=&~1,2,\ldots,n,~j=1,2,\ldots,\left\lfloor \frac{n}{2}\right\rfloor
,\nonumber\\
& \nonumber\\
t_{i,n-j}^{n}=  &  ~\varphi_{n,i}\left(  \beta\right)  -\frac{1}%
{b_{n,n-j}^{\left(  j\right)  }\left(  \beta\right)  }\cdot\left.  \sum
_{r=1}^{j-1}\frac{\varphi_{n,i}^{\left(  r\right)  }\left(  \beta\right)
}{b_{n,n-r}^{\left(  r\right)  }\left(  \beta\right)  }\right(  b_{n,n-r}%
^{\left(  j\right)  }\left(  \beta\right)  +\label{eq:last_half}\\
&  ~\left.  +\sum_{\ell=1}^{j-r-1}\left(  -1\right)  ^{\ell}\sum
_{r<k_{1}<k_{2}<\ldots<k_{\ell}<j}\frac{b_{n,n-r}^{\left(  k_{1}\right)
	}\left(  \beta\right)  b_{n,n-k_{1}}^{\left(  k_{2}\right)  }\left(
	\beta\right)  b_{n,n-k_{2}}^{\left(  k_{3}\right)  }\left(  \beta\right)
	\ldots b_{n,n-k_{\ell-1}}^{\left(  k_{\ell}\right)  }\left(  \beta\right)
	b_{n,n-k_{\ell}}^{\left(  j\right)  }\left(  \beta\right)  }{b_{n,n-k_{1}%
	}^{\left(  k_{1}\right)  }\left(  \beta\right)  b_{n,n-k_{2}}^{\left(
	k_{2}\right)  }\left(  \beta\right)  \ldots b_{n,n-k_{\ell}}^{\left(  k_{\ell
	}\right)  }\left(  \beta\right)  }\right) \nonumber\\
&  ~+\frac{\varphi_{n,i}^{\left(  j\right)  }\left(  \beta\right)  }%
{b_{n,n-j}^{\left(  j\right)  }\left(  \beta\right)  },\,~i=  ~1,2,\ldots,n,~j=1,2,\ldots,\left\lfloor \frac{n}{2}\right\rfloor
.\nonumber
\end{align}
\end{theorem}

\begin{remark}[Evaluation]
	If in formulas (\ref{eq:first_half}) or (\ref{eq:last_half}),
	for some $\ell=1,2,\ldots,j-r-1$ (with $r=1,2,\ldots,j-1$ and $j=1,2,\ldots
	,\left\lfloor \frac{n}{2}\right\rfloor $) there exist no integers
	$k_{1},k_{2},\ldots,k_{\ell}$ such that $r<k_{1}<k_{2}<\ldots<k_{\ell}<j$
	then, by convention, the summation corresponding to $\ell$ equals $0$.
	If $n=2z\geq2$, then for $j=z$ one can evaluate the entries
	$\left[  t_{i,z}\right]  _{i=1}^{n}$ of the middle column by using either
	of these formulas, since the $z$th coefficients of the ordinary basis functions (\ref{eq:ordinary_basis}) in the normalized B-basis (\ref{eq:B-basis}) are unique. 
\end{remark}

Except some special but important cases, in general, one does not know the closed form of the normalized B-basis (\ref{eq:B-basis}) of  $\mathbb{S}_{n}^{\alpha,\beta}$. In case of EC spaces of traditional, trigonometric or hyperbolic polynomials of finite degree we have explicit closed formulas cf.\ \cite{Carnicer1993}, \cite{Sanchez1998} and \cite{ShenWang2005}, respectively; in case of a special class of mixed (e.g.\ algebraic trigonometric, algebraic hyperbolic, or both trigonometric and hyperbolic) EC spaces these functions appear in recursive integral form cf.\ \cite{MainarPena2010} and references therein; while the most general (determinant based) formulas that can be applied in such spaces was published in \cite{Mazure1999}. 
Thus, concerning the evaluation of (\ref{eq:first_half}) and (\ref{eq:last_half}), in general, one can differentiate the formulas presented in \cite[Theorem 3.4, p. 658]{Mazure1999} in order to calculate the higher order derivatives of the normalized B-basis functions (\ref{eq:B-basis}) at the endpoints of the interval $\left[\alpha,\beta\right]$. Namely, by using the function
\[
\phi\left(u\right) := 
\left[
\begin{array}{cccc}
\varphi_{n,1}\left(u\right)~&
\varphi_{n,2}\left(u\right)~&
\cdots~&
\varphi_{n,n}\left(u\right)
\end{array}
\right]^T,\,u\in\left[\alpha,\beta\right],
\]
one has to substitute the parameter values $u=\alpha$ and $u=\beta$ into the derivative formulas
\begin{align}
\label{eq:general_derivative_0}
b_{n,0}^{\left(j\right)}\left(u\right) = &~
\frac
{
	\det
	\left[
	\begin{array}{cccc}
	\phi^{\left(1\right)}(\beta)~&
	\cdots~&
	\phi^{\left(n-1\right)}\left(\beta\right)~&
	\phi^{\left(j\right)}\left(u\right)
	\end{array}
	\right]
}
{
	\det
	\left[
	\begin{array}{cccc}
	\phi^{\left(1\right)}(\beta)~&
	\cdots~&
	\phi^{\left(n-1\right)}\left(\beta\right)~&
	\phi\left(\alpha\right)-\phi\left(\beta\right)
	\end{array}
	\right]
},
\\
\label{eq:general_derivative_n}
b_{n,n}^{\left(j\right)}\left(u\right) = &~
\frac
{
	\det
	\left[
	\begin{array}{cccc}
	\phi^{\left(1\right)}(\alpha)~&
	\cdots~&
	\phi^{\left(n-1\right)}\left(\alpha\right)~&
	\phi^{\left(j\right)}\left(u\right)
	\end{array}
	\right]
}
{
	\det
	\left[
	\begin{array}{cccc}
	\phi^{\left(1\right)}(\alpha)~&
	\cdots~&
	\phi^{\left(n-1\right)}\left(\alpha\right)~&
	\phi\left(\beta\right)-\phi\left(\alpha\right)
	\end{array}
	\right]
},
\\
\label{eq:general_derivative_i}
b_{n,i}^{\left(j\right)}\left(u\right) = &~
\frac
{
	\det
	\left[
	\begin{array}{cccccccc}
	\phi^{\left(1\right)}(\alpha)~&
	\cdots~&
	\phi^{\left(i-1\right)}\left(\alpha\right)~&
	\phi^{\left(i\right)}\left(\alpha\right)~&
	\phi^{\left(1\right)}\left(\beta\right)~&
	\cdots~&
	\phi^{\left(n-i-1\right)}\left(\beta\right)~&
	\phi^{\left(n-i\right)}\left(\beta\right)
	\end{array}
	\right]
}
{
	\det
	\left[
	\begin{array}{cccccccc}
	\phi^{\left(1\right)}(\alpha)~&
	\cdots~&
	\phi^{\left(i-1\right)}\left(\alpha\right)~&
	\phi\left(\beta\right)-\phi\left(\alpha\right)~&
	\phi^{\left(1\right)}\left(\beta\right)~&
	\cdots~&
	\phi^{\left(n-i-1\right)}\left(\beta\right)~&
	\phi^{\left(n-i\right)}\left(\beta\right)
	\end{array}
	\right]
}
\\	
&~\cdot
\frac
{
	\det
	\left[
	\begin{array}{cccccccc}
	\phi\left(\beta\right)-\phi\left(\alpha\right)~&
	\phi^{\left(1\right)}(\alpha)~&
	\cdots~&
	\phi^{\left(i-1\right)}\left(\alpha\right)~&
	\phi^{\left(j\right)}\left(u\right)~&
	\phi^{\left(1\right)}\left(\beta\right)~&
	\cdots~&
	\phi^{\left(n-i-1\right)}\left(\beta\right)
	\end{array}
	\right]
}
{
	\det
	\left[
	\begin{array}{cccccccc}
	\phi\left(\beta\right)-\phi\left(\alpha\right)~&
	\phi^{\left(1\right)}(\alpha)~&
	\cdots~&
	\phi^{\left(i-1\right)}\left(\alpha\right)~&
	\phi^{\left(i\right)}\left(\alpha\right)~&
	\phi^{\left(1\right)}\left(\beta\right)~&
	\cdots~&
	\phi^{\left(n-i-1\right)}\left(\beta\right)
	\end{array}
	\right]		
}
\nonumber
\end{align}
for all $i = 1,2,\ldots,n-1$ and $j=1,2,\ldots,n$. However, as it is also mentioned in \cite{Mazure1999}, these general relations are difficult and computationally expensive to evaluate even in the most simple cases for either arbitrarily big or general values of the order $n$. Therefore, Section \ref{sec:examples} provides explicit closed formulas for the required endpoint derivatives in several special cases. Due to properties (\ref{eq:Hermite_conditions_0}) and (\ref{eq:Hermite_conditions_alpha}), these expressions should only be used whenever one does not know the exact value of the required endpoint derivatives.

Another core result of the current section is presented in the next
statement which is an immediate corollary of Theorem
\ref{thm:basis_transformation}.

\begin{corollary}
	[Exact description of ordinary integral functions]\label{cor:ordinary_functions}Let
	$\left\{  \lambda_{i}\right\}  _{i=0}^{n}$ be real numbers and consider the
	linear combination%
	\begin{equation}
	c\left(  u\right)  =\sum_{i=0}^{n}\lambda_{i}\varphi_{n,i}\left(  u\right)
	,~u\in\left[  \alpha,\beta\right]   \label{eq:ordinary_integral_function}%
	\end{equation}
	of ordinary basis functions. Then, we have the equality %
	$
	c\left(  u\right)  \equiv\sum_{j=0}^{n}p_{j}b_{n,j}\left(  u\right)  ,~\forall
	u\in\left[  \alpha,\beta\right]  ,
	$
	where %
	$
	p_{j}=\sum_{i=0}^{n}\lambda_{i}t_{i,j}^{n},~\allowbreak{}j=0,1,\ldots,n\text{.}%
	$
	
\end{corollary}

Based on Corollary \ref{cor:ordinary_functions}, the exact description of
ordinary integral curves as convex combinations of control points and
normalized B-basis functions (\ref{eq:B-basis}) is presented in the next theorem.

\begin{theorem}
	[Exact description of ordinary integral curves]\label{thm:integral_curves}The ordinary integral parametric curve%
	\begin{equation}
	\mathbf{c}\left(  u\right)  =\sum_{i=0}^{n}%
	\boldsymbol{\lambda}%
	_{i}\varphi_{n,i}\left(  u\right)  ,~u\in\left[  \alpha,\beta\right]  ,~%
	\boldsymbol{\lambda}%
	_{i}=\left[  \lambda_{i}^{\ell}\right]  _{\ell=1}^{\delta}\in%
	\mathbb{R}
	^{\delta},~\delta\geq2 \label{eq:ordinary_integral_curve}%
	\end{equation}
	of order $n$ can be written as an EC B-curve
	\begin{equation}
	\mathbf{c}\left(  u\right)  \equiv\sum_{j=0}^{n}\mathbf{p}_{j}b_{n,j}\left(
	u\right)  ,~\forall u\in\left[  \alpha,\beta\right]  ,~\mathbf{p}_{j}=\left[
	p_{j}^{\ell}\right]  _{\ell=1}^{\delta}\in%
	\mathbb{R}
	^{\delta}, \label{eq:convex_combination}%
	\end{equation}
	of the same order, 
	where %
	$
	p_{j}^{\ell}=\sum_{i=0}^{n}\lambda_{i}^{\ell}t_{i,j}^{n},~j=0,1,\ldots
	,n,~\ell=1,2,\ldots,\delta.
	$
	
\end{theorem}

Using tensor products of convex combinations of type
(\ref{eq:convex_combination}), one can also exactly describe large families of
surfaces as it is specified in the following theorem.

\begin{theorem}
	[Exact description of ordinary integral surfaces]\label{thm:integral_surfaces}%
	Let
	\[
	\mathcal{F}_{n_{r}}^{\alpha_{r},\beta_{r}}=\left\{  \varphi_{n_{r},i_{r}}\left(  u_{r}\right)  :u_{r}\in\left[
	\alpha_{r},\beta_{r}\right]  \right\}  _{i_{r}=0}^{n_{r}},~\varphi_{n_{r},0} \equiv 1,~r=1,2
	\]
	be two ordinary EC bases of some vector spaces $\mathbb{S}%
	_{n_{r}}^{\alpha_{r},\beta_{r}}$ of functions and also consider
	their unique normalized B-bases %
	$
	\mathcal{B}_{n_{r}}^{\alpha_{r},\beta_{r}}=\left\{  b_{n_{r},j_{r}}\left(
	u_{r}\right)  :u_{r}\in\left[  \alpha_{r},\beta_{r}\right]  \right\}
	_{j_{r}=0}^{n_{r}},~r=1,2.
	$
	Denote by $[  t_{i_{r},j_{r}}^{n_{r}}]  _{i_{r}=0,~j_{r}=0}%
	^{n_{r},~n_{r}}$ the regular square matrix that transforms $\mathcal{B}%
	_{n_{r}}^{\alpha_{r},\beta_{r}}$ to $\mathcal{F}_{n_{r}}^{\alpha_{r},\beta
		_{r}}$ and consider the ordinary integral surface%
	\begin{equation}
	\mathbf{s}\left(  \mathbf{u}\right)  =\left[
	\begin{array}
	[c]{ccc}%
	s^{1}\left(  \mathbf{u}\right)   & s^{2}\left(  \mathbf{u}\right)   &
	s^{3}\left(  \mathbf{u}\right)
	\end{array}
	\right]  ^{T}\in%
	\mathbb{R}
	^{3},~\mathbf{u=}\left[  u_{r}\right]  _{r=1}^{2}\in\left[  \alpha_{1}%
	,\beta_{1}\right]  \times\left[  \alpha_{2},\beta_{2}\right]
	\label{eq:ordinary_integral_surface}%
	\end{equation}
	of order $\mathbf{n=}\left[  n_{r}\right]  _{r=1}^{2}$, where%
	\begin{equation}
	s^{\ell}\left(  \mathbf{u}\right)  =\sum_{\zeta=1}^{\sigma_{\ell}}\prod_{r=1}%
	^{2}\left(  \sum_{i_{r}=0}^{n_{r}}\lambda_{i_{r}}^{\ell,\zeta}\varphi
	_{n_{r},i_{r}}\left(  u_{r}\right)  \right)  ,~\sigma_{\ell}\geq 1, ~\ell
	=1,2,3.\label{eq:ordinary_integral_surface_cf}%
	\end{equation}
	Then, the surface
	(\ref{eq:ordinary_integral_surface}) can be written in the tensor product form
	with the EC B-surface%
	\begin{equation}
	\mathbf{s}\left(  \mathbf{u}\right)  \equiv\sum_{j_{1}=0}^{n_{1}}\sum
	_{j_{2}=0}^{n_{2}}\mathbf{p}_{j_{1},j_{2}}b_{n_{1},j_{1}}\left(  u_{1}\right)
	b_{n_{2},j_{2}}\left(  u_{2}\right)  ,~\forall\mathbf{u=}\left[  u_{r}\right]
	_{r=1}^{2}\in\left[  \alpha_{1},\beta_{1}\right]  \times\left[  \alpha
	_{2},\beta_{2}\right]  \label{eq:tensor_product}%
	\end{equation}
	of the same order, where the vectors $\mathbf{p}_{j_{1},j_{2}}=[
	p_{j_{1},j_{2}}^{\ell}]  _{\ell=1}^{3}\in%
	\mathbb{R}
	^{3}$ form the control net defined by coordinates%
	\begin{equation}
	p_{j_{1},j_{2}}^{\ell}=\sum_{\zeta=1}^{\sigma_{\ell}}\prod_{r=1}^{2}p_{j_{r}}%
	^{\ell,\zeta},~~~p_{j_{r}}^{\ell,\zeta}:=\sum_{i_{r}=0}^{n_{r}}\lambda_{i_{r}%
	}^{\ell,\zeta}t_{i_{r},j_{r}}^{n_{r}},~j_{r}=0,1,\ldots,n_{r},~r=1,2,~\zeta = 1,2,\ldots,\sigma_{\ell},~\ell
	=1,2,3.\label{eq:surface_cp}%
	\end{equation}
	
\end{theorem}

\begin{remark}
	[Exact description of ordinary integral volumes]Naturally, Theorem
	\ref{thm:integral_surfaces} can easily be extended to the control point based
	exact description of those tri- or higher variate integral multivariate
	surfaces (volumes) that are specified in traditional parametric form with
	coordinate functions described as sums of separable products of
	linear combinations of the type (\ref{eq:ordinary_integral_function}).
\end{remark}

If the denominator of the rational counterpart of the ordinary integral curve
(\ref{eq:ordinary_integral_curve}) is strictly positive, then, by means of
control points and non-negative weights of rank $1$, one can also exactly
describe ordinary rational curves as it is illustrated in the steps of the
next algorithm.

\begin{algorithm}
	[Exact description of ordinary rational curves]\label{alg:ordinary_rational_curves}%
	Consider in $\mathbb{R}^{\delta}$ the rational curve%
	\begin{equation}
	\mathbf{c}\left(  u\right)  =\frac{1}{c^{\delta+1}\left(  u\right)  }\left[
	c^{\ell}\left(  u\right)  \right]  _{\ell=1}^{\delta},~u\in\left[
	\alpha,\beta\right]  \label{eq:ordinary_rational_curve}%
	\end{equation}
	given in ordinary parametric form, where%
	\begin{align*}
	c^{\ell}\left(  u\right)   &  =\sum_{i=0}^{n}\lambda_{i}^{\ell}\varphi
	_{n,i}\left(  u\right)  ,~\ell=1,2,\ldots,\delta+1,~~
	c^{\delta+1}\left(  u\right)    >0,~\forall u\in\left[  \alpha,\beta\right]
	.
	\end{align*}
	Using the rational counterpart of EC B-curves (\ref{eq:convex_combination}), the process that provides the control point and weight
	based exact representation%
	\begin{equation}
	\mathbf{c}\left(  u\right)  \equiv\frac{\sum_{j=0}^{n}w_{j}\mathbf{p}%
		_{j}b_{n,j}\left(  u\right)  }{\sum_{r=0}^{n}w_{r}b_{n,r}\left(  u\right)
	},~\forall u\in\left[  \alpha,\beta\right]
	\label{eq:rational_curve_representation}%
	\end{equation}
	consists of the following steps:
	
	\begin{itemize}[nolistsep]
		\item apply Theorem \ref{thm:integral_curves} to the higher dimensional
		pre-image %
		$
		\mathbf{c}^{\wp}\left(  u\right)  =\left[  c^{\ell}\left(  u\right)  \right]
		_{\ell=1}^{\delta+1},~u\in\left[  \alpha,\beta\right]  ,
		$
		i.e., compute control points %
		$
		\mathbf{p}_{j}^{\wp}=\left[  p_{j}^{\ell}\right]  _{\ell=1}^{\delta
			+1},~j=0,1,\ldots,n
		$
		for the exact description of $\mathbf{c}^{\wp}$ in the pre-image space $%
		\mathbb{R}
		^{\delta+1}$;
		
		\item project the obtained control points from the origin
		$\mathbf{0}_{\delta+1}\in%
		\mathbb{R}
		^{\delta+1}$ onto the hyperplane $x^{\delta+1}=1$ that results in the control
		points %
		$
		\mathbf{p}_{j}=\dfrac{1}{p_{j}^{\delta+1}}\left[  p_{j}^{\ell}\right]
		_{\ell=1}^{\delta}\in%
		\mathbb{R}
		^{\delta}
		$
		and weights %
		$
		w_{j}=p_{j}^{\delta+1}
		$
		needed for the rational representation (\ref{eq:rational_curve_representation});
		
		\item the above generation process does not necessarily ensure the
		non-negativity of all weights, since the last coordinate of some control
		points $\mathbf{p}_{j}^{\wp}$ in the pre-image space $%
		\mathbb{R}
		^{\delta+1}$ can be negative; if this is the case, one should elevate the dimension (and consequently the order $n$ of the normalized B-basis $\mathcal{B}_{n}^{\alpha, \beta}$) of the underlying EC space with an algorithm that generates a sequence of control polygons in $\mathbb{R}^{\delta + 1}$ that converges to $\mathbf{c}^{\wp}$ which, by definition, is a
		geometric object of one branch that does not intersect the vanishing plane
		$x^{\delta+1}=0$, since the $\left(  \delta+1\right)  $th coordinate of all
		its points are strictly positive; therefore, by using proper dimension elevation methods, it is guaranteed that exists a
		finite and minimal order for which all weights are non-negative.
	\end{itemize}
\end{algorithm}

\begin{remark}[About the last step of Algorithm \ref{alg:ordinary_rational_curves}]
	If the pre-image $\mathbf{c}^{\wp}$ of (\ref{eq:ordinary_rational_curve}) is described as B-curves of type (\ref{eq:convex_combination}) by means of the normalized B-bases of the EC spaces $\mathbb{S}_n^{\alpha, \beta}\subset \mathbb{S}_{n+1}^{\alpha,\beta}$, then 
	\[
	\mathbf{c}^{\wp}\left(u\right) = \sum_{j=0}^{n}\mathbf{p}_{j}^{\wp}b_{n,j}\left(
	u\right) \equiv
	\sum_{j=0}^{n+1}\mathbf{p}^{\wp}_{1,j}b_{n+1,j}\left(
	u\right),\,\forall u \in \left[\alpha, \beta\right],
	\]
	where $\mathbf{p}_{1,0}^{\wp}\equiv \mathbf{p}_0^{\wp}$, $\mathbf{p}_{1,n+1}^{\wp}\equiv \mathbf{p}_n^{\wp}$, while $\mathbf{p}_{1,j}^{\wp} = \left(1-\xi_j\right) \mathbf{p}_{j-1}^{\wp} + \xi_j \mathbf{p}_j^{\wp}$ for some real numbers $\xi_j \in \left(0,1\right)$, $j=0,1,\ldots,n$. Iterating this corner cutting based representation of $\mathbf{c}^{\wp}$ in the normalized B-bases of the nested EC spaces $\mathbb{S}_n^{\alpha, \beta}\subset \mathbb{S}_{n+1}^{\alpha,\beta}\subset \ldots \subset \mathbb{S}_{n+z}^{\alpha,\beta}\subset \ldots$, one obtains a sequence of control polygons which converges to a Lipschitz-continuous limit curve \cite{deBoor1990} that, in general, does not necessarily coincide with $\mathbf{c}^{\wp}$. As it is pointed out in a unified manner in \cite[Remark 2.3, p. 76]{Roth2015},  in case of vector spaces of finite order trigonometric/hyperbolic polynomials, the sequence of order elevated control polygons always converges to the curve generated by the first term of the sequence. In case of traditional polynomials of finite degree, one can use the well-known degree elevation techniques of (rational) B\'ezier curves. However, in general, the initial ordinary basis $\mathcal{F}^{\alpha, \beta}_{n}$ can iteratively be appended by new linearly independent functions in infinitely many ways and not every choice of functions leads to a sequence of order elevated control polygons that fulfills the desired convergence property, e.g.\ in EC M\"untz spaces a recent characterization of the required convergence of the dimension/order elevation process can be found in \cite{AitHadou2014}. In order to illustrate the last step of Algorithm \ref{alg:ordinary_rational_curves}, Fig.\ \ref{fig:last_step_of_the_algorithm} shows two different control point configurations for the exact representation of the rational trigonometric curve
	\begin{equation}
	\label{eq:second_order_rational_trigonometric_curve}
	\mathbf{c}\left(u\right)
	=
	\frac{1}{\frac{5}{8}-\frac{1}{2}\sin\left(2u\right)}
	\left[
	\begin{array}{c}
	\frac{1}{\sqrt{2}}\left(\sin\left(u\right)+\cos\left(u\right)\right)\\
	\frac{3}{2}\cos\left(2u\right)
	\end{array}
	\right],
	~
	u \in \left[0, \frac{\pi}{2}\right],
	\end{equation}
	by means of second and fourth order normalized trigonometric basis functions of type (\ref{eq:trigonometric_B-basis-n}).
\end{remark}

\begin{figure}[!h]
	\centering
	\includegraphics[scale = 1]{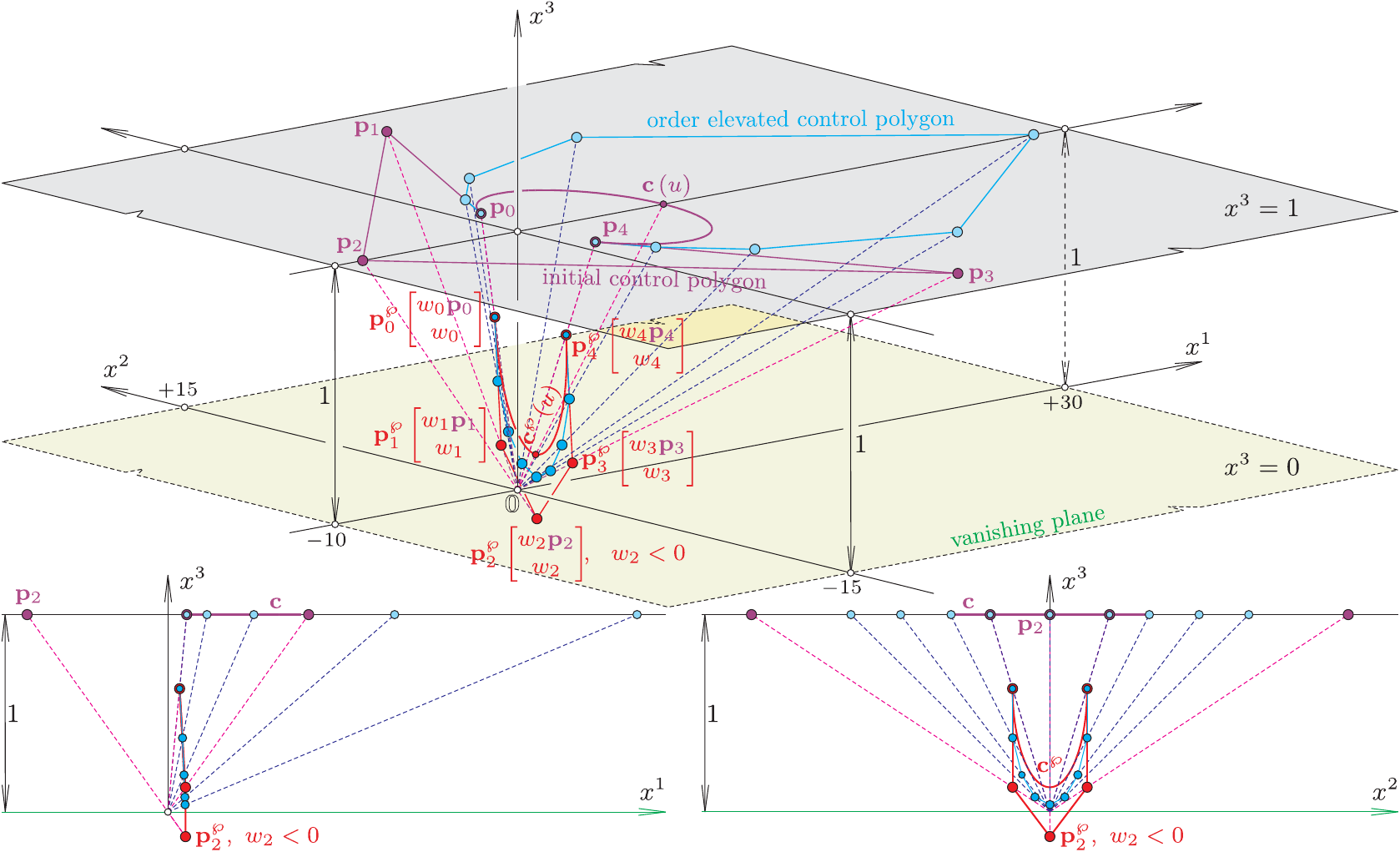}
	\caption{
			Although the integral pre-image $\mathbf{c}^{\wp}$ of (\ref{eq:second_order_rational_trigonometric_curve}) is above the vanishing plane $x^3=0$, minimal (i.e., second) order trigonometric normalized B-basis functions of type (\ref{eq:trigonometric_B-basis-n}) do not ensure the non-negativity of all weights required by the rational representation (\ref{eq:rational_curve_representation}), since the third coordinate $w_2$ of the control point $\mathbf{p}_2^{\wp}$ is strictly negative. Such pathological cases usually violate the convex hull property of (\ref{eq:rational_curve_representation}). Increasing the order of the applied normalized B-basis functions from $2$ to $4$, all control points needed for the exact representation of the pre-image $\mathbf{c}^{\wp}$ will be above the vanishing plane, i.e., in this case there exists a minimal order starting from which the applied dimension elevation technique ensures the non-negativity of all weights.}
	\label{fig:last_step_of_the_algorithm} 
\end{figure}

\begin{remark}
	[Exact description of ordinary rational surfaces]\label{rem:rational_surfaces}%
	The steps of Algorithm \ref{alg:ordinary_rational_curves} can easily be
	extended to the control point based exact description of those ordinary
	rational surfaces%
	\begin{equation}
	\mathbf{s}\left(  \mathbf{u}\right)  =\frac{1}{s^{4}\left(  \mathbf{u}\right)
	}\left[
	\begin{array}
	[c]{ccc}%
	s^{1}\left(  \mathbf{u}\right)   & s^{2}\left(  \mathbf{u}\right)   &
	s^{3}\left(  \mathbf{u}\right)
	\end{array}
	\right]  ^{T}\in%
	\mathbb{R}
	^{3},~\mathbf{u=}\left[  u_{r}\right]  _{r=1}^{2}\in\left[  \alpha_{1}%
	,\beta_{1}\right]  \times\left[  \alpha_{2},\beta_{2}\right]
	,\label{eq:ordinary_rational_surface}%
	\end{equation}
	in case of which%
	\[
	s^{\ell}\left(  \mathbf{u}\right)  =\sum_{\zeta=1}^{\sigma_{\ell}}\prod_{r=1}%
	^{2}\left(  \sum_{i_{r}=0}^{n_{r}}\lambda_{i_{r}}^{\ell,\zeta}\varphi
	_{n_{r},i_{r}}\left(  u_{r}\right)  \right)  ,~\sigma_{\ell}\geq 1,~\ell=1,2,3,4
	\]
	and $s^{4}\left(  \mathbf{u}\right)  >0,~\forall\mathbf{u}\in\left[
	\alpha_{1},\beta_{1}\right]  \times\left[  \alpha_{2},\beta_{2}\right]  $.
\end{remark}

\begin{theorem}
	[Computational complexity]\label{thm:computational_complexity}Provided that
	endpoint derivatives $\{  \varphi_{n,i}^{\left(  j\right)  }\left(
	\alpha\right)  ,\varphi_{n,i}^{\left(  j\right)  }\left(  \beta\right)
	,b_{n,i}^{\left(  j\right)  }\left(  \alpha\right)  ,\linebreak{}b_{n,i}^{\left(
		j\right)  }\left(  \beta\right)  \}  _{i=1,~j=0}^{n,~\left\lfloor
		\frac{n}{2}\right\rfloor }$ are calculated and stored in
	advance in permanent lookup tables  for a fixed value of $n\geq 1$, the total number of floating point operations (or shortly flops) that
	have to be performed for the evaluation of formulas (\ref{eq:first_half}) and (\ref{eq:last_half}) is%
	\begin{equation}
	\label{eq:total_computational_cost}
	\def\arraystretch{1.5}
	\kappa\left(n\right)=\left\{
	\begin{array}
	[c]{lc}%
	2^{\left\lfloor \frac{n}{2}\right\rfloor +1}\left(  \left\lfloor \frac{n}%
	{2}\right\rfloor -3\right)  +2\left\lfloor \frac{n}{2}\right\rfloor
	+6+2n\left\lfloor \frac{n}{2}\right\rfloor \left(  \left\lfloor \frac{n}%
	{2}\right\rfloor +1\right)  , & n\left(  \operatorname{mod}2\right)  =1,\\
	2^{ \frac{n}{2} -1}\left(\frac{3n}%
	{2}-10\right)  +n
	+5+ \frac{n^3}{2}, & n\left(  \operatorname{mod}%
	2\right)  =0.
	\end{array}
	\right.
	\def\arraystretch{1.0}
	\end{equation}
\end{theorem}

Using the normalized B-basis $\mathcal{B}_{n}^{\alpha,\beta}$, the control point based exact description of the ordinary integral curve (\ref{eq:ordinary_integral_curve}) can also be imagined either as a curve interpolation problem or as the least square approximation of the considered curve. Both of these alternative numeric methods can be reduced to the solutions of $\delta$ systems of linear equations that determine the unknown coordinates of the control points $\left[\mathbf{p}_j\right]_{j=0}^{n}$ appearing in the EC B-curve representation (\ref{eq:convex_combination}). Such methods also depend heavily either on the choice of the interpolation conditions or on the applied quadrature formulas used for the approximation of those integrals that appear in the equivalent quadratic form of the $L_2$ distance function used in case of least square approximations -- but let us neglect both the floating-point round-off errors and the computational cost (i.e., the number of flops) of the regular main matrices of the size $\left(n+1\right)\times\left(n+1\right)$ of these alternative methods, and let us compare the exponential computational cost (\ref{eq:total_computational_cost}) to the total work of a possible algorithm that efficiently solves $\delta$ systems of linear equations. The fastest currently known matrix inversion algorithms \cite{CoppersmithWinograd1990} and \cite{Vassilevska2012} are based on the fast matrix multiplication algorithm of Strassen \cite{Strassen1969} and have an asymptotic cost of order $\mathcal{O}\left(\left(n+1\right)^{2.376}\right)$ and $\mathcal{O}\left(\left(n+1\right)^{2.373}\right)$, respectively, instead of $\mathcal{O}\left(\left(n+1\right)^3\right)$ of traditional matrix inversion algorithms based e.g.\ on $LU$ decomposition. However, if one estimates how large $n+1$ has to be before the difference
between exponents $2.373|2.376$ and $3$ is substantial enough to outweigh the bookkeeping overhead, arising
from the complicated nature of the recursive Strassen algorithm, one finds that $LU$ decomposition \textit{is in no
	immediate danger of becoming obsolete} \cite[p. 108]{Press2007}. The fast matrix multiplication/inversion algorithm of Strassen is typically used for $n+1 > 99$, thus it is not practical to implement it for modeling purposes,
since it is very unlikely that one would describe an arc of an ordinary integral curve with more than $99$ control points.

In case of a regular main matrix of the size $\left(n+1\right)\times\left(n+1\right)$, the number of 
flops performed by a numerical
curve interpolation or least square approximation method based on $LU$ decomposition is
\begin{equation}
\label{eq:LU_cost}
\kappa_{LU}\left(n,\delta\right)=
\frac{2}{3}\left(n+1\right)^3-\frac{1}{2}\left(n+1\right)^2-\frac{1}{6}\left(n+1\right)+\left(2\left(n+1\right)^2-\left(n+1\right)\right)\delta
\end{equation}
which covers the cost of computing of all multipliers, of all row operations and of all forward and backward
substitutions as well. Naturally, as $n$ tends to infinity, the growth rate of the exponential cost function (\ref{eq:total_computational_cost}) is substantially bigger than that of the cubic one (\ref{eq:LU_cost}), however if $n \in \left\{1,2,\ldots,15\right\}$ then (\ref{eq:total_computational_cost}) is less than (\ref{eq:LU_cost}) as it is illustrated in Fig.\ \ref{fig:computational_cost}, i.e.,  compared with other cubic time numerical algorithms, the proposed general basis transformation can more efficiently be implemented up to $16$-dimensional EC spaces despite the seemingly complicated nature of formulas (\ref{eq:first_half}) and (\ref{eq:last_half}). Considering that, in practice, curves and surfaces are mostly composed of continuously joined lower order arcs and patches, even a sequential but clever implementation of Theorem \ref{thm:basis_transformation} can be useful in case of real-life applications. Nevertheless, if $n>15$ then the presented results are mainly of theoretical interest.

\begin{figure}[!h]
	\centering
	\includegraphics[scale = 1]{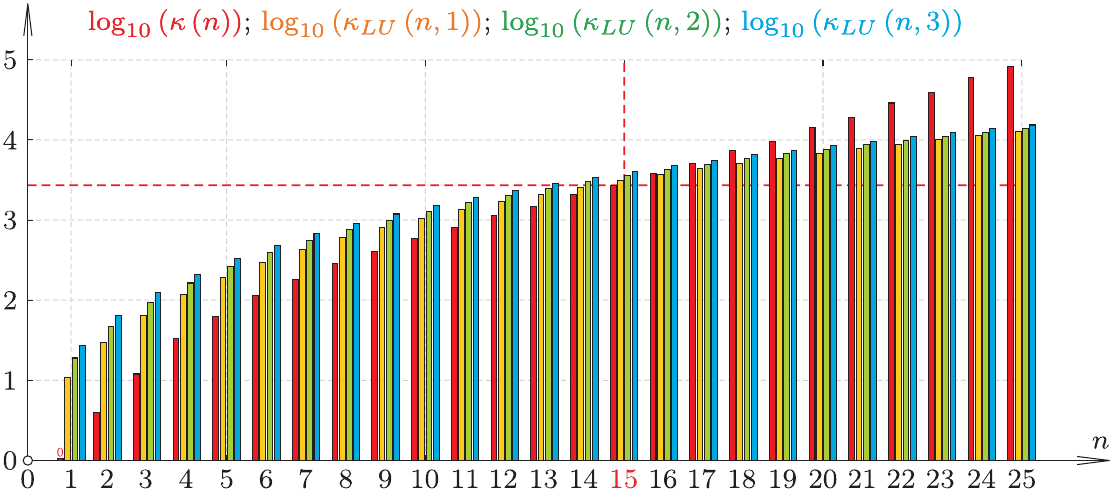}
	\caption{Logarithmic scales of computational costs (\ref{eq:total_computational_cost}) and (\ref{eq:LU_cost}) for different values of $n\geq 1$ and $\delta \geq 1$.}
	\label{fig:computational_cost}
\end{figure}

\section{Examples\label{sec:examples}}

This section applies closed formulas
(\ref{eq:first_half}) and (\ref{eq:last_half}) in case of different vector
spaces of functions that can be spanned by ordinary EC bases
of the type (\ref{eq:ordinary_basis}). Our intention is only to emphasize the
global applicability of the general basis transformation described in
Theorem \ref{thm:basis_transformation} with examples that can be compared to
possible already existing results in the literature. Formulas
(\ref{eq:first_half}) and (\ref{eq:last_half}) depend on the higher order
endpoint derivatives of the ordinary and normalized B-basis of the underlying
vector space. The following subsections specify these values in case of vector
spaces of functions that may be important in many areas of applied or computational mathematics. We consider several reflection invariant EC spaces, since in practice usually one uses unbiased or symmetric systems of basis functions that also provide some computational advantages. Naturally, general formulas (\ref{eq:first_half})--(\ref{eq:last_half}) are  valid in not necessarily reflection invariant EC spaces as well.

\subsection{Trigonometric polynomials\label{subsec:trigonometric_polynomials}}

Let $\alpha=0$ and $\beta\in\left(  0,\pi\right)  $ be fixed parameters and
consider the ordinary basis%
\begin{equation}
\mathcal{F}_{2n}^{0,\beta}=\left\{\varphi_{2n,0}\left(  u\right)
\equiv 1,~\left\{  \varphi_{2n,2i-1}\left(
u\right)  =\sin\left(  iu\right)  ,~\varphi_{2n,2i}\left(  u\right)
=\cos\left(  iu\right)  \right\}  _{i=1}^{n}:u\in\left[  0,\beta\right] \right\}
\label{eq:trigonometric_polynomials_n}%
\end{equation}
of trigonometric polynomials of order at most $n$ (degree $2n$). Using the
results of \cite{Sanchez1998}, the normalized B-basis of the vector space
$\mathbb{S}_{2n}^{0,\beta}=\left\langle \mathcal{F}_{2n}^{0,\beta
}\right\rangle $ can linearly be reparametrized into the form%
\begin{equation}
\mathcal{B}_{2n}^{0,\beta}=\left\{  b_{2n,i}\left(  u\right)  =c_{2n,i}%
^{\beta}\sin^{2n-i}\left(  \frac{\beta-u}{2}\right)  \sin^{i}\left(  \frac
{u}{2}\right)  :u\in\left[  0,\beta\right]  \right\}  _{i=0}^{2n},
\label{eq:trigonometric_B-basis-n}%
\end{equation}
where%
\begin{equation}
c_{2n,i}^{\beta}=c_{2n,2n-i}^{\beta}=\frac{1}{\sin^{2n}\left(  \frac{\beta}%
	{2}\right)  }\sum_{r=0}^{\left\lfloor \frac{i}{2}\right\rfloor }\binom{n}%
{i-r}\binom{i-r}{r}\left(  2\cos\left(  \frac{\beta}{2}\right)  \right)
^{i-2r},~i=0,1,\ldots,n \label{eq:trigonometric_normalizing_constants_n}%
\end{equation}
are symmetric normalizing coefficients. It is obvious that%
\begin{align*}
\varphi_{2n,2i-1}^{\left(  j\right)  }\left(  0\right)   &  =i^{j}\varphi
_{2n,2i-1}\left(  \frac{j\pi}{2}\right)  ,~\varphi_{2n,2i-1}^{\left(  j\right)
}\left(  \beta\right)  =i^{j}\varphi_{2n,2i-1}\left(  \beta+\frac{j\pi}%
{2}\right)  ,~j=0,1,\ldots,n,~i=1,2,\ldots,n,\\
\varphi_{2n,2i}^{\left(  j\right)  }\left(  0\right)   &  =i^{j}\varphi
_{2n,2i}\left(  \frac{j\pi}{2}\right)  ,~\varphi_{2n,2i}^{\left(  j\right)
}\left(  \beta\right)  =i^{j}\varphi_{2n,2i}\left(  \beta+\frac{j\pi}%
{2}\right)  ,~j=0,1,\ldots,n,~i=1,2,\ldots,n,\\
b_{2n,0}\left(  0\right)   &  =b_{2n,2n}\left(  \beta\right)  =1,~b_{2n,i}%
\left(  0\right)  =b_{2n,2n-i}\left(  \beta\right)  =0,~i=1,2,\ldots,2n,
\end{align*}
while the higher order derivatives $\{  b_{2n,i}^{\left(  j\right)
}\left(  0\right)  ,~b_{2n,i}^{\left(  j\right)  }\left(  \beta\right)
\}  _{i=0,~j=1}^{2n,~n}$ are specified by the next theorem.

\begin{theorem}
	[Trigonometric endpoint derivatives]%
	\label{thm:trigonometric_endpoint_derivatives}For arbitrary derivative order
	$j=1,2,\ldots,n$ we have that%
	\begin{align}
	\frac{b_{2n,2r+1}^{\left(  j\right)  }\left(  0\right)  }{c_{2n,2r+1}^{\beta}%
	}=  &  ~\frac{1}{2^{2n-1}}\sum_{k=0}^{n-r-1}\sum_{\ell=0}^{r}\left(
	-1\right)  ^{n+1-k-\ell}\binom{2\left(  n-r-1\right)  +1}{k}\binom{2r+1}{\ell
	}\cdot\label{eq:trigonometric_B-basis_n_odd_0}\\
	&  ~\cdot\left(  \left(  n-k-\ell\right)  ^{j}-\left(  n-k-2r+\ell-1\right)
	^{j}\right)  \cos\left(  \left(  2\left(  n-r-k\right)  -1\right)  \frac
	{\beta}{2}-\frac{j\pi}{2}\right)  ,\nonumber
	\end{align}
	for all $r=0,1,\ldots,n-1$ and%
	\begin{align}
	\frac{b_{2n,2r}^{\left(  j\right)  }\left(  0\right)  }{c_{2n,2r}^{\beta}}=
	&  ~\frac{\binom{2n-2r}{n-r}}{2^{2n-1}}\sum_{\ell=0}^{r-1}\left(  -1\right)
	^{r-\ell}\binom{2r}{\ell}\left(  r-\ell\right)  ^{j}\cos\left(  \frac{j\pi}%
	{2}\right) \label{eq:trigonometric_B-basis_n_even_0}\\
	&  ~+\frac{\binom{2r}{r}}{2^{2n-1}}\sum_{k=0}^{n-r-1}\left(  -1\right)
	^{n-r-k}\binom{2\left(  n-r\right)  }{k}\left(  n-r-k\right)  ^{j}\cos\left(
	\left(  n-r-k\right)  \beta-\frac{j\pi}{2}\right) \nonumber\\
	&  ~+\frac{1}{2^{2n-1}}\sum_{k=0}^{n-r-1}\sum_{\ell=0}^{r-1}\left(  -1\right)
	^{n-k-\ell}\binom{2\left(  n-r\right)  }{k}\binom{2r}{\ell}\nonumber\\
	&  ~\cdot\left(  \left(  n-k-\ell\right)  ^{j}+\left(  n-k-2r+\ell\right)
	^{j}\right)  \cos\left(  \left(  n-r-k\right)  \beta-\frac{j\pi}{2}\right)
	,\nonumber
	\end{align}
	for all $r=0,1,\ldots,n$. At the same time $b_{2n,i}^{\left(  j\right)
	}\left(  \beta\right)  =\left(  -1\right)  ^{j}b_{2n,2n-i}^{\left(  j\right)
}\left(  0\right)  ,~i=0,1,\ldots,2n,~j=0,1,\ldots,n.$
\end{theorem}

\begin{example}
	[Second order trigonometric polynomials]\label{exmp:trigonometric} Consider
	the ordinary basis
	\begin{align*}
	\mathcal{F}_{4}^{0,\beta}=  &  ~\left\{  \varphi_{4,0}\left(  u\right)
	=1,~\varphi_{4,1}\left(  u\right)  =\sin\left(  u\right)  ,~\varphi
	_{4,2}\left(  u\right)  =\cos\left(  u\right)  \right.  ,\\
	&  ~\left.  \varphi_{4,3}\left(  u\right)  =\sin\left(  2u\right)
	,~\varphi_{4,4}\left(  u\right)  =\cos\left(  2u\right)  :u\in\left[
	0,\beta\right]  \right\}  ,~\beta\in\left(  0,\pi\right)
	\end{align*}
	of the vector space of trigonometric polynomials of order at most two (or
	degree $4$) and its normalized B-basis%
	\begin{equation}
	\mathcal{B}_{4}^{0,\beta}=\left\{  b_{4,i}\left(  u\right)  =c_{4,i}^{\beta
	}\sin^{4-i}\left(  \frac{\beta-u}{2}\right)  \sin^{i}\left(  \frac{u}%
	{2}\right)  :u\in\left[  0,\beta\right]  \right\}  _{i=0}^{4},
	\label{eq:trigonometric_B-basis}%
	\end{equation}
	where 
	\begin{footnotesize}
		$
		c_{4,0}^{\beta}=c_{4,4}^{\beta}=\frac{1}{\sin^{4}\left(  \frac{\beta}%
			{2}\right)  },~c_{4,1}^{\beta}=c_{4,3}^{\beta}=\frac{4\cos\left(  \frac{\beta
			}{2}\right)  }{\sin^{4}\left(  \frac{\beta}{2}\right)  },~c_{4,2}^{\beta
		}=\frac{2+4\cos^{2}\left(  \frac{\beta}{2}\right)  }{\sin^{4}\left(
		\frac{\beta}{2}\right)  },
	$
\end{footnotesize}
\begin{align*}
&
\begin{footnotesize}%
\begin{array}
[c]{lcrlcllcllcllcr}%
b_{4,0}\left(  0\right)  & = & 1, & b_{4,1}\left(  0\right)  & = & 0, &
b_{4,2}\left(  0\right)  & = & 0, & b_{4,3}\left(  0\right)  & = & 0, &
b_{4,4}\left(  0\right)  & = & 0,\\
b_{4,0}\left(  \beta\right)  & = & 0, & b_{4,1}\left(  \beta\right)  & = &
0, & b_{4,2}\left(  \beta\right)  & = & 0, & b_{4,3}\left(  \beta\right)  &
= & 0, & b_{4,4}\left(  \beta\right)  & = & 1,
\\
\\
&& & b_{4,1}^{\left(  1\right)  }\left(  0\right)  &= & \frac{c_{4,1}^{\beta}}%
{2}\sin^{3}\left(  \frac{\beta}{2}\right)  , & b_{4,2}^{\left(  1\right)
}\left(  0\right)  &= & 0, & b_{4,3}^{\left(  1\right)  }\left(  0\right)  &= &
0,\\
&&& b_{4,1}^{\left(  1\right)  }\left(  \beta\right)  &= & 0, & b_{4,2}^{\left(
	1\right)  }\left(  \beta\right)  &= & 0, & b_{4,3}^{\left(  1\right)  }\left(
\beta\right)  &= & -\frac{c_{4,3}^{\beta}}{2}\sin^{3}\left(  \frac{\beta}%
{2}\right),
\\
\\
&&& b_{4,1}^{\left(  2\right)  }\left(  0\right)  &= & -\frac{3c_{4,1}^{\beta}}%
{4}\sin\left(  \frac{\beta}{2}\right)  \sin\left(  \beta\right)  , &
b_{4,2}^{\left(  2\right)  }\left(  0\right)  &= & \frac{c_{4,2}^{\beta}}%
{2}\sin^{2}\left(  \frac{\beta}{2}\right)  , & b_{4,3}^{\left(  2\right)
}\left(  0\right)  &= & 0,\\
&&& b_{4,1}^{\left(  2\right)  }\left(  \beta\right)  &= & 0, & b_{4,2}^{\left(
	2\right)  }\left(  \beta\right)  &= & \frac{c_{4,2}^{\beta}}{2}\sin^{2}\left(
\frac{\beta}{2}\right)  , & b_{4,3}^{\left(  2\right)  }\left(  \beta\right)
&= & -\frac{3c_{4,3}^{\beta}}{4}\sin\left(  \frac{\beta}{2}\right)  \sin\left(
\beta\right)
\end{array}%
\end{footnotesize}%
\end{align*}
and%
\[%
\begin{footnotesize}%
\begin{array}
[c]{lllllllllllllll}%
\varphi_{4,0}\left(  0\right)  & = & 1, & \varphi_{4,1}\left(  0\right)  & = &
0, & \varphi_{4,2}\left(  0\right)  & = & 1, & \varphi_{4,3}\left(  0\right)
& = & 0, & \varphi_{4,4}\left(  0\right)  & = & 1,\\
\varphi_{4,0}\left(  \beta\right)  & = & 1, & \varphi_{4,1}\left(
\beta\right)  & = & \sin\left(  \beta\right)  , & \varphi_{4,2}\left(
\beta\right)  & = & \cos\left(  \beta\right)  , & \varphi_{4,3}\left(
\beta\right)  & = & \sin\left(  2\beta\right)  , & \varphi_{4,4}\left(
\beta\right)  & = & \cos\left(  2\beta\right)  ,\\
&  &  &  &  &  &  &  &  &  &  &  &  &  & \\
\varphi_{4,0}^{\left(  1\right)  }\left(  0\right)  & = & 0, & \varphi
_{4,1}^{\left(  1\right)  }\left(  0\right)  & = & 1, & \varphi_{4,2}^{\left(
	1\right)  }\left(  0\right)  & = & 0, & \varphi_{4,3}^{\left(  1\right)
}\left(  0\right)  & = & 2, & \varphi_{4,4}^{\left(  1\right)  }\left(
0\right)  & = & 0,\\
\varphi_{4,0}^{\left(  1\right)  }\left(  \beta\right)  & = & 0, &
\varphi_{4,1}^{\left(  1\right)  }\left(  \beta\right)  & = & \cos\left(
\beta\right)  , & \varphi_{4,2}^{\left(  1\right)  }\left(  \beta\right)  &
= & -\sin\left(  \beta\right)  , & \varphi_{4,3}^{\left(  1\right)  }\left(
\beta\right)  & = & 2\cos\left(  2\beta\right)  , & \varphi_{4,4}^{\left(
	1\right)  }\left(  \beta\right)  & = & -2\sin\left(  2\beta\right)  ,\\
&  &  &  &  &  &  &  &  &  &  &  &  &  & \\
\varphi_{4,0}^{\left(  2\right)  }\left(  0\right)  & = & 0, & \varphi
_{4,1}^{\left(  2\right)  }\left(  0\right)  & = & 0, & \varphi_{4,2}^{\left(
	2\right)  }\left(  0\right)  & = & -1, & \varphi_{4,3}^{\left(  2\right)
}\left(  0\right)  & = & 0, & \varphi_{4,4}^{\left(  2\right)  }\left(
0\right)  & = & -4,\\
\varphi_{4,0}^{\left(  2\right)  }\left(  \beta\right)  & = & 0, &
\varphi_{4,1}^{\left(  2\right)  }\left(  \beta\right)  & = & -\sin\left(
\beta\right)  , & \varphi_{4,2}^{\left(  2\right)  }\left(  \beta\right)  &
= & -\cos\left(  \beta\right)  , & \varphi_{4,3}^{\left(  2\right)  }\left(
\beta\right)  & = & -4\sin\left(  2\beta\right)  , & \varphi_{4,4}^{\left(
	2\right)  }\left(  \beta\right)  & = & -4\cos\left(  2\beta\right)  .
\end{array}%
\end{footnotesize}%
\]
Substituting for $n=4$ and $\alpha=0$ the derivatives above into identities (\ref{eq:first_half}) and (\ref{eq:last_half}), one obtains
the transformation matrix%
\begin{equation}
\left[  t_{i,j}^{4}\right]  _{i=0,~j=0}^{4,~4}=~%
\begin{footnotesize}%
\def\arraystretch{1.75}%
\left[
\begin{array}
[c]{ccccc}%
1 & 1 & 1 & 1 & 1\\
0 & \frac{1}{2}\tan\left(  \frac{\beta}{2}\right)  & \frac{3\sin\left(
	\beta\right)  }{2+4\cos^{2}\left(  \frac{\beta}{2}\right)  } & \sin\left(
\beta\right)  -\frac{1}{2}\cos\left(  \beta\right)  \tan\left(  \frac{\beta
}{2}\right)  & \sin\left(  \beta\right) \\
1 & 1 & \frac{3\left(  1+\cos\left(  \beta\right)  \right)  }{2+4\cos
	^{2}\left(  \frac{\beta}{2}\right)  } & \cos\left(  \beta\right)  \fcolorbox{white}{Lavender!15}{\DPurple{$\boldsymbol{+}$}}\frac{1}%
{2}\sin\left(  \beta\right)  \tan\left(  \frac{\beta}{2}\right)  & \cos\left(
\beta\right) \\
0 & \tan\left(  \frac{\beta}{2}\right)  & \frac{6\sin\left(  \beta\right)
}{2+4\cos^{2}\left(  \frac{\beta}{2}\right)  } & \sin\left(  2\beta\right)
-\cos\left(  2\beta\right)  \tan\left(  \frac{\beta}{2}\right)  & \sin\left(
2\beta\right) \\
1 & 1 & \frac{6\cos\left(  \beta\right)  }{2+4\cos^{2}\left(  \frac{\beta}%
	{2}\right)  } & \cos\left(  2\beta\right)  \fcolorbox{white}{Lavender!15}{\DPurple{$\boldsymbol{+}$}}\sin\left(  2\beta\right)
\tan\left(  \frac{\beta}{2}\right)  & \cos\left(  2\beta\right)
\end{array}
\right],  %
\def\arraystretch{1.0}%
\end{footnotesize}
\label{eq:trigonometric_basis_transformation}%
\end{equation}
based on which Fig.\ \ref{fig:integral_and_rational_trigonometric_surfaces} shows control net
configurations for the exact description of patches of some integral and
rational trigonometric surfaces.
\end{example}

\begin{figure}
	[!h]
	\begin{center}
		\includegraphics[scale = 1]{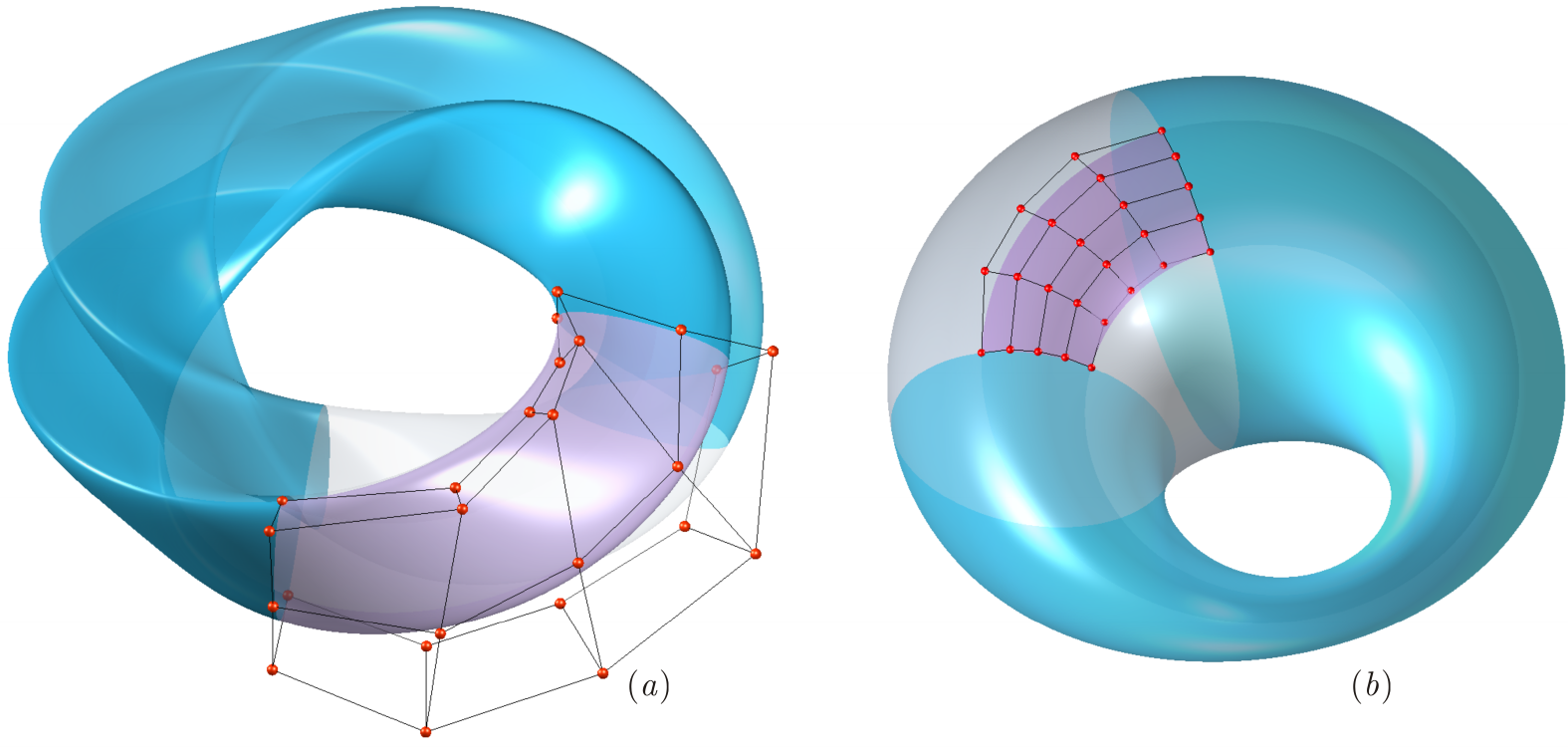}
		
		\caption{Control point configurations for the exact description of patches of
			(\textit{a}) a special integral variant of Alfred Gray's non-orientable Klein
			Bottle and of (\textit{b}) a rational ring Dupin cyclide. Besides the
			trigonometric basis transformation
			(\ref{eq:trigonometric_basis_transformation}), in cases (\textit{a}) and
			(\textit{b}) Theorem \ref{thm:integral_surfaces} and Remark
			\ref{rem:rational_surfaces} (i.e., the extension of Algorithm
			\ref{alg:ordinary_rational_curves}) were applied, respectively.}%
		
		\label{fig:integral_and_rational_trigonometric_surfaces}%
	\end{center}
\end{figure}

\vspace{-1.25cm}
\subsection{Hyperbolic polynomials\label{subsec:hyperbolic_polynomials}}

Now, let $\alpha = 0$ and $\beta>0$ be fixed parameters. Using hyperbolic sine and cosine
functions in expressions (\ref{eq:trigonometric_polynomials_n}%
)--(\ref{eq:trigonometric_normalizing_constants_n}) instead of the
trigonometric ones, we obtain the vector space of hyperbolic polynomials of
order at most $n$ (or degree $2n$) the unique normalized B-basis of which was
introduced in \cite{ShenWang2005}. In this case%
\begin{align*}
\varphi_{2n,0}^{\left(  j\right)  }\left(  u\right)   &  =0,~j=1,2,\ldots,n,\\
b_{2n,0}\left(  0\right)   &  =b_{2n,2n}\left(  \beta\right)  =1,~b_{2n,i}%
\left(  0\right)  =b_{2n,2n-i}\left(  \beta\right)  =0,~i=1,2,\ldots,n
\end{align*}
and%
\begin{align*}
\varphi_{2n,2i-1}^{\left(  j\right)  }\left(  0\right)   &  =\left\{
\begin{array}
[c]{ll}%
0, & j\left(  \operatorname{mod}2\right)  =0,\\
i^{j}, & j\left(  \operatorname{mod}2\right)  =1,
\end{array}
\right.  ~~\varphi_{2n,2i-1}^{\left(  j\right)  }\left(  \beta\right)
=\left\{
\begin{array}
[c]{ll}%
i^{j}\sinh\left(  i\beta\right)  , & j\left(  \operatorname{mod}2\right)
=0,\\
i^{j}\cosh\left(  i\beta\right)  , & j\left(  \operatorname{mod}2\right)  =1,
\end{array}
\right. \\
\varphi_{2n,2i}^{\left(  j\right)  }\left(  0\right)   &  =\left\{
\begin{array}
[c]{ll}%
i^{j}, & j\left(  \operatorname{mod}2\right)  =0,\\
0, & j\left(  \operatorname{mod}2\right)  =1,
\end{array}
\right.  ,~~\varphi_{2n,2i}^{\left(  j\right)  }\left(  \beta\right)
=\left\{
\begin{array}
[c]{ll}%
i^{j}\cosh\left(  i\beta\right)  , & ~j\left(  \operatorname{mod}2\right)
=0,\\
i^{j}\sinh\left(  i\beta\right)  , & ~j\left(  \operatorname{mod}2\right)  =1
\end{array}
\right.
\end{align*}
for all $i,j\in\left\{  1,2,\ldots,n\right\}  $, while the higher order
derivatives $\{  b_{2n,i}^{\left(  j\right)  }\left(  0\right)
,~b_{2n,i}^{\left(  j\right)  }\left(  \beta\right)  \}  _{i=0,~j=1}%
^{2n,~n}$ are specified by the next theorem.

\begin{theorem}
	[Hyperbolic endpoint derivatives]\label{thm:hyperbolic_endpoint_derivatives}%
	For arbitrary derivative order $j=1,2,\ldots,n$, one has that%
	\[%
	\begin{footnotesize}%
	\frac{b_{2n,2r+1}^{\left(  j\right)  }\left(  0\right)  }{c_{2n,2r+1}^{\beta}%
	}=~\left\{
	\begin{array}
	[c]{ll}%
	\frac{1}{2^{2n-1}}%
	{\displaystyle\sum\limits_{k=0}^{n-r-1}}
	{\displaystyle\sum\limits_{\ell=0}^{r}}
	\left(  -1\right)  ^{n+\left(  \left(  n-r-1\right)  \operatorname{mod}%
		2\right)  +\left(  r\operatorname{mod}2\right)  -k-\ell-1}\binom{2\left(
		n-r-1\right)  +1}{k}\binom{2r+1}{\ell}\cdot & \\
	\cdot\left(  \left(  n-k-2r+\ell-1\right)  ^{j}-\left(  n-k-\ell\right)
	^{j}\right)  \cosh\left(  \left(  2\left(  n-k-r\right)  -1\right)
	\frac{\beta}{2}\right)  , & j\left(  \operatorname{mod}2\right)  =0,\\
	& \\
	\frac{1}{2^{2n-1}}%
	{\displaystyle\sum\limits_{k=0}^{n-r-1}}
	{\displaystyle\sum\limits_{\ell=0}^{r}}
	\left(  -1\right)  ^{n+\left(  \left(  n-r-1\right)  \operatorname{mod}%
		2\right)  +\left(  r\operatorname{mod}2\right)  -k-\ell}\binom{2\left(
		n-r-1\right)  +1}{k}\binom{2r+1}{\ell}\cdot & \\
	\cdot\left(  \left(  n-k-2r+\ell-1\right)  ^{j}-\left(  n-k-\ell\right)
	^{j}\right)  \sinh\left(  \left(  2\left(  n-k-r\right)  -1\right)
	\frac{\beta}{2}\right)  , & j\left(  \operatorname{mod}2\right)  =1,
	\end{array}
	\right.
	\end{footnotesize}%
	\]
for all $r=0,1,\ldots,n-1$, while%
\begin{align*}
&  ~%
\begin{footnotesize}%
\frac{b_{2n,2r}^{\left(  j\right)  }\left(  0\right)  }{c_{2n,2r}^{\beta}}=%
\end{footnotesize}%
\\
=  &  ~%
\begin{footnotesize}%
\left\{
\begin{array}
[c]{ll}%
\frac{\binom{2\left(  n-r\right)  }{n-r}\binom{2r}{r}}{2^{2n}}+\frac
{\binom{2\left(  n-r\right)  }{n-r}}{2^{2n-1}}%
{\displaystyle\sum\limits_{\ell=0}^{r+\left(  r\operatorname{mod}2\right)
		-1}}
\left(  -1\right)  ^{r+\left(  r\operatorname{mod}2\right)  -\ell}\binom
{2r}{\ell}2^{j}\left(  r-\ell\right)  ^{j} & \\
+\frac{\binom{2r}{r}}{2^{2n-1}}%
{\displaystyle\sum\limits_{k=0}^{\left(  n-r\right)  +\left(  \left(
		n-r\right)  \operatorname{mod}2\right)  -1}}
\left(  -1\right)  ^{\left(  n-r\right)  +\left(  \left(  n-r\right)
	\operatorname{mod}2\right)  -k}\binom{2\left(  n-r\right)  }{k}\left(
n-r-k\right)  ^{j}\cosh\left(  \left(  n-r-k\right)  \beta\right)  & \\
+\frac{1}{2^{2n-1}}%
{\displaystyle\sum\limits_{k=0}^{\left(  n-r\right)  +\left(  \left(
		n-r\right)  \operatorname{mod}2\right)  -1}}
~~%
{\displaystyle\sum\limits_{\ell=0}^{r+\left(  r\operatorname{mod}2\right)
		-1}}
\left(  -1\right)  ^{n+\left(  \left(  n-r\right)  \operatorname{mod}2\right)
	+\left(  r\operatorname{mod}2\right)  -k-\ell}\binom{2\left(  n-r\right)  }%
{k}\binom{2r}{\ell}\cdot & \\
\cdot\left(  \left(  n-k-3r+2\ell\right)  ^{j}+\left(  n+r-k-2\ell\right)
^{j}\right)  \cosh\left(  \left(  n-r-k\right)  \beta\right)  , & j\left(
\operatorname{mod}2\right)  =0,\\
& \\
& \\
\frac{\binom{2\left(  n-r\right)  }{n-r}\binom{2r}{r}}{2^{2n}}+\frac
{\binom{2\left(  n-r\right)  }{n-r}}{2^{2n-1}}%
{\displaystyle\sum\limits_{\ell=0}^{r+\left(  r\operatorname{mod}2\right)
		-1}}
\left(  -1\right)  ^{r+\left(  r\operatorname{mod}2\right)  -\ell}\binom
{2r}{\ell}2^{j}\left(  r-\ell\right)  ^{j}\sinh\left(  2\left(  r-\ell\right)
\beta\right)  & \\
+\frac{1}{2^{2n-1}}%
{\displaystyle\sum\limits_{k=0}^{\left(  n-r\right)  +\left(  \left(
		n-r\right)  \operatorname{mod}2\right)  -1}}
~~%
{\displaystyle\sum\limits_{\ell=0}^{r+\left(  r\operatorname{mod}2\right)
		-1}}
\left(  -1\right)  ^{n+\left(  \left(  n-r\right)  \operatorname{mod}2\right)
	+\left(  r\operatorname{mod}2\right)  -k-\ell}\binom{2\left(  n-r\right)  }%
{k}\binom{2r}{\ell}\cdot & \\
\cdot\left(  \left(  n+r-k-2\ell\right)  ^{j}-\left(  n-k-3r+2\ell\right)
^{j}\right)  \sinh\left(  2\beta\left(  r-\ell\right)  \right)  , & j\left(
\operatorname{mod}2\right)  =1
\end{array}
\right.
\end{footnotesize}%
\end{align*}
for all $r=1,2,\ldots,n$ and $b_{2n,i}^{\left(  j\right)  }\left(
\beta\right)  =\left(  -1\right)  ^{j}b_{2n,2n-i}^{\left(  j\right)  }\left(
0\right)  ,~i=0,1,\ldots,2n,~j=0,1,\ldots,n.$
\end{theorem}

\begin{example}
	[Second order hyperbolic polynomials]\label{exmp:hyperbolic}Using hyperbolic
	sine, cosine and tangent functions instead of the trigonometric ones that
	appear in Example \ref{exmp:trigonometric} and applying the second order
	hyperbolic normalized B-basis \cite{ShenWang2005} with the shape parameter
	$\beta>0$, one can easily construct the hyperbolic counterpart%
	\begin{equation}
	\left[  t_{i,j}^{4}\right]  _{i=0,~j=0}^{4,~4}=%
	\begin{footnotesize}%
	\def\arraystretch{1.75}%
	\left[
	\begin{array}
	[c]{ccccc}%
	1 & 1 & 1 & 1 & 1\\
	0 & \frac{1}{2}\tanh\left(  \frac{\beta}{2}\right)  & \frac{3\sinh\left(
		\beta\right)  }{2+4\cosh^{2}\left(  \frac{\beta}{2}\right)  } & \sinh\left(
	\beta\right)  -\frac{1}{2}\cosh\left(  \beta\right)  \tanh\left(  \frac{\beta
	}{2}\right)  & \sinh\left(  \beta\right) \\
	1 & 1 & \frac{3\left(  1+\cosh\left(  \beta\right)  \right)  }{2+4\cosh
		^{2}\left(  \frac{\beta}{2}\right)  } & \cosh\left(  \beta\right) \fcolorbox{white}{Lavender!15}{\DPurple{$\boldsymbol{-}$}} \frac
	{1}{2}\sinh\left(  \beta\right)  \tanh\left(  \frac{\beta}{2}\right)  &
	\cosh\left(  \beta\right) \\
	0 & \tanh\left(  \frac{\beta}{2}\right)  & \frac{6\sinh\left(  \beta\right)
	}{2+4\cosh^{2}\left(  \frac{\beta}{2}\right)  } & \sinh\left(  2\beta\right)
	-\cosh\left(  2\beta\right)  \tanh\left(  \frac{\beta}{2}\right)  &
	\sinh\left(  2\beta\right) \\
	1 & 1 & \frac{6\cosh\left(  \beta\right)  }{2+4\cosh^{2}\left(  \frac{\beta
		}{2}\right)  } & \cosh\left(  2\beta\right)  \fcolorbox{white}{Lavender!15}{\DPurple{$\boldsymbol{-}$}}\sinh\left(  2\beta\right)
	\tanh\left(  \frac{\beta}{2}\right)  & \cosh\left(  2\beta\right)
	\end{array}
	\right]
	\def\arraystretch{1.0}%
	\end{footnotesize}
	\label{eq:hyperbolic_basis_transformation}%
	\end{equation}
	of the trigonometric basis transformation
	(\ref{eq:trigonometric_basis_transformation}), the structurally difference of which consists in the highlighted operators.
\end{example}


\subsection{A class of mixed spaces\label{subsec:mixed}}

In order to be as self-contained as possible,
we recall the construction process \cite{CarnicerMainarPena2004} of the normalized B-bases for a
family of mixed EC vector spaces of functions. 

Let $\alpha=0$ and $\beta > 0$ be fixed parameters and consider the homogeneous
linear differential equation%
\begin{equation}
\sum_{i=0}^{n+1}\gamma_{i}v^{\left(  i\right)  }\left(  u\right)
=0,~\gamma_{i}\in%
\mathbb{R}
,~u\in\left[  0,\beta\right]  \label{eq:differential_equation}%
\end{equation}
of order $n+1$ with constant coefficients and assume that its characteristic polynomial $p_{n+1}\left(r\right),~r\in\mathbb{C}$ is an either even or odd function %
such that $r=0$ is one of its (presumably higher order) zeros.
Hereafter we assume that the ordinary basis (\ref{eq:ordinary_basis})
corresponds to the system of those linearly independent functions that are
implied by all (higher order) zeros of $p_{n+1}$,
i.e., $\mathbb{S}_{n}^{0,\beta}$ is the $\left(  n+1\right)  $-dimensional
vector space of functions that is formed by all solutions of (\ref{eq:differential_equation}). Under these
conditions, $1\in\mathbb{S}_{n}^{0,\beta}$, moreover the space $\mathbb{S}%
_{n}^{0,\beta}$ is also invariant under reflections and consequently under
translations as well, i.e., for any function $f\in\mathbb{S}_{n}^{0,\beta}$ and
fixed scalar $\tau\in%
\mathbb{R}
$ the functions $g_{\tau}\left(  u\right)  :=f\left(  \tau-u\right)  $ and
$h_{\tau}\left(  u\right)  :=f\left(  u-\tau\right)  $ also belong to
$\mathbb{S}_{n}^{0,\beta}$.

Following \cite{CarnicerMainarPena2004}, one can both to determine (or at least to numerically approximate) the range of the shape parameter $\beta > 0$ for which $\mathbb{S}_{n}^{0,\beta}$ is an EC space and to construct its normalized B-basis as follows. Denote by%
\begin{equation}
W_{\left[  v_{n,0},v_{n,1},\ldots,v_{n,n}\right]  }\left(  u\right)  :=
\left[
v_{n,i}^{\left(j\right)}\left(u\right)
\right]_{i=0,~j=0}^{n,~n}
\label{eq:forward_Wronskian}
\end{equation}
the Wronskian matrix of those particular integrals
\begin{equation}
v_{n,i}:=\sum_{j=0}^{n}\rho_{i,j}\varphi_{n,j}\in\mathbb{S}_{n}^{0,\beta
},~i=0,1,\ldots,n\label{eq:particular_integrals}%
\end{equation}
of (\ref{eq:differential_equation}) that correspond to the initial conditions%
\begin{equation}
\left\{
\begin{array}{rcl}
v_{n,i}^{\left(  j\right)  }\left(  0\right)    & = & 0,~j=0,\ldots,i-1,\\
v_{n,i}^{\left(  i\right)  }\left(  0\right)    & = & 1,\\
v_{n,i}^{\left(  j\right)  }\left(  \beta\right)    & =& 0 ,~j=0,\ldots,n-1-i,
\end{array}
\right.
\label{eq:initial_conditions}
\end{equation}
i.e., the system $\left\{  v_{n,i}\left(  u\right)  :u\in\left[
0,\beta\right]  \right\}  _{i=0}^{n}$ is a bicanonical basis on the interval $\left[0, \beta\right]$ such that the Wronksian (\ref{eq:forward_Wronskian}) at $u=0$ is a
lower triangular matrix with positive (unit) diagonal entries.

Consider the functions (or Wronskian determinants)%
\begin{equation}
w_{n,i}\left(  u\right)  :=\det W_{\left[  v_{n,i},v_{n,i+1},\ldots
	,v_{n,n}\right]  }\left(  u\right)  ,~i=\left\lfloor \frac{n}{2}\right\rfloor
,\left\lfloor \frac{n}{2}\right\rfloor +1,\ldots
,n,\label{eq:Wronskian_determinants}%
\end{equation}
define the critical length%
\begin{equation}
\beta_{n}^{\star}:=\min_{i=\left\lfloor \frac{n}{2}\right\rfloor ,\left\lfloor \frac{n}{2}\right\rfloor +1,\ldots
	,n}\min\left\{  \left\vert u\right\vert :w_{n,i}\left(  u\right)
=0,~u\neq0\right\}
\label{eq:critical_length}
\end{equation}
and, in what follows, assume that $\beta\in\left(  0,\beta_{n}^{\star}\right)
$ is an arbitrarily fixed shape parameter (we write $\beta_{n}^{\star
}=+\infty$ whenever the Wronskian determinants
(\ref{eq:Wronskian_determinants}) do not have non-zero real zeros). Under these conditions, $\mathbb{S}_{n}^{0,\beta}$ is a reflection and translation invariant EC space that also has a unique normalized B-basis, since $1 \in \mathbb{S}_{n}^{0,\beta}$.

Consider the Wronskian matrix $W_{\left[  v_{n,n},v_{n,n-1},\ldots
	,v_{n,0}\right]  }\left(  \beta\right)  $ of the reverse ordered system
$\{  v_{n,n-i}\left(  u\right)  :\allowbreak{}u\in\left[  0,\beta\right]  \}
_{i=0}^{n}$ at the parameter value $u=\beta$ and obtain its Doolittle factorization%
\[
L\cdot U=W_{\left[  v_{n,n},v_{n,n-1},\ldots,v_{n,0}\right]  }\left(
\beta\right),
\]
where $L$ is a lower triangular matrix with unit diagonal,
while $U$ is a non-singular upper triangular matrix. Calculate the inverse matrices%
\[
U^{-1}:=\left[
\begin{array}
[c]{cccc}%
\mu_{0,0} & \mu_{0,1} & \cdots & \mu_{0,n}\\
0 & \mu_{1,1} & \cdots & \mu_{1,n}\\
\vdots & \vdots & \ddots & \vdots\\
0 & 0 & \cdots & \mu_{n,n}%
\end{array}
\right]  ,~L^{-1}:=\left[
\begin{array}
[c]{cccc}%
\lambda_{0,0} & 0 & \cdots & 0\\
\lambda_{1,0} & \lambda_{1,1} & \cdots & 0\\
\vdots & \vdots & \ddots & \vdots\\
\lambda_{n,0} & \lambda_{n,1} & \cdots & \lambda_{n,n}%
\end{array}
\right]
\]
and construct the reflection invariant normalized B-basis%
\begin{equation}
\mathcal{B}_{n}^{0,\beta}=\left\{  b_{n,i}\left(  u\right)  =\lambda
_{n-i,0}\widetilde{b}_{n,i}\left(  u\right)  :u\in\left[  0,\beta\right]
\right\}  _{i=0}^{n}\label{eq:construction}%
\end{equation}
defined by%
\[
\left[
\begin{array}
[c]{cccc}%
\widetilde{b}_{n,n}\left(  u\right)   & \widetilde{b}_{n,n-1}\left(
u\right)   & \cdots & \widetilde{b}_{0}\left(
u\right)
\end{array}
\right]  :=\left[
\begin{array}
[c]{cccc}%
v_{n,n}\left(  u\right)   & v_{n,n-1}\left(  u\right)   & \cdots &
v_{n,0}\left(  u\right)
\end{array}
\right]  \cdot U^{-1}%
\]
and%
\[
\left[
\begin{array}
[c]{cccc}%
\lambda_{0,0} & \lambda_{1,0} & \cdots & \lambda_{n,0}%
\end{array}
\right]  ^{T}:=L^{-1}\cdot\left[
\begin{array}
[c]{cccc}%
1 & 0 & \cdots & 0
\end{array}
\right]  ^{T}.
\]
Since the EC space $\mathbb{S}_{n}^{0,\beta}$ is invariant under reflections,
one has that%
\[
b_{n,i}\left(  u\right)  =b_{n,n-i}\left(  \beta-u\right)  ,~i=0,1,\ldots
,\left\lfloor \frac{n}{2}\right\rfloor ,
\]
i.e., we only need to determine the half of the basis functions
(\ref{eq:construction}).

\begin{proposition}[Endpoint derivatives]
Assuming that the derivatives $\left\{  \varphi_{n,c}^{\left(  j\right)
}\left(  0\right)  ,\varphi_{n,c}^{\left(  j\right)  }\left(  \beta\right)
\right\}  _{c=0,~j=0}^{n,~n}$ are already known, one has to substitute the
parameter values $u=0$ and $u=\beta$ into the derivative formulas
\begin{align}
b_{n,n-i}^{\left(  j\right)  }\left(  u\right)    & =\lambda_{i,0}%
\widetilde{b}_{n,n-i}^{\left(  j\right)  }\left(  u\right)  \nonumber\\
& =\lambda_{i,0}\sum_{r=0}^{i}\mu_{r,i}v_{n,n-r}^{\left(  j\right)  }\left(
u\right)  \nonumber\\
& =\lambda_{i,0}\sum_{r=0}^{i}\mu_{r,i}\sum_{c=0}^{n}\rho_{n-r,c}\varphi
_{n,c}^{\left(  j\right)  }\left(  u\right)  ,~i=0,1,\ldots,\left\lfloor
\frac{n}{2}\right\rfloor ,\label{eq:mixed_b_derivatives}\\
& \nonumber\\
b_{n,i}^{\left(  j\right)  }\left(  u\right)    & =\left(  -1\right)
^{j}b_{n,n-i}^{\left(  j\right)  }\left(  \beta-u\right)  ,~i=0,1,\ldots
,\left\lfloor \frac{n}{2}\right\rfloor
\label{eq:mixed_b_symmetric_derivatives}
\end{align}
in order to determine the entries (\ref{eq:first_half}) and (\ref{eq:last_half}) of the transformation matrix that
maps the normalized B-basis (\ref{eq:construction}) of the (generally mixed) EC space $\mathbb{S}_{n}^{0,\beta}$ to its ordinary basis. 
\end{proposition}

From computational and algorithmic viewpoints, formulas (\ref{eq:mixed_b_derivatives})--(\ref{eq:mixed_b_symmetric_derivatives}) are significantly easier to both evaluate and implement for fixed values of the shape parameter $\beta$ than to calculate the general determinant based formulas (\ref{eq:general_derivative_0})--(\ref{eq:general_derivative_i}) or to differentiate the integral representation described e.g. in the special case \cite{MainarPena2010} and references therein.

\begin{example}
	[EC spaces generated by characteristic polynomials]Consider the characteristic
	polynomials %
	\begin{align*}
	p_{n+1}\left(  r\right)   
	=r^{n+1},~ 
	p_{\left(  n+1\right)  ^{2}}\left(  r\right)   
	=r^{n+1}\prod_{k=1}%
	^{n}\left(  r^{2}+\omega_{k}^{2}\right)  ^{n+1-k},~%
	p_{\left(  n+1\right)  ^{2}}\left(  r\right)  
	=r^{n+1}\prod_{k=1}^{n}\left(
	r^{2}-\omega_{k}^{2}\right)  ^{n+1-k}, 
	\end{align*}
	where parameters $\left\{  \omega_{k}\right\}  _{k=1}^{n}$ are pairwise
	distinct non-zero real numbers. In these cases, one has that for appropriately selected definition
	domains $\left[  0,\beta\right] $ the vector spaces%
	\begin{align*}
	\mathbb{P}_{n}^{0,\beta}&:=\left\langle\left\{  1,u,\ldots,u^{n}:u\in\left[
	0,\beta\right]  \right\}\right\rangle ,~\dim\mathbb{P}_{n}^{0,\beta}=n+1,\\
	\mathbb{AT}_{n\left(
		n+2\right)  }^{0,\beta}&  :=\mathbb{P}_n \cup \left\langle\left\{  u^{\ell}\cos\left(
	\omega_{k}u\right)  ,u^{\ell}\sin\left(  \omega_{k}u\right)  :u\in\left[
	0,\beta\right]  \right\}  _{k=1,~\ell=0}^{n,~n-k}\right\rangle,~\dim\mathbb{AT}_{n\left(
		n+2\right)  }^{0,\beta}=\left(  n+1\right)  ^{2},
	\end{align*}
	and%
	\[
	\mathbb{AH}_{n\left(
		n+2\right)  }^{0,\beta}:=\mathbb{P}_n \cup\left\langle\left\{  u^{\ell}\cosh\left(
	\omega_{k}u\right)  ,u^{\ell}\sinh\left(  \omega_{k}u\right)  :u\in\left[
	0,\beta\right]  \right\}  _{k=1,~\ell=0}^{n,~n-k}\right\rangle,~\dim\mathbb{AH}_{n\left(
		n+2\right)  }^{0,\beta}=\left(  n+1\right)  ^{2},
	\]
	respectively, are reflection invariant EC spaces that also possess unique normalized B-basis functions. As special cases, the vector spaces of trigonometric and hyperbolic polynomials of order at most $n$ correspond to the characteristic polynomials
	\[
	p_{2n+1}\left(r\right) = r \prod_{k = 1}^n \left(r^2 + k^2\right)
	\text{ and }~
	p_{2n+1}\left(r\right) = r \prod_{k = 1}^n \left(r^2 - k^2\right),
	\]
	respectively, i.e., $\omega_k = k$ for all $k=1,2,\ldots,n$. However, in these two latter cases it is much easier to apply Theorems \ref{thm:trigonometric_endpoint_derivatives} and \ref{thm:hyperbolic_endpoint_derivatives}, respectively, than to evaluate the required endpoint derivatives by means of formulas (\ref{eq:mixed_b_derivatives})-(\ref{eq:mixed_b_symmetric_derivatives}). Concerning the characteristic polynomial $p_{n+1}\left(r\right) = r^{n+1}$, Example \ref{exmp:polynomials} and Appendix \ref{sec:Bernstein_to_monomials} provide further details.
\end{example}

\begin{example}
	[Traditional polynomials]\label{exmp:polynomials}The system 
	$
	\left\{  b_{n,i}\left(  u\right)  =\binom{n}{i}%
	u^{i}\left(  1-u\right)  ^{n-i}:u\in\left[  0,1\right]  \right\}  _{i=0}^{n}%
	$
	of Bernstein polynomials of degree $n$ is the normalized B-basis of the
	EC space %
	$
	\left\langle \left\{  \varphi_{n,i}\left(  u\right)  =u^{i}:u\in\left[
	0,1\right]  \right\}  _{i=0}^{n}\right\rangle .
	$
	In this case one has that%
	\begin{align}
	\varphi_{n,i}^{\left(  j\right)  }\left(  0\right)    & =\left\{
	\begin{array}
	[c]{ll}%
	i!, & ~j=i,\\
	0, & ~j\neq i,
	\end{array}
	\right.  ~~\varphi_{n,i}^{\left(  j\right)  }\left(  1\right)  =\left\{
	\begin{array}
	[c]{ll}%
	\frac{i!}{\left(  i-j\right)  !}, & i\geq j\geq0,\\
	0, & i<j\leq n,
	\end{array}
	\right.  
	\label{eq:monomial_derivative}\\
	b_{n,i}^{\left(  j\right)  }\left(  0\right)    & =\left\{
	\begin{array}
	[c]{ll}%
	0, & ~i>j\geq0,\\
	\left(  -1\right)  ^{j-i}\cdot j!\cdot\binom{n}{j}\cdot\binom{j}{i}, & ~i\leq
	j\leq n,
	\end{array}
	\right.  ~~b_{n,i}^{\left(  j\right)  }\left(  1\right)  =\left(  -1\right)
	^{j}b_{n,n-i}^{\left(  j\right)  }\left(  0\right)  ,~i=0,1,\ldots,n
	\label{eq:Bernstein_derivative}
	\end{align}
	for all $j=0,1,\ldots,n$. 
	As it is proved in Appendix \ref{sec:Bernstein_to_monomials}, the substitution for $\alpha=0$ and $\beta=1$ of these derivatives
	into formulas (\ref{eq:first_half}) and (\ref{eq:last_half}) leads to the expected closed form of the classical transformation matrix of entries %
	\begin{equation}
	\def\arraystretch{1.75}
	t_{i,j}^{n}=\left\{
	\begin{array}
	[c]{cc}%
	\frac{\binom{j}{i}}{\binom{n}{i}}, & j=i,i+1,\ldots,n,\\
	0, & j=0,1,\ldots,i-1,
	\end{array}
	\right.
	\def\arraystretch{1.0}
	\label{eq:Bernstein_to_monomials}
	\end{equation}
	where $i=0,1,\ldots,n$.
\end{example}

Naturally, characteristic polynomials may also have (conjugate) complex roots of higher order multiplicity and with non-vanishing real and imaginary parts, which may lead to mixed (algebraic) exponential trigonometric EC spaces as it is illustrated in the next example.

\begin{example}[A 5-dimensional exponential trigonometric space]
	\label{exmp:exponential_trigonometric_space}
	Let $\omega > 0$ be a fixed parameter and consider the $5$th order homogeneous linear differential equation
	\begin{equation}
	v^{\left(5\right)}\left(u\right) - 2\left(\omega^2 - 1\right) v^{\left(3\right)}\left(u\right) + \left(\omega^2 + 1\right)^2 v^{\left(1\right)}\left(u\right) = 0
	\label{eq:5th_order_exp_trig}
	\end{equation}
	with the odd characteristic polynomial
	\begin{align*}
	p_5\left(r\right) &= r^5 - 2\left(\omega^2 - 1\right) r^3 + \left(\omega^2 + 1\right)^2 r 
	\\
	&= r\left(r - \left(-\omega - \mathbf{i}\right)\right)\left(r - \left(-\omega + \mathbf{i}\right)\right)\left(r - \left(\omega - \mathbf{i}\right)\right)\left(r - \left(\omega + \mathbf{i}\right)\right),
	\end{align*}
	where $\mathbf{i}=\sqrt{-1}$. It follows that the vector space formed by all solutions of (\ref{eq:5th_order_exp_trig}) can be spanned by the ordinary basis
	\begin{align}
	\mathcal{F}_{4}^{\beta} = &~
	\left\{
	\varphi_0\left(u\right) \equiv 1, \,
	\varphi_1\left(u\right) = e^{-\omega u} \cos\left(u\right),\,
	\varphi_2\left(u\right) = e^{-\omega u} \sin\left(u\right),\,
	\right.
	\label{eq:5th_order_exp_trig_ordinary_basis}\\
	&
	~~\left.
	\varphi_3\left(u\right) = e^{\omega u} \cos\left(u\right),\,
	\varphi_4\left(u\right) = e^{\omega u} \sin\left(u\right):
	u\in\left[0,\beta\right]
	\right\},~\beta \in \left(0,\beta_4^{\star}\right),
	\nonumber
	\end{align}
	where $\beta_4^{\star}$ denotes the corresponding special case of the critical length (\ref{eq:critical_length}). In order to avoid lengthy cumbersome formulations, in this case we provide only a numerical example the values of which can be verified by means of Listings \ref{lst:Doolittle} and \ref{lst:ExponentialTrigonometricSpace} of Appendix \ref{sec:implementation_details}. Assume that the growth rate $\omega = \frac{1}{3\pi}$ and the shape parameter $\beta = \frac{5\pi}{6}<\beta_4^{\star}$ are fixed. If one intends e.g.\ to represent the arc
	\begin{equation}
	\mathbf{c}\left(u\right) = 
	\left[
	\begin{array}{c}
	e^{\omega u} \cos\left(u\right)\\
	e^{\omega u} \sin\left(u\right)
	\end{array}
	\right]
	=
	\left[
	\begin{array}{c}
	1\\0
	\end{array}
	\right] \varphi_{3}\left(u\right) +
	\left[
	\begin{array}{c}
	0\\1
	\end{array}
	\right] \varphi_{4}\left(u\right),~ u \in \left[0,\beta\right]
	\label{eq:logarithmic_spiral}
	\end{equation}
	of a logarithmic spiral by means of the normalized B-basis $\left\{b_{4,i}\left(u\right):u\in\left[0,\beta\right]\right\}_{i=0}^4$ of the underlying reflection invariant EC space $\mathbb{S}_{4}^{0,\beta}$, then one has to construct the system (\ref{eq:construction}) as follows:
	\begin{itemize}[nolistsep]
		\item
		at first, one has to determine the transformation matrix
		\begin{footnotesize}
			\begin{equation*}
			\left[\rho_{i,j}\right]_{i=0,\,j=0}^{4,\,4}
			=
			\left[
			\begin{array}{rrrrr}
			0.8038 &  -0.7765 & -2.4061 &  0.9728 &  1.0761\\
			-1.2028 &   4.2007 &  2.2514 & -2.9979 & -0.4876\\
			1.4484 &  -4.8494 &  0.5805 &  3.4009 & -1.4558\\
			-1.1191 &   2.8895 & -2.3598 & -1.7704 &  2.8542\\
			0.9779 &  -0.4889 &  2.2781 & -0.4889 & -2.2781
			\end{array}
			\right]
			\end{equation*}
		\end{footnotesize}that maps the ordinary basis $\left\{\varphi_{4,i}\left(u\right):u\in\left[0,\beta\right]\right\}_{i=0}^4$ to the particular integrals $\{v_{4,i}\left(u\right):u\in\left[0,\beta\right]\}_{i=0}^4$ of the form (\ref{eq:particular_integrals}) which fulfill the initial conditions (\ref{eq:initial_conditions});
		\item
		next, one has to obtain the Doolittle $LU$-decomposition of the Wronskian matrix
		\begin{footnotesize}
			\begin{equation*}
			W_{\left[  v_{4,4},v_{4,3},v_{4,2},v_{4,1}, v_{4,0}\right]}\left(
			\beta\right)
			= 
			L \cdot U
			=
			\left[
			\begin{array}{rrrrr}
			1.2166 & -0.0000 & -0.0000 &       0 & -0.0000 \\
			1.3923 & -0.9304 & -0.0000 &       0 & -0.0000 \\
			0.6629 & -1.6472 &  1.0000 & -0.0000 &       0 \\
			-0.8886 & -0.5383 &  1.7705 & -1.0748 &  0.0000 \\
			-1.5551 &  2.1130 & -0.4963 & -1.2300 &  0.8219
			\end{array}
			\right],
			\end{equation*}
		\end{footnotesize}of the reversed ordered system $\left\{v_{4,4-i}\left(u\right):u\in\left[0,\beta\right]\right\}_{i=0}^4$ at $u = \beta$, i.e.,
		\begin{footnotesize}
			\begin{equation*}
			L
			=
			\left[
			\begin{array}{rrrrr}
			1.0000 &        0 &        0 &       0 &       0 \\
			1.1444 &   1.0000 &        0 &       0 &       0 \\
			0.5449 &   1.7705 &   1.0000 &       0 &       0 \\
			-0.7303 &   0.5786 &   1.7705 &  1.0000 &       0 \\
			-1.2782 &  -2.2711 &  -0.4963 &  1.1444 &  1.0000
			\end{array}
			\right],\,
			U
			=
			\left[
			\begin{array}{rrrrr}
			1.2166 & -0.0000 & -0.0000 &        0 & -0.0000 \\
			0      & -0.9304 &  0.0000 &        0 & -0.0000 \\
			0      &       0 &  1.0000 &  -0.0000 &  0.0000 \\
			0      &       0 &       0 &  -1.0748 & -0.0000 \\
			0      &       0 &       0 &        0 &  0.8219
			\end{array}
			\right];
			\end{equation*}
		\end{footnotesize}
		\item
		then, by using the essential parts of the inverse matrices
		\begin{footnotesize}
			\begin{equation*}
			L^{-1}
			=
			\left[
			\begin{array}{rrrrr}
			1.0000 = \lambda_{0,0} &     0 &      0 &      0 &      0 \\
			-1.1444 = \lambda_{1,0} & \star &      0 &      0 &      0 \\
			1.4812 = \lambda_{2,0} & \star &  \star &      0 &      0 \\
			\star & \star &  \star &  \star &      0 \\
			\star & \star &  \star &  \star &  \star
			\end{array}
			\right],\,
			U^{-1}
			=
			\left[
			\begin{array}{rrrrr}
			0.8219 = \mu_{0,0} & -0.0000 = \mu_{0,1} & 0.0000 = \mu_{0,2}  &  \star &  \star \\
			0                  & -1.0748 = \mu_{1,1} &  0.0000 = \mu_{1,2} &  \star &  \star \\
			0                  &       0             &  1.0000 = \mu_{2,2} &  \star &  \star \\
			0                  &       0             &       0             &  \star &  \star \\
			0                  &       0             &       0             &    0  &  \star
			\end{array}
			\right],
			\end{equation*}
		\end{footnotesize}	
		and the reflection invariant property of the vector space, one has that
		\begin{align}
		b_{4,4}\left(u\right) &= \lambda_{0, 0} \mu_{0,0} v_{4,4}\left(u\right),& b_{4,0}\left(u\right) &= b_{4,4}\left(\beta - u\right),\nonumber\\
		b_{4,3}\left(u\right) &= \lambda_{1, 0} \left(\mu_{0,1} v_{4,4}\left(u\right) + \mu_{1,1} v_{4,3}\left(u\right)\right),& b_{4,1}\left(u\right) &= b_{4,3}\left(\beta - u\right),\label{eq:5th_order_exp_trig_normalized_B_basis}\\
		b_{4,2}\left(u\right) &= \lambda_{2, 0} \left(\mu_{0,2} v_{4,4}\left(u\right) + \mu_{1,2} v_{4,3}\left(u\right) + \mu_{2,2}v_{4,2}\left(u\right)\right),& u&\in\left[0,\beta\right].\nonumber
		\end{align}
	\end{itemize}
	Fig.\ \ref{fig:exponential_trigonometric_space}(\textit{a}) shows the image of these normalized B-basis functions.
	
	\begin{figure}[!h]
		\centering
		\includegraphics[scale=1]{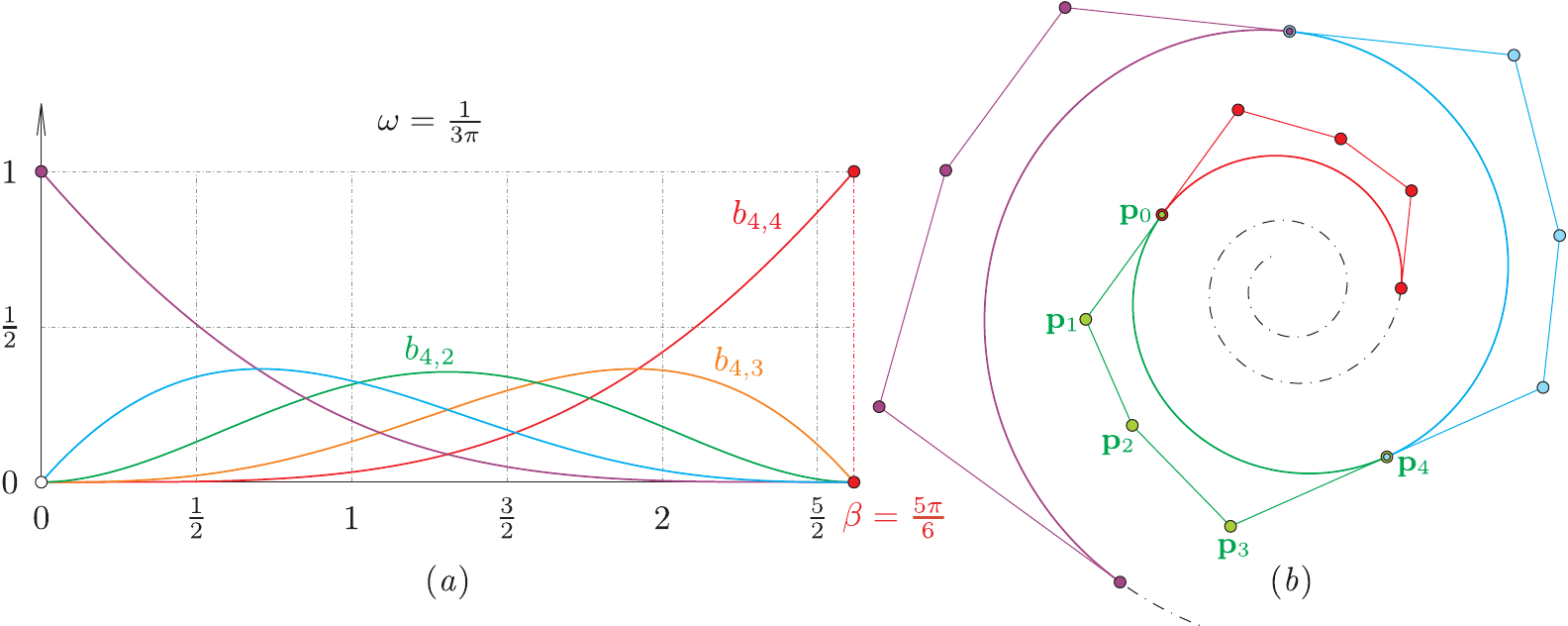}
		\caption{(\textit{a}) Exponential trigonometric normalized B-basis functions of order $4$ which correspond to the shape parameter $\beta = \frac{5\pi}{6}$ and growth rate $\omega = \frac{1}{3\pi}$. (\textit{b}) Control point based exact description of different arcs of the logarithmic spiral (\ref{eq:logarithmic_spiral}).}
		\label{fig:exponential_trigonometric_space}
	\end{figure}
	
	Using the higher order derivatives of the ordinary basis functions (\ref{eq:5th_order_exp_trig_ordinary_basis}) at $u=0$ and $u = \beta$, one can also easily evaluate the higher order derivatives of the obtained normalized B-basis functions by means of formulas (\ref{eq:mixed_b_derivatives})--(\ref{eq:mixed_b_symmetric_derivatives}). Substituting these derivatives into (\ref{eq:first_half})--(\ref{eq:last_half}), one also obtains the transformation matrix
	\begin{footnotesize}
		\begin{equation*}
		\left[t_{i,j}\right]_{i=0,\,j=0}^{4,\,4}
		=
		\left[
		\begin{array}{rrrrr}
		1.0000 & 1.0000 & 1.0000 &  1.0000 &  1.0000 \\
		1.0000 & 0.9073 & 0.2057 & -0.3859 & -0.6560 \\
		0      & 0.8738 & 1.0520 &  0.9871 &  0.3787 \\
		1.0000 & 1.0927 & 0.4593 & -0.4605 & -1.1433 \\
		0      & 0.8738 & 1.3386 &  1.5980 &  0.6601
		
		\end{array}
		\right]
		\end{equation*}
	\end{footnotesize}that is required for the control point based exact description (\ref{eq:convex_combination}) of any arc of the logarithmic spiral (\ref{eq:logarithmic_spiral}) that is defined over an interval of length $\beta$ (see Fig.\ \ref{fig:exponential_trigonometric_space}(\textit{b})). Observe that from algorithmic and implementation viewpoints, the steps above are much easier and more efficient to perform than the evaluation of other possible integral or determinant based representations. As long as parameters $\omega$ or $\beta$ are not modified, the calculations above do not have to be reevaluated.
\end{example}

\begin{example}
	[Quadratic algebraic trigonometric functions]%
	\label{exmp:algebraic_trigonometric}The normalized B-basis%
	\begin{align*}
		\mathcal{B}_{4}^{0,\beta}=  &  ~\left\{  b_{4,0}\left(  u\right)
		=b_{4,4}\left(  \beta-u\right)  ,~b_{4,1}\left(  u\right)  =b_{4,3}\left(
		\beta-u\right)  ,\right. \\
		&  ~~\left.  b_{4,2}\left(  u\right)  =c_{4,2}^{\beta}\left(  2\beta\left(
		\sin\left(  u\right)  -\sin\left(  \beta\right)  \right)  -2\beta\left(
		1-\cos\left(  \beta\right)  \right)  u+\beta^{2}+2\beta\sin\left(
		\beta-u\right)  -\beta^{2}\cos\left(  \beta-u\right)  +\right.  \right. \\
		&  ~~~~~~~~~~~~~~\left.  +\beta^{2}\left(  \cos\left(  \beta\right)
		-\cos\left(  u\right)  \right)  +2\left(  1-\cos\left(  \beta\right)  \right)
		u^{2}+\beta\left(  \beta-u\right)  u\sin\left(  \beta\right)  \right)  ,\\
		&  ~~\left.  b_{4,3}\left(  u\right)  =c_{4,3}^{\beta}\left(  2\left(
		\beta-u\right)  +2\left(  \sin\left(  u\right)  -\sin\left(  \beta\right)
		\right)  +2\left(  u\cos\left(  \beta\right)  -\beta\cos\left(  u\right)
		\right)  +2\sin\left(  \beta-u\right)  +\right.  \right. \\
		&  ~~~~~~~~~~~~~~\left.  +\beta^{2}\left(  u-\sin\left(  u\right)  \right)
		-\left(  \beta-\sin\left(  \beta\right)  \right)  u^{2}\right)  ,\\
		&  ~~\left.  b_{4,4}\left(  u\right)  =c_{4,4}^{\beta}\left(  2\cos\left(
		u\right)  +u^{2}-2\right)  :u\in\left[  0,\beta\right]  \right\}  ,~\beta
		\in\left(  0,\beta^{\star}_4\right), \,\beta_4^{\star} = 2\pi
	\end{align*}
	of the EC space %
	$
	\mathbb{S}_{4}^{0,\beta}    =\langle \mathcal{F}_{4}^{0,\beta
	}\rangle 
	=\langle \{  \varphi_{4,0}\left(  u\right)  =1,~\varphi
	_{4,1}\left(  u\right)  =u,~\varphi_{4,2}\left(  u\right)  =u^{2}%
	,\varphi_{4,3}\left(  u\right)  =\sin\left(  u\right)  ,~\varphi_{4,4}\left(
	u\right)  =\cos\left(  u\right)  :u\in\left[  0,\beta\right] \}
	\rangle
	$
	of algebraic trigonometric functions can also be constructed, e.g.\ by using either
	the differential equation based iterative integral representation published in
	\cite{MainarPena2010} and references therein or the determinant based formulas
	of \cite[Theorem 3.4]{Mazure1999}. The critical length %
	$\beta^{\star}_4 = 2\pi$ was determined in \cite[Section 5]{CarnicerMainarPena2004} or \cite[Proposition 3]{CarnicerMainarPena2007}, while positive scalars
	\begin{footnotesize}
		\begin{align*}
			c_{4,2}^{\beta}  &  =\frac{4-4\cos\left(  \beta\right)  -2\beta\sin\left(
				\beta\right)  }{\left(  \beta^{2}-4\cos\left(  \beta\right)  -4\beta
				\sin\left(  \beta\right)  +\beta^{2}\cos\left(  \beta\right)  +4\right)  ^{2}%
			},\\
			c_{4,3}^{\beta}  &  =\frac{2\left(  \beta-\sin\left(  \beta\right)  \right)
			}{\left(  2\cos\left(  \beta\right)  +\beta^{2}-2\right)  \left(  \beta
			^{2}-4\cos\left(  \beta\right)  -4\beta\sin\left(  \beta\right)  +\beta
			^{2}\cos\left(  \beta\right)  +4\right)  },\\
		c_{4,4}^{\beta}  &  =\frac{1}{2\cos\left(  \beta\right)  +\beta^{2}-2}
	\end{align*}
\end{footnotesize}are normalizing coefficients.
Applying Theorem \ref{thm:basis_transformation} with the settings above, one
obtains the transformation matrix%
\begin{equation}
	\left[  t_{i,j}^{4}\right]  _{i=0,~j=0}^{4,~4}=%
	\begin{footnotesize}%
		\def\arraystretch{1.75}%
		\left[
		\begin{array}
			[c]{ccccc}%
			1 & 1 & 1 & 1 & 1\\
			0 & \frac{2\cos\beta+\beta^{2}-2}{2\left(  \beta-\sin\beta\right)  } &
			\frac{\left(  2-2\cos\beta-\beta\sin\beta\right)  \beta}{4-4\cos\beta
				-2\beta\sin\beta} & \beta-\frac{2\cos\beta+\beta^{2}-2}{2\left(  \beta
				-\sin\beta\right)  } & \beta\\
			0 & 0 & \frac{\beta^{2}-4\cos\beta-4\beta\sin\beta+\beta^{2}\cos\beta
				+4}{2-2\cos\beta-\beta\sin\beta} & \beta^{2}-\frac{\left(  2\cos\beta
				+\beta^{2}-2\right)  \beta}{\left(  \beta-\sin\beta\right)  } & \beta^{2}\\
			0 & \frac{2\cos\beta+\beta^{2}-2}{2\left(  \beta-\sin\beta\right)  } &
			\frac{\left(  2-2\cos\beta-\beta\sin\beta\right)  \beta}{4-4\cos\beta
				-2\beta\sin\beta} & \sin\left(  \beta\right)  -\frac{\left(  2\cos\beta
				+\beta^{2}-2\right)  \cos\left(  \beta\right)  }{2\left(  \beta-\sin
				\beta\right)  } & \sin\left(  \beta\right) \\
			1 & 1 & \frac{\left(  2\sin\beta-\beta-\beta\cos\beta\right)  \beta}%
			{4-4\cos\beta-2\beta\sin\beta} & \cos\left(  \beta\right)  +\frac{\left(
				2\cos\beta+\beta^{2}-2\right)  \sin\left(  \beta\right)  }{2\left(  \beta
				-\sin\beta\right)  } & \cos\left(  \beta\right)
		\end{array}
		\right]
		\def\arraystretch{1.0}%
	\end{footnotesize}
	\label{eq:algebraic_trigonometric_basis_transformation}%
\end{equation}
that maps $\mathcal{B}_{4}^{0,\beta}$ to $\mathcal{F}_{4}^{0,\beta}$, based on which Fig.\ \ref{fig:cycloids_and_helices} illustrates the
control point based exact description of cycloids and helices of different
shape parameters, while Fig.\ \ref{fig:cylindrical_helicoid_and_hyperboloid}(a)
shows the control net of a cylindrical helicoid.
\end{example}

\begin{figure}
	[!h]
	\begin{center}
		\includegraphics[scale = 0.925]{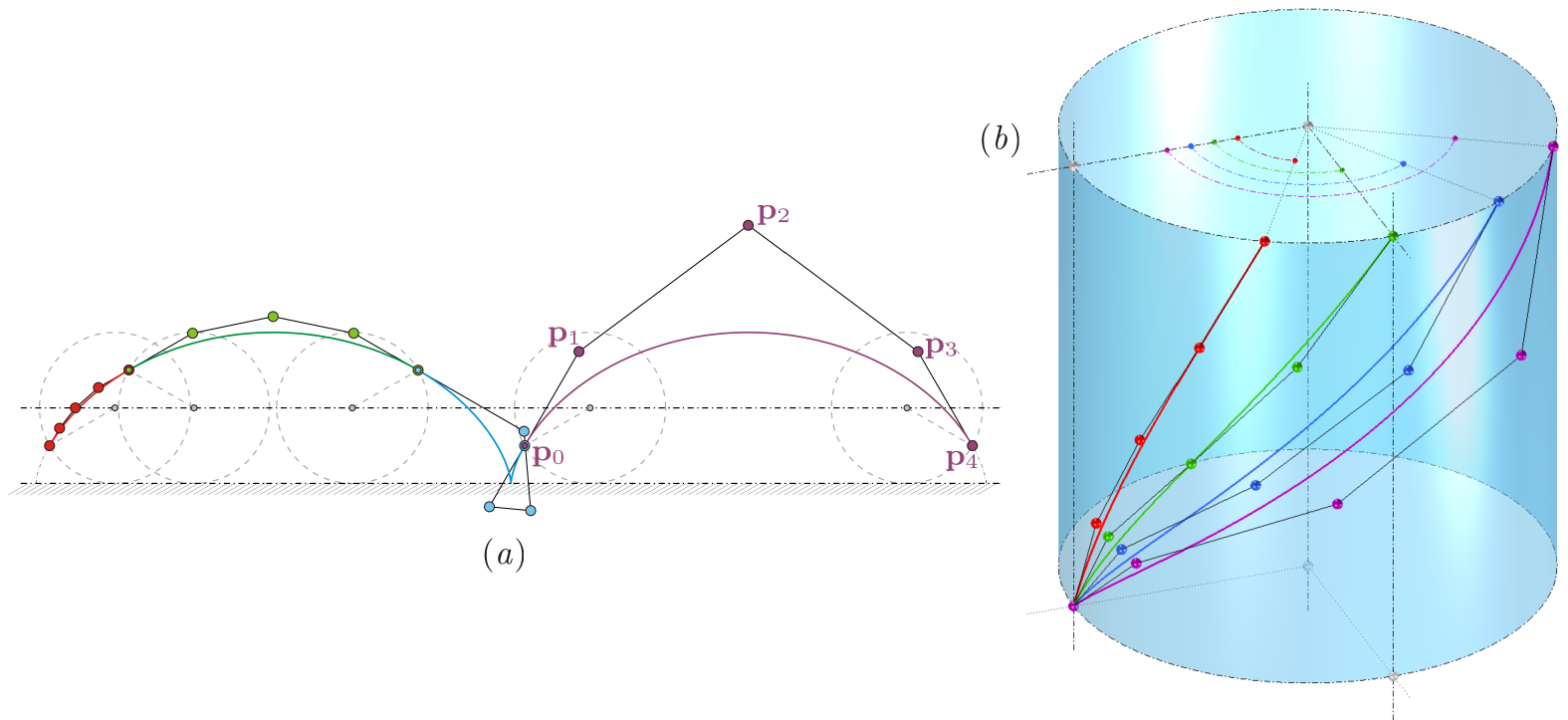}
		
		\caption{(\textit{a}) Cycloids and (\textit{b}) helices of different shape
			parameters described by means of Theorem \ref{thm:integral_curves} and of the
			basis transformation detailed in Example \ref{exmp:algebraic_trigonometric}.}%
		
		\label{fig:cycloids_and_helices}%
	\end{center}
\end{figure}

\begin{remark}
	[Hybrid EC B-surfaces]Naturally, one can also combine different
	types of normalized B-basis functions in order to describe hybrid surfaces as
	it is shown in Fig.\ \ref{fig:cylindrical_helicoid_and_hyperboloid}(b) that
	illustrates the control point based exact description of a hyperboloidal patch.
\end{remark}

\begin{figure}
	[!h]
	\begin{center}
		\includegraphics[scale = 0.925]%
		{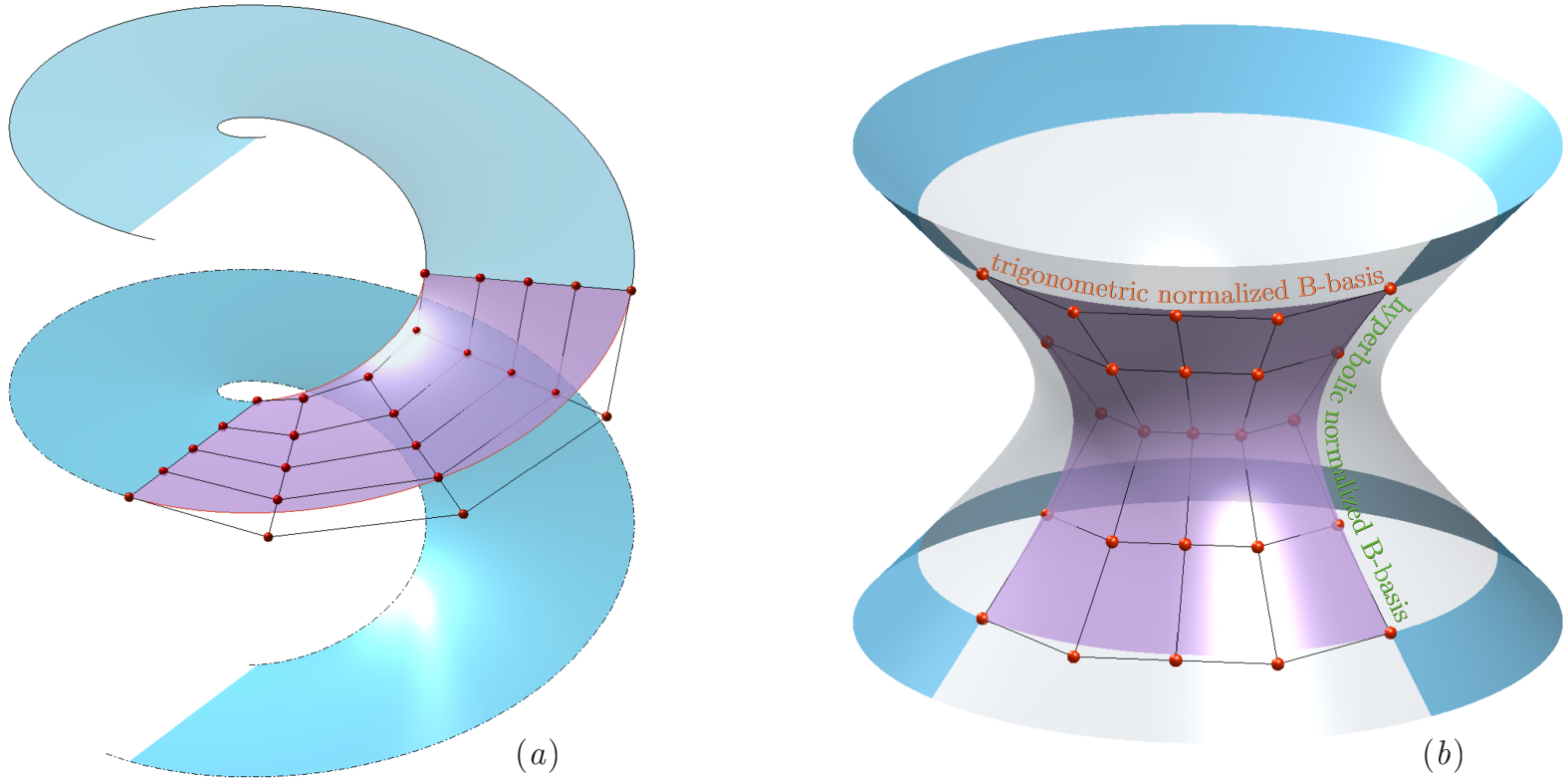}%
		
		\caption{Control point based exact description of (\textit{a}) helicoidal and
			of (\textit{b}) hyperboloidal patches, respectively. The control net of the
			patch (\textit{a}) was constructed in both direction by means of basis
			transformations of the type
			(\ref{eq:algebraic_trigonometric_basis_transformation}) with different shape
			parameters. In case of patch (\textit{b}) the control point configuration was
			obtained by using both the trigonometric and hyberbolic basis transformations
			(\ref{eq:trigonometric_basis_transformation}) and
			(\ref{eq:hyperbolic_basis_transformation}), respectively.}%
		
		\label{fig:cylindrical_helicoid_and_hyperboloid}%
	\end{center}
\end{figure}

\section{Proof of main results\label{sec:proofs}}

\begin{proof}
	[Proof of Theorem \ref{thm:basis_transformation}]The linear transformation
	$[  t_{i,j}^{n}]  _{i=0,j=0}^{n,n}$ that maps the normalized
	B-basis $\mathcal{B}_{n}^{\alpha,\beta}$ of the vector space $\mathbb{S}%
	_{n}^{\alpha,\beta}$ to its ordinary basis $\mathcal{F}_{n}^{\alpha,\beta}$
	will be constructed by mathematical induction on the column index $j$ or
	$n-j$, where $j=0,1,\ldots,\lfloor \frac{n}{2}\rfloor $. Using one
	of the properties (\ref{eq:partition_of_unity}%
	)--(\ref{eq:Hermite_conditions_alpha}), at each step $j$ we will compare the
	left and right side of the $j$th order derivative of the matrix equality
	(\ref{eq:basis_transformation}), thus obtaining an iterative process that is
	outlined in Fig.\ \ref{fig:basis_transformation}.%
	
	\begin{figure}
		[!h]
		\begin{center}
			\includegraphics[scale = 0.95]{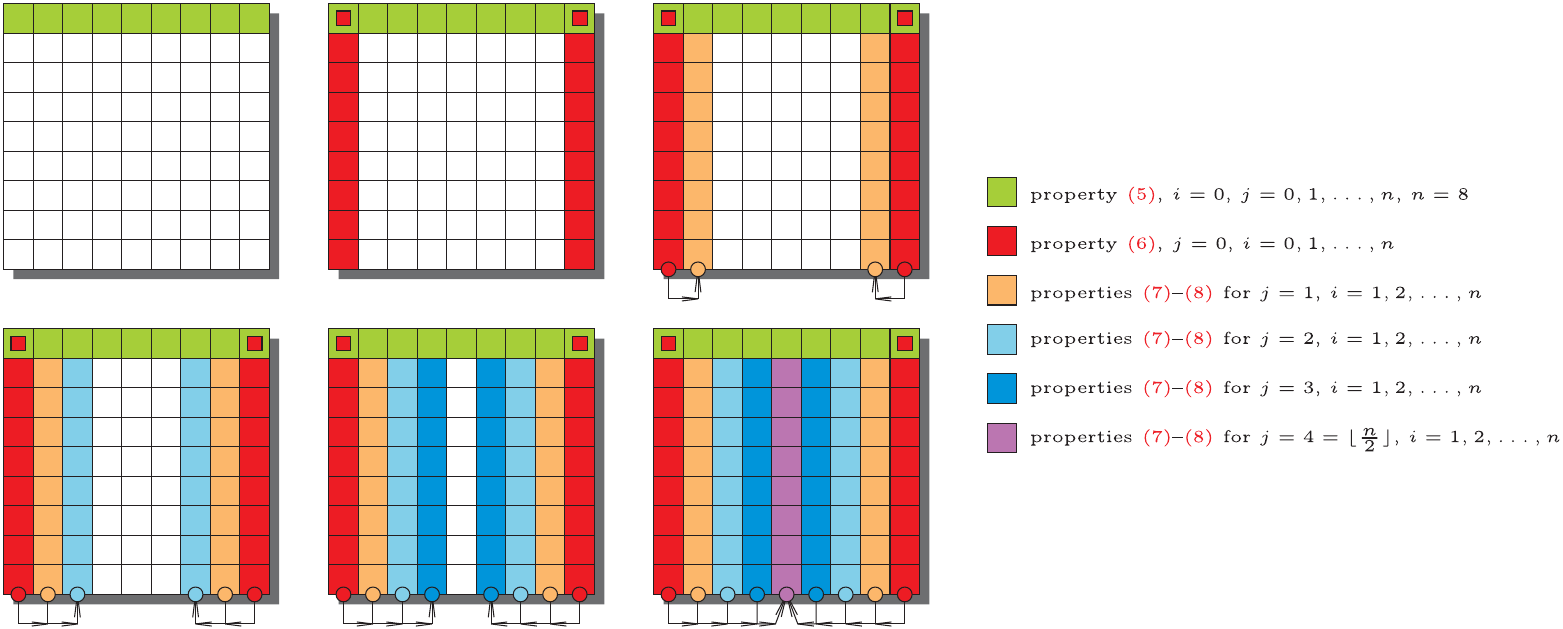}%
			\caption{Outline of the proof.}%
			\label{fig:basis_transformation}%
		\end{center}
	\end{figure}
	
	First of all, observe that
	\[
	t_{0,j}=1,~\forall j=0,1,\ldots,n
	\]
	and%
	\[
	t_{i,0}^{n}b_{n,i}\left(  \alpha\right)  =t_{i,0}^{n}=\varphi_{n,i}\left(
	\alpha\right)  ,~~~t_{i,n}^{n}b_{n,n}\left(  \beta\right)  =t_{i,n}%
	^{n}=\varphi_{n,i}\left(  \beta\right)  ,~i=0,1,\ldots,n,
	\]
	due to the partition of unity property (\ref{eq:partition_of_unity}) and to
	the endpoint interpolation property (\ref{eq:endpoint_interpolation}),
	respectively. Using forward substitutions, the elements of the columns %
	$
	[  t_{i,j}^{n}]  _{i=1}^{n},~j=1,2,\ldots,\lfloor \frac{n}%
	{2}\rfloor
	$
	are iteratively determined by differentiating the matrix equality
	(\ref{eq:basis_transformation}) with gradually increasing order and applying
	the Hermite conditions (\ref{eq:Hermite_conditions_0}) at $u=\alpha$. In order to
	formulate a mathematical induction hypothesis, let us consider some special
	cases. When $j=1$ one obtains that%
	\[
	\varphi_{n,i}^{\left(  1\right)  }\left(  \alpha\right)  =t_{i,0}^{n}%
	b_{n,0}^{\left(  1\right)  }\left(  \alpha\right)  +t_{i,1}^{n}b_{n,1}%
	^{\left(  1\right)  }\left(  \alpha\right)  ,~i=0,1,\ldots,n,
	\]
	where $b_{n,1}^{\left(  1\right)  }\left(  \alpha\right)  \neq0$ and for the
	special subcase $i=0$ one has that%
	\[
	b_{n,0}^{\left(  1\right)  }\left(  \alpha\right)  +b_{n,1}^{\left(  1\right)
	}\left(  \alpha\right)  =\varphi_{n,0}^{\left(  1\right)  }\left(
	\alpha\right)  =0,
	\]
	i.e.,%
	\begin{align*}
	b_{n,0}^{\left(  1\right)  }\left(  \alpha\right)   &  =-b_{n,1}^{\left(
		1\right)  }\left(  \alpha\right)  ,\\
	& \\
	t_{i,1}^{n}  &  =\frac{1}{b_{n,1}^{\left(  1\right)  }\left(  \alpha\right)
	}\left(  \varphi_{n,i}^{\left(  1\right)  }\left(  \alpha\right)  -t_{i,0}%
	^{n}b_{n,0}^{\left(  1\right)  }\left(  \alpha\right)  \right)  =\frac
	{1}{b_{n,1}^{\left(  1\right)  }\left(  \alpha\right)  }\left(  \varphi
	_{n,i}^{\left(  1\right)  }\left(  \alpha\right)  +\varphi_{n,i}\left(
	\alpha\right)  b_{n,1}^{\left(  1\right)  }\left(  \alpha\right)  \right)
	=\varphi_{n,i}\left(  \alpha\right)  +\frac{\varphi_{n,i}^{\left(  1\right)
		}\left(  \alpha\right)  }{b_{n,1}^{\left(  1\right)  }\left(  \alpha\right)
	},\\
	i  &  =1,2,\ldots,n.
	\end{align*}
	For $j=2$, we have that%
	\[
	\varphi_{n,i}^{\left(  2\right)  }\left(  \alpha\right)  =t_{i,0}^{n}%
	b_{n,0}^{\left(  2\right)  }\left(  \alpha\right)  +t_{i,1}^{n}b_{n,1}%
	^{\left(  2\right)  }\left(  \alpha\right)  +t_{i,2}^{n}b_{n,2}^{\left(
		2\right)  }\left(  \alpha\right)  ,
	\]
	where $b_{n,2}^{\left(  2\right)  }\left(  \alpha\right)  \neq0$ and for the
	special subcase $i=0$ we obtain that%
	\[
	b_{n,0}^{\left(  2\right)  }\left(  \alpha\right)  +b_{n,1}^{\left(  2\right)
	}\left(  \alpha\right)  +b_{n,2}^{\left(  2\right)  }\left(  \alpha\right)
	=\varphi_{n,0}^{\left(  2\right)  }\left(  \alpha\right)  =0,
	\]
	i.e.,%
	\begin{align*}
	b_{n,0}^{\left(  2\right)  }\left(  \alpha\right)  +b_{n,1}^{\left(  2\right)
	}\left(  \alpha\right)   &  =-b_{n,2}^{\left(  2\right)  }\left(
	\alpha\right)  ,\\
	& \\
	t_{i,2}^{n}  &  =\frac{1}{b_{n,2}^{\left(  2\right)  }\left(  \alpha\right)
	}\left(  \varphi_{n,i}^{\left(  2\right)  }\left(  \alpha\right)  -t_{i,0}%
	^{n}b_{n,0}^{\left(  2\right)  }\left(  \alpha\right)  -t_{i,1}^{n}%
	b_{n,1}^{\left(  2\right)  }\left(  \alpha\right)  \right) \\
	&  =\frac{1}{b_{n,2}^{\left(  2\right)  }\left(  \alpha\right)  }\left(
	\varphi_{n,i}^{\left(  2\right)  }\left(  \alpha\right)  -\varphi_{n,i}\left(
	\alpha\right)  b_{n,0}^{\left(  2\right)  }\left(  \alpha\right)  -\left(
	\varphi_{n,i}\left(  \alpha\right)  +\frac{\varphi_{n,i}^{\left(  1\right)
		}\left(  \alpha\right)  }{b_{n,1}^{\left(  1\right)  }\left(  \alpha\right)
	}\right)  b_{n,1}^{\left(  2\right)  }\left(  \alpha\right)  \right) \\
	&  =\varphi_{n,i}\left(  \alpha\right)  -\frac{1}{b_{n,2}^{\left(  2\right)
		}\left(  \alpha\right)  }\cdot\frac{\varphi_{n,i}^{\left(  1\right)  }\left(
		\alpha\right)  }{b_{n,1}^{\left(  1\right)  }\left(  \alpha\right)  }%
	b_{n,1}^{\left(  2\right)  }\left(  \alpha\right)  +\frac{\varphi
		_{n,i}^{\left(  2\right)  }\left(  \alpha\right)  }{b_{n,2}^{\left(  2\right)
		}\left(  \alpha\right)  },~i=1,2,\ldots,n.
	\end{align*}
	In case of $j=3$ one obtains that%
	\[
	\varphi_{n,i}^{\left(  3\right)  }\left(  \alpha\right)  =t_{i,0}^{n}%
	b_{n,0}^{\left(  3\right)  }\left(  \alpha\right)  +t_{i,1}^{n}b_{n,1}%
	^{\left(  3\right)  }\left(  \alpha\right)  +t_{i,2}^{n}b_{n,2}^{\left(
		3\right)  }\left(  \alpha\right)  +t_{i,3}^{n}b_{n,3}^{\left(  3\right)
	}\left(  \alpha\right)  ,~i=0,1,\ldots,n,
	\]
	where $b_{n,3}^{\left(  3\right)  }\left(  \alpha\right)  \neq0$ and for the
	special subcase $i=0$, one has that%
	\[
	b_{n,0}^{\left(  3\right)  }\left(  \alpha\right)  +b_{n,1}^{\left(  3\right)
	}\left(  \alpha\right)  +b_{n,2}^{\left(  3\right)  }\left(  \alpha\right)
	+b_{n,3}^{\left(  3\right)  }\left(  \alpha\right)  =\varphi_{n,0}^{\left(
		3\right)  }\left(  \alpha\right)  =0,
	\]
	i.e.,%
	\begin{align*}
	-b_{n,3}^{\left(  3\right)  }\left(  \alpha\right)  =  &  ~b_{n,0}^{\left(
		3\right)  }\left(  \alpha\right)  +b_{n,1}^{\left(  3\right)  }\left(
	\alpha\right)  +b_{n,2}^{\left(  3\right)  }\left(  \alpha\right)  ,\\
	& \\
	t_{i,3}^{n}=  &  ~\frac{1}{b_{n,3}^{\left(  3\right)  }\left(  \alpha\right)
	}\left(  \varphi_{n,i}^{\left(  3\right)  }\left(  \alpha\right)  -t_{i,0}%
	^{n}b_{n,0}^{\left(  3\right)  }\left(  \alpha\right)  -t_{i,1}^{n}%
	b_{n,1}^{\left(  3\right)  }\left(  \alpha\right)  -t_{i,2}^{n}b_{n,2}%
	^{\left(  3\right)  }\left(  \alpha\right)  \right) \\
	=  &  ~\frac{1}{b_{n,3}^{\left(  3\right)  }\left(  \alpha\right)  }\left(
	\varphi_{n,i}^{\left(  3\right)  }\left(  \alpha\right)  -\varphi_{n,i}\left(
	\alpha\right)  b_{n,0}^{\left(  3\right)  }\left(  \alpha\right)  -\left(
	\varphi_{n,i}\left(  \alpha\right)  +\frac{\varphi_{n,i}^{\left(  1\right)
		}\left(  \alpha\right)  }{b_{n,1}^{\left(  1\right)  }\left(  \alpha\right)
	}\right)  b_{n,1}^{\left(  3\right)  }\left(  \alpha\right)  \right. \\
	&  ~\left.  -\left(  \varphi_{n,i}\left(  \alpha\right)  -\frac{1}%
	{b_{n,2}^{\left(  2\right)  }\left(  \alpha\right)  }\cdot\frac{\varphi
		_{n,i}^{\left(  1\right)  }\left(  \alpha\right)  }{b_{n,1}^{\left(  1\right)
		}\left(  \alpha\right)  }b_{n,1}^{\left(  2\right)  }\left(  \alpha\right)
	+\frac{\varphi_{n,i}^{\left(  2\right)  }\left(  \alpha\right)  }%
	{b_{n,2}^{\left(  2\right)  }\left(  \alpha\right)  }\right)  b_{n,2}^{\left(
		3\right)  }\left(  \alpha\right)  \right) \\
	=  &  ~\varphi_{n,i}\left(  \alpha\right)  -\frac{1}{b_{n,3}^{\left(
			3\right)  }\left(  \alpha\right)  }\left(  \frac{\varphi_{n,i}^{\left(
			1\right)  }\left(  \alpha\right)  }{b_{n,1}^{\left(  1\right)  }\left(
		\alpha\right)  }\left(  b_{n,1}^{\left(  3\right)  }\left(  \alpha\right)
	-\frac{b_{n,1}^{\left(  2\right)  }\left(  \alpha\right)  b_{n,2}^{\left(
			3\right)  }\left(  \alpha\right)  }{b_{n,2}^{\left(  2\right)  }\left(
		\alpha\right)  }\right)  +\frac{\varphi_{n,i}^{\left(  2\right)  }\left(
		\alpha\right)  }{b_{n,2}^{\left(  2\right)  }\left(  \alpha\right)  }%
	b_{n,2}^{\left(  3\right)  }\left(  \alpha\right)  \right)  +\frac
	{\varphi_{n,i}^{\left(  3\right)  }\left(  \alpha\right)  }{b_{n,3}^{\left(
			3\right)  }\left(  \alpha\right)  },\\
	i=  &  ~1,2,\ldots,n.
	\end{align*}
	One can observe that expressions corresponding to these special cases are
	in accordance with formula (\ref{eq:first_half}). Now, fix the column index
	$j\in\left\{  1,2,\ldots,\lfloor \frac{n}{2}\rfloor -1\right\}  $
	and assume that formula (\ref{eq:first_half}) is valid up to the selected
	index $j$ and we will also prove it for $j+1$. We can proceed as follows:%
	\[
	\varphi_{n,i}^{\left(  j+1\right)  }\left(  \alpha\right)  =\sum_{\gamma
		=0}^{j+1}t_{i,\gamma}^{n}b_{n,\gamma}^{\left(  j+1\right)  }\left(
	\alpha\right)  ,~i=0,1,\ldots,n,
	\]
	where $b_{n,j+1}^{\left(  j+1\right)  }\left(  \alpha\right)  \neq0$ and for
	the special subcase $i=0$ we have that%
	\[
	\sum_{\gamma=0}^{j+1}b_{n,\gamma}^{\left(  j+1\right)  }\left(  \alpha\right)
	=\varphi_{n,0}^{\left(  j+1\right)  }\left(  \alpha\right)  =0,
	\]
	i.e.,%
	\begin{align*}
	-b_{n,j+1}^{\left(  j+1\right)  }\left(  \alpha\right)  = &  ~\sum_{\gamma
		=0}^{j}b_{n,\gamma}^{\left(  j+1\right)  }\left(  \alpha\right)  ,\\
	& \\
	t_{i,j+1}^{n}= &  ~\frac{1}{b_{n,j+1}^{\left(  j+1\right)  }\left(
		\alpha\right)  }\cdot\left(  \varphi_{n,i}^{\left(  j+1\right)  }\left(
	\alpha\right)  -\sum_{\gamma=0}^{j}t_{i,\gamma}^{n}b_{n,\gamma}^{\left(
		j+1\right)  }\left(  \alpha\right)  \right)  \\
	& \\
	= &  ~\frac{\varphi_{n,i}^{\left(  j+1\right)  }\left(  \alpha\right)
	}{b_{n,j+1}^{\left(  j+1\right)  }\left(  \alpha\right)  }-\frac{1}%
	{b_{n,j+1}^{\left(  j+1\right)  }\left(  \alpha\right)  }\sum_{\gamma=0}%
	^{j}b_{n,\gamma}^{\left(  j+1\right)  }\left(  \alpha\right)  t_{i,\gamma}%
	^{n}\\
	& \\
	= &  ~\frac{\varphi_{n,i}^{\left(  j+1\right)  }\left(  \alpha\right)
	}{b_{n,j+1}^{\left(  j+1\right)  }\left(  \alpha\right)  }-\frac{1}%
	{b_{n,j+1}^{\left(  j+1\right)  }\left(  \alpha\right)  }\cdot\sum_{\gamma
		=0}^{j}b_{n,\gamma}^{\left(  j+1\right)  }\left(  \alpha\right)  \left(
	\varphi_{n,i}\left(  \alpha\right)  +\frac{\varphi_{n,i}^{\left(
			\gamma\right)  }\left(  \alpha\right)  }{b_{n,\gamma}^{\left(  \gamma\right)
		}\left(  \alpha\right)  }-\frac{1}{b_{n,\gamma}^{\left(  \gamma\right)
	}\left(  \alpha\right)  }\cdot\left.  \sum_{r=1}^{\gamma-1}\frac{\varphi
	_{n,i}^{\left(  r\right)  }\left(  \alpha\right)  }{b_{n,r}^{\left(  r\right)
	}\left(  \alpha\right)  }\right(  b_{n,r}^{\left(  \gamma\right)  }\left(
\alpha\right)  +\right.  \\
&  ~\left.  \left.  +\sum_{\ell=1}^{\gamma-r-1}\left(  -1\right)  ^{\ell}%
\sum_{r<k_{\gamma,1}<k_{\gamma,2}<\ldots<k_{\gamma,\ell}<\gamma}\frac
{b_{n,r}^{\left(  k_{\gamma,1}\right)  }\left(  \alpha\right)  b_{n,k_{\gamma
			,1}}^{\left(  k_{\gamma,2}\right)  }\left(  \alpha\right)  b_{n,k_{\gamma,2}%
	}^{\left(  k_{\gamma,3}\right)  }\left(  \alpha\right)  \ldots b_{n,k_{\gamma
		,\ell-1}}^{\left(  k_{\gamma,\ell}\right)  }\left(  \alpha\right)
b_{n,k_{\gamma,\ell}}^{\left(  \gamma\right)  }\left(  \alpha\right)
}{b_{n,k_{\gamma,1}}^{\left(  k_{\gamma,1}\right)  }\left(  \alpha\right)
b_{n,k_{\gamma,2}}^{\left(  k_{\gamma,2}\right)  }\left(  \alpha\right)
\ldots b_{n,k_{\gamma,\ell}}^{\left(  k_{\gamma,\ell}\right)  }\left(
\alpha\right)  }\right)  \right)  \\
& \\
= &  ~\frac{\varphi_{n,i}^{\left(  j+1\right)  }\left(  \alpha\right)
}{b_{n,j+1}^{\left(  j+1\right)  }\left(  \alpha\right)  }+\varphi
_{n,i}\left(  \alpha\right)  -\frac{1}{b_{n,j+1}^{\left(  j+1\right)  }\left(
	\alpha\right)  }\cdot\sum_{\gamma=0}^{j}b_{n,\gamma}^{\left(  j+1\right)
}\left(  \alpha\right)  \left(  \frac{\varphi_{n,i}^{\left(  \gamma\right)
}\left(  \alpha\right)  }{b_{n,\gamma}^{\left(  \gamma\right)  }\left(
\alpha\right)  }-\frac{1}{b_{n,\gamma}^{\left(  \gamma\right)  }\left(
\alpha\right)  }\cdot\left.  \sum_{r=1}^{\gamma-1}\frac{\varphi_{n,i}^{\left(
	r\right)  }\left(  \alpha\right)  }{b_{n,r}^{\left(  r\right)  }\left(
\alpha\right)  }\right(  b_{n,r}^{\left(  \gamma\right)  }\left(
\alpha\right)  +\right.  \\
&  ~\left.  \left.  +\sum_{\ell=1}^{\gamma-r-1}\left(  -1\right)  ^{\ell}%
\sum_{r<k_{\gamma,1}<k_{\gamma,2}<\ldots<k_{\gamma,\ell}<\gamma}\frac
{b_{n,r}^{\left(  k_{\gamma,1}\right)  }\left(  \alpha\right)  b_{n,k_{\gamma
			,1}}^{\left(  k_{\gamma,2}\right)  }\left(  \alpha\right)  b_{n,k_{\gamma,2}%
	}^{\left(  k_{\gamma,3}\right)  }\left(  \alpha\right)  \ldots b_{n,k_{\gamma
		,\ell-1}}^{\left(  k_{\gamma,\ell}\right)  }\left(  \alpha\right)
b_{n,k_{\gamma,\ell}}^{\left(  \gamma\right)  }\left(  \alpha\right)
}{b_{n,k_{\gamma,1}}^{\left(  k_{\gamma,1}\right)  }\left(  \alpha\right)
b_{n,k_{\gamma,2}}^{\left(  k_{\gamma,2}\right)  }\left(  \alpha\right)
\ldots b_{n,k_{\gamma,\ell}}^{\left(  k_{\gamma,\ell}\right)  }\left(
\alpha\right)  }\right)  \right)  \\
& \\
= &  ~\varphi_{n,i}\left(  \alpha\right)  -\frac{1}{b_{n,j+1}^{\left(
		j+1\right)  }\left(  \alpha\right)  }\cdot\left.  \sum_{r=1}^{j}\frac
{\varphi_{n,i}^{\left(  r\right)  }\left(  \alpha\right)  }{b_{n,r}^{\left(
		r\right)  }\left(  \alpha\right)  }\right(  b_{n,r}^{\left(  j+1\right)
}\left(  \alpha\right)  +\\
&  ~\left.  +\sum_{\ell=1}^{j-r}\left(  -1\right)  ^{\ell}\sum_{r<k_{1}%
	<k_{2}<\ldots<k_{\ell}<j+1}\frac{b_{n,r}^{\left(  k_{1}\right)  }\left(
	\alpha\right)  b_{n,k_{1}}^{\left(  k_{2}\right)  }\left(  \alpha\right)
	b_{n,k_{2}}^{\left(  k_{3}\right)  }\left(  \alpha\right)  \ldots
	b_{n,k_{\ell-1}}^{\left(  k_{\ell}\right)  }\left(  \alpha\right)
	b_{n,k_{\ell}}^{\left(  j+1\right)  }\left(  \alpha\right)  }{b_{n,k_{1}%
	}^{\left(  k_{1}\right)  }\left(  \alpha\right)  b_{n,k_{2}}^{\left(
	k_{2}\right)  }\left(  \alpha\right)  \ldots b_{n,k_{\ell}}^{\left(  k_{\ell
	}\right)  }\left(  \alpha\right)  }\right)  +\frac{\varphi_{n,i}^{\left(
	j+1\right)  }\left(  \alpha\right)  }{b_{n,j+1}^{\left(  j+1\right)  }\left(
\alpha\right)  },
\end{align*}
which means that our induction hypothesis is correct for all column indices
$j=1,2,\ldots,\lfloor \frac{n}{2}\rfloor $.

Using a similar technique based on backward substitutions, the entries of the columns 
$
[  t_{i,n-j}^{n}]  _{i=1}^{n},~j=1,2,\ldots,\lfloor \frac
{n}{2}\rfloor
$
can also iteratively be determined by differentiating the matrix equality
(\ref{eq:basis_transformation}) with gradually increasing order and applying
the Hermite conditions (\ref{eq:Hermite_conditions_alpha}) at $u=\beta$. After
correct reformulations one obtains exactly the formula (\ref{eq:last_half}).
\end{proof}

\bigskip

\begin{proof}
	[Proof of Theorem \ref{thm:integral_curves}]Using Theorem
	\ref{thm:basis_transformation} and Corollary \ref{cor:ordinary_functions}, the
	$\ell$th coordinate function ($\ell=1,2,\ldots,\delta$) of the ordinary
	integral curve (\ref{eq:ordinary_integral_curve}) can be rewritten into%
	\[
	c^{\ell}\left(  u\right)  =\sum_{i=0}^{n}\lambda_{i}^{\ell}\varphi
	_{n,i}\left(  u\right)  =\sum_{j=0}^{n}p_{j}^{\ell}b_{n,j}\left(  u\right)
	,~\forall u\in\left[  \alpha,\beta\right]  ,
	\]
	where%
	\[
	p_{j}^{\ell}=\sum_{i=0}^{n}\lambda_{i}^{\ell}t_{i,j}^{n},~j=0,1,\ldots,n.
	\]
	Repeating this transformation along all coordinate functions and collecting
	the coefficients of the normalized B-basis functions, one obtains the vertices
	$\mathbf{p}_{j}=\left[  p_{j}^{\ell}\right]  _{\ell=1}^{\delta}$ of the
	required control polygon.
\end{proof}

\bigskip

\begin{proof}
	[Proof of Theorem \ref{thm:integral_surfaces}]By means of Theorem
	\ref{thm:basis_transformation}, one can construct for all $r=1,2$ the regular
	transformation matrix $\left[  t_{i_{r},j_{r}}^{n_{r}}\right]  _{i_{r}%
		=0,~j_{r}=0}^{n_{r},~n_{r}}$ that maps the normalized B-basis $\mathcal{B}%
	_{n_{r}}^{\alpha_{r},\beta_{r}}$ of the vector space $\mathbb{S}_{n_{r}%
	}^{\alpha_{r},\beta_{r}}$ to its ordinary basis $\mathcal{F}_{n_{r}}%
	^{\alpha_{r},\beta_{r}}$. Observe that the $\ell$th coordinate function
	($\ell=1,2,3$) of the ordinary integral surface
	(\ref{eq:ordinary_integral_surface}) can be written in the form%
	\begin{align*}
	s^{\ell}\left(  \mathbf{u}\right)   &  ~=\sum_{\zeta=1}^{\sigma_{\ell}}\prod
	_{r=1}^{2}\left(  \sum_{i_{r}=0}^{n_{r}}\lambda_{i_{r}}^{\ell,\zeta}%
	\varphi_{n_{r},i_{r}}\left(  u_{r}\right)  \right)  
	=\sum_{\zeta=1}^{\sigma_{\ell}}\prod_{r=1}^{2}\left(  \sum_{j_{r}=0}^{n_{r}%
	}p_{j_{r}}^{\ell,\zeta}b_{n_{r},j_{r}}\left(  u_{r}\right)  \right)  \\
	&  ~=\sum_{j_{1}=0}^{n_{1}}\sum_{j_{2}=0}^{n_{2}}\left(  \sum_{\zeta
		=1}^{\sigma_{\ell}}\prod_{r=1}^{2}p_{j_{r}}^{\ell,\zeta}\right)  b_{n_{1},j_{1}%
	}\left(  u_{1}\right)  b_{n_{2},j_{2}}\left(  u_{2}\right)  
	=\sum_{j_{1}=0}^{n_{1}}\sum_{j_{2}=0}^{n_{2}}p_{j_{1},j_{2}}^{\ell}%
	b_{n_{1},j_{1}}\left(  u_{1}\right)  b_{n_{2},j_{2}}\left(  u_{2}\right)
	\end{align*}
	for all $\mathbf{u=}\left[  u_{r}\right]  _{r=1}^{2}\in\left[  \alpha
	_{1},\beta_{1}\right]  \times\left[  \alpha_{2},\beta_{2}\right]  $, where
	\[
	p_{j_{1},j_{2}}^{\ell}:=\sum_{\zeta=1}^{\sigma_{\ell}}\prod_{r=1}^{2}p_{j_{r}}%
	^{\ell,\zeta}%
	\]
	and the values%
	\[
	p_{j_{r}}^{\ell,\zeta}=\sum_{i_{r}=0}^{n_{r}}\lambda_{i_{r}}^{\ell,\zeta
	}t_{i_{r},j_{r}}^{n_{r}},~\zeta=1,2,\ldots,\sigma_{\ell},~j_{r}=0,1,\ldots
	,n_{r},~r=1,2
	\]
	can be obtained by means of Corollary \ref{cor:ordinary_functions}. Repeating
	this reformulation for all coordinate functions and collecting the
	coefficients of the product of normalized B-basis functions, one obtains all
	coordinates of all control net points $\mathbf{p}_{j_{1},j_{2}}=\left[
	p_{j_{1},j_{2}}^{\ell}\right]  _{\ell=1}^{3}$.
\end{proof}

\bigskip

\begin{proof}[Proof of Theorem \ref{thm:computational_complexity}]Let $i\in\{  1,2,\ldots,n\}$, $j\in\{  1,2,\ldots
	,\left\lfloor \frac{n}{2}\right\rfloor\}$, $r\in\{
	1,2,\ldots,j-1\}$ and $\ell\in\{  1,2,\ldots,j-r-1\}$ be
	fixed indices at the moment and consider formula (\ref{eq:first_half}). There are $\binom{j-r-1}{\ell}$ pairwise distinct
	strictly increasing sequences $r<k_{1}<k_{2}<\ldots<k_{\ell}<j$ of length
	$\ell$ between $r$ and $j$ that can also be stored in permanent lookup tables, since they are independent of the applied normalized B-basis. In case of each of these sequences one has to evaluate
	the fraction%
	\[
	f_{j,r,\ell}:=\frac{b_{n,r}^{\left(  k_{1}\right)  }\left(  \alpha\right)  b_{n,k_{1}%
		}^{\left(  k_{2}\right)  }\left(  \alpha\right)  b_{n,k_{2}}^{\left(
		k_{3}\right)  }\left(  \alpha\right)  \cdot\ldots\cdot b_{n,k_{\ell-1}%
	}^{\left(  k_{\ell}\right)  }\left(  \alpha\right)  b_{n,k_{\ell}}^{\left(
	j\right)  }\left(  \alpha\right)  }{b_{n,k_{1}}^{\left(  k_{1}\right)
}\left(  \alpha\right)  b_{n,k_{2}}^{\left(  k_{2}\right)  }\left(
\alpha\right)  \cdot\ldots\cdot b_{n,k_{\ell}}^{\left(  k_{\ell}\right)
}\left(  \alpha\right)  }%
\]
that includes $2\ell$ flops (i.e., $\ell$ multiplications in the nominator,
$\ell-1$ multiplications in the denominator and $1$ division). Thus, the total
number of flops required for the evaluation of the summation%
\begin{equation}
\label{eq:summation_cost_of_sequences_of_length_ell}
s_{j,r,\ell}:=\sum_{r<k_{1}<k_{2}<\ldots<k_{\ell}<j} f_{j,r,\ell}%
\end{equation}
equals%
\[
2\ell\cdot\binom{j-r-1}{\ell}+\left(  \binom{j-r-1}{\ell}-1\right)
\]
for each fixed values of $\ell=1,2,\ldots,r-1$, where the last term in the parentheses appears due to additions that have to
be performed in (\ref{eq:summation_cost_of_sequences_of_length_ell}). If one considers all possible values of $\ell$ and observes that
$\left(  -1\right)  ^{\ell}$ is just an alternating sign (implying either
addition or subtraction), the number of flops performed during the evaluation
of the expression%
\[
g_{j,r}:=\frac{b_{n,r}^{\left(  j\right)  }\left(  \alpha\right)  +\displaystyle\sum\limits
	_{\ell=1}^{j-r-1}\left(  -1\right)  ^{\ell}s_{j,r,\ell}}{b_{n,r}^{\left(
		r\right)  }\left(  \alpha\right)  }%
\]
is%
\begin{equation*}
\sum_{\ell=1}^{j-r-1}\left(  2\ell\cdot\binom{j-r-1}{\ell}+\binom
{j-r-1}{\ell}-1\right)  +\left[1+\left(  j-r-2\right)\right]  +1
=
~2^{j-r-1}\cdot\left(  j-r\right)
\end{equation*}
that consists of the evaluation cost of all $\{s_{j,r,\ell}\}_{\ell=1}^{j-r-1}$, of $1+\left(  j-r-2\right)=j-r-1$ additions and of $1$ division.

Observe that values $\left\{  g_{j,r}\right\}  _{r=1}^{j-1}$ are independent
of the row index $i$ for all fixed column indices  $j=1,2,\ldots,\left\lfloor\frac{n}{2}\right\rfloor$, i.e., they can be evaluated and stored in a temporary lookup table
by performing%
\[
\sum_{r=1}^{j-1}2^{j-r-1}\cdot\left(  j-r\right)  =2^{j-1}\cdot j-2^{j}+1
\]
flops and later they can be reused for the evaluation of all quantities%
\[
h_{i,j}:=\sum_{r=1}^{j-1}\varphi_{n,i}^{\left(  r\right)  }\left(
\alpha\right) \cdot  g_{j,r},~i=1,2,\ldots,n.%
\]
Thus, independently of $i$, the
calculation of $h_{i,j}$ takes an additional $2j-3$ flops (i.e., $j-1$ multiplication and
$j-2$ addition) for all fixed values of $j=1,2,\ldots,\left\lfloor \frac{n}%
{2}\right\rfloor $. Finally, each of the column entries%
\[
t_{i,j}^{n}=\varphi_{n,i}\left(  \alpha\right)  -\frac{h_{i,j}-\varphi
	_{n,i}^{\left(  j\right)  }\left(  \alpha\right)  }{b_{n,j}^{\left(  j\right)
	}\left(  \alpha\right)  },~i=1,2,\ldots,n
\]
can be evaluated by means of $3$ additional flops for all fixed values of $j=1,2,\ldots
,\left\lfloor \frac{n}{2}\right\rfloor $.

Thus, the total number of flops required for the evaluation of unknown entries
$\left[  t_{i,j}^{n}\right]  _{i=1,~j=1}^{n,~\left\lfloor \frac{n}%
	{2}\right\rfloor }$ of the general transformation matrix is%
\begin{equation}
\sum_{j=1}^{\left\lfloor \frac{n}{2}\right\rfloor }\left(  2^{j-1}\cdot
j-2^{j}+1\right)  +n\cdot\sum_{j=1}^{\left\lfloor \frac{n}{2}\right\rfloor }\left[
\left(  2j-3\right)  +3\right]  =2^{\left\lfloor \frac{n}{2}\right\rfloor
}\left(  \left\lfloor \frac{n}{2}\right\rfloor -3\right)  +\left\lfloor
\frac{n}{2}\right\rfloor +3+n\left\lfloor \frac{n}{2}\right\rfloor \left(
\left\lfloor \frac{n}{2}\right\rfloor +1\right)  .\label{eq:half_cost}%
\end{equation}

Since the structure of formula (\ref{eq:last_half}) is very similar to that of (\ref{eq:first_half}), one can conclude that for odd numbers $n$ the total computational cost of all unknown entries of the general transformation matrix
is twice of (\ref{eq:half_cost}), while for even values of $n$ the entries of the middle column do not have to be reevaluated by means of (\ref{eq:last_half}), i.e., in this latter case the partial computational cost (\ref{eq:half_cost}) has to be increased by
\[
\sum_{j=1}^{
	\frac{n}{2}
	-1}\left(  2^{j-1}\cdot
j-2^{j}+1\right)  +n\cdot\sum_{j=1}^{
	\frac{n}{2}
	-1}\left[
\left(  2j-3\right)  +3\right]  =2^{
	\frac{n}{2}
	-1
}\left(  
\frac{n}{2}
-4\right)  +
\frac{n}{2}
+2+
\frac{n^2}{2}\cdot
\left(
\frac{n}{2}
-1\right)  ,\label{eq:other_half_cost}
\]
which leads to the final expression (\ref{eq:total_computational_cost}).
\end{proof}%

\bigskip

\begin{proof}
	[Proof of Theorem \ref{thm:trigonometric_endpoint_derivatives}]In order to
	determine the higher order derivatives of normalized B-basis functions
	(\ref{eq:trigonometric_B-basis-n}) at the endpoints of the interval $\left[
	0,\beta\right]  $, we will make use of trigonometric identities%
	\begin{align*}
	\sin^{2r+1}\left(  \theta\right)   &  =\frac{2}{2^{2r+1}}\sum_{\ell=0}%
	^{r}\left(  -1\right)  ^{r-\ell}\binom{2r+1}{\ell}\sin\left(  \left(  2\left(
	r-\ell\right)  +1\right)  \theta\right)  ,\\
	\sin^{2r}\left(  \theta\right)   &  =\frac{1}{2^{2r}}\binom{2r}{r}+\frac
	{2}{2^{2r}}\sum_{\ell=0}^{r-1}\left(  -1\right)  ^{r-\ell}\binom{2r}{\ell}%
	\cos\left(  \left(  2\left(  r-\ell\right)  \right)  \theta\right)  ,
	\end{align*}
	where $r\in%
	\mathbb{N}
	$ and $\theta\in%
	\mathbb{R}
	$. E.g. if $i=2r+1$ ($r=0,1,\ldots,n-1$), then%
	\begin{align*}
	\frac{b_{2n,2r+1}\left(  u\right)  }{c_{2n,2r+1}^{\beta}}=  &  ~\sin^{2\left(
		n-r-1\right)  +1}\left(  \frac{\beta-u}{2}\right)  \sin^{2r+1}\left(  \frac
	{u}{2}\right) \\
	=  &  ~\left(  \frac{2}{2^{2\left(  n-r-1\right)  +1}}\sum_{k=0}%
	^{n-r-1}\left(  -1\right)  ^{n-r-1-k}\binom{2\left(  n-r-1\right)  +1}{k}%
	\sin\left(  \left(  2\left(  n-r-k\right)  -1\right)  \frac{\beta-u}%
	{2}\right)  \right) \\
	&  ~\cdot\left(  \frac{2}{2^{2r+1}}\sum_{\ell=0}^{r}\left(  -1\right)
	^{r-\ell}\binom{2r+1}{\ell}\sin\left(  \left(  2\left(  r-\ell\right)
	+1\right)  \frac{u}{2}\right)  \right) \\
	=  &  ~\frac{1}{2^{2\left(  n-1\right)  }}\sum_{k=0}^{n-r-1}\sum_{\ell=0}%
	^{r}\left(  -1\right)  ^{n+1-k-\ell}\binom{2\left(  n-r-1\right)  +1}{k}%
	\binom{2r+1}{\ell}\\
	&  ~\cdot\sin\left(  \left(  2\left(  n-r-k\right)  -1\right)  \frac{\beta
		-u}{2}\right)  \sin\left(  \left(  2\left(  r-\ell\right)  +1\right)  \frac
	{u}{2}\right) \\
	=  &  ~\frac{1}{2^{2n-1}}\sum_{k=0}^{n-r-1}\sum_{\ell=0}^{r}\left(  -1\right)
	^{n+1-k-\ell}\binom{2\left(  n-r-1\right)  +1}{k}\binom{2r+1}{\ell}\\
	&  ~\cdot\left(  \cos\left(  \left(  n-k-\ell\right)  u-\left(  2\left(
	n-r-k\right)  -1\right)  \frac{\beta}{2}\right)  \right. \\
	&  ~\left.  -\cos\left(  \left(  n-k-2r+\ell-1\right)  u-\left(  2\left(
	n-r-k\right)  -1\right)  \frac{\beta}{2}\right)  \right)  ,
	\end{align*}
	from which follows that%
	\begin{align*}
	\frac{b_{2n,2r+1}^{\left(  j\right)  }\left(  u\right)  }{c_{2n,2r+1}^{\beta}%
	}=  &  ~\frac{1}{2^{2n-1}}\sum_{k=0}^{n-r-1}\sum_{\ell=0}^{r}\left(
	-1\right)  ^{n+1-\left(  k+\ell\right)  }\binom{2\left(  n-r-1\right)  +1}%
	{k}\binom{2r+1}{\ell}\\
	&  ~\cdot\left(  \left(  n-\left(  k+\ell\right)  \right)  ^{j}\cos\left(
	\left(  n-\left(  k+\ell\right)  \right)  u-\left(  2\left(  n-r-k\right)
	-1\right)  \frac{\beta}{2}+\frac{j\pi}{2}\right)  \right. \\
	&  ~\left.  -\left(  n-k-2r+\ell-1\right)  ^{j}\cos\left(  \left(
	n-k-2r+\ell-1\right)  u-\left(  2\left(  n-r-k\right)  -1\right)  \frac{\beta
	}{2}+\frac{j\pi}{2}\right)  \right)
	\end{align*}
	for all $j\geq0$. Substituting $u=0$ into the last expression, one obtains
	exactly the formula (\ref{eq:trigonometric_B-basis_n_odd_0}). If $i=2r$
	($r=0,1,\ldots,n$), then one can proceed analogously.
\end{proof}

\bigskip

\begin{proof}
	[Proof of Theorem \ref{thm:hyperbolic_endpoint_derivatives}]In order to
	determine the higher order derivatives of the hyperbolic counterpart of the
	normalized B-basis functions (\ref{eq:trigonometric_B-basis-n}) (see also
	Subsection \ref{subsec:hyperbolic_polynomials}) at the endpoints of the
	interval $\left[  0,\beta\right]  $, one can follow the steps of the proof of
	Theorem \ref{thm:trigonometric_endpoint_derivatives} by applying the
	hyperbolic identities%
	\begin{align*}
	\sinh^{2r+1}\left(  \theta\right)   &  =\frac{2}{2^{2r+1}}\sum_{\ell=0}%
	^{r}\left(  -1\right)  ^{r+\left(  r\operatorname{mod}2\right)  -\ell}%
	\binom{2r+1}{\ell}\sinh\left(  \left(  2\left(  r-\ell\right)  +1\right)
	\theta\right)  ,\\
	\sinh^{2r}\left(  \theta\right)   &  =\frac{1}{2^{2r}}\binom{2r}{r}+\frac
	{2}{2^{2r}}\sum_{\ell=0}^{r+\left(  r\operatorname{mod}2\right)  -1}\left(
	-1\right)  ^{r+\left(  r\operatorname{mod}2\right)  -\ell}\binom{2r}{\ell
	}\cosh\left(  \left(  2\left(  r-\ell\right)  \right)  \theta\right)  
	\end{align*}
	and basic properties%
	\begin{align*}
	\sinh^{\left(  j\right)  }\left(  \omega u\right)   &  =\left\{
	\begin{array}
	[c]{cc}%
	\omega^{j}\sinh\left(  \omega u\right)  , & j\left(  \operatorname{mod}%
	2\right)  =0,\\
	\omega^{j}\cosh\left(  \omega u\right)  , & j\left(  \operatorname{mod}%
	2\right)  =1,
	\end{array}
	\right. \\
	\cosh^{\left(  j\right)  }\left(  \omega u\right)   &  =\left\{
	\begin{array}
	[c]{cc}%
	\omega^{j}\cosh\left(  \omega u\right)  , & j\left(  \operatorname{mod}%
	2\right)  =0,\\
	\omega^{j}\sinh\left(  \omega u\right)  , & j\left(  \operatorname{mod}%
	2\right)  =1,
	\end{array}
	\right. \\
	2\cosh\left(  \theta_{1}\right)  \cosh\left(  \theta_{2}\right)   &
	=\cosh\left(  \theta_{1}+\theta_{2}\right)  +\cosh\left(  \theta_{1}%
	-\theta_{2}\right)
	\end{align*}
	of the hyperbolic sine and cosine functions, where $r\in%
	\mathbb{N}
	$
	and $
	\theta, \theta_1,\theta_2,\omega\in%
	\mathbb{R}%
	$.
\end{proof}

\section{Final remarks\label{sec:final_remarks}}

As listed in Section \ref{sec:introduction}, concerning geometric modeling,
the normalized B-bases (of EC spaces that also comprise
the constant functions) ensure many optimal shape preserving properties and
algorithms. Moreover, they may also provide useful design or shape parameters that
can arbitrarily be specified by the user or the engineer. In Section
\ref{sec:examples}, we have seen that polynomial, trigonometric, hyperbolic or
mixed EC spaces allow us to obtain the control point based
exact description of many (rational) curves and surfaces that are important in
several areas of applied mathematics. The investigated large
classes of vector spaces also ensure the description of famous geometrical objects (like
ellipses; epi- and hypocycloids; Lissajous curves; torus knots; foliums; rose
curves; the witch of Agnesi; the cissoid of Diocles; Bernoulli's lemniscate;
Zhukovsky airfoil profiles; cycloids; hyperbolas; helices; catenaries; Archimedean and logarithmic spirals;
ellipsoids; tori; hyperboloids; catenoids; helicoids; ring, horn and spindle
Dupin cyclides; non-orientable surfaces such as Boy's and Steiner's surfaces
and the Klein Bottle of Gray).

EC bases of type (\ref{eq:ordinary_basis}) represent a
large family of vector spaces that can be used in real-world applications,
e.g.\ besides of examples described in Section \ref{sec:examples}, general formulas of Theorem
\ref{thm:basis_transformation} can also be applied in the exponential space %
$
\left\langle \left\{  1,e^{\lambda_{1}u},e^{\lambda_{2}u},\ldots
,e^{\lambda_{n}u}:u\in\left[  \alpha,\beta\right]  \right\}  \right\rangle
,~0<\lambda_{1}<\lambda_{2}<\ldots<\lambda_{n},~\alpha<\beta,
$
or in the space %
$
\left\langle \left\{  1,u^{\lambda_{1}},u^{\lambda_{2}},\ldots,u^{\lambda_{n}%
}:u\in\left[  \alpha,\beta\right]  \right\}  \right\rangle ,~0<\lambda_{1}%
<\lambda_{2}<\ldots<\lambda_{n},~\left[\alpha,\beta\right]\subset\left(0,\infty\right)
$
of restricted M\"{u}ntz polynomials among many others.

Storing in permanent lookup tables the zeroth and higher order endpoint derivatives of
the ordinary EC basis (\ref{eq:ordinary_basis}) and
of the normalized B-basis (\ref{eq:B-basis}) induced by it, general formulas (\ref{eq:first_half}%
)--(\ref{eq:last_half}) and the proposed control point based curve/surface
modeling tools can efficiently be implemented up to $n=15$, even by means of a sequential algorithm. If one uses multi-threading, the value of $n$, for which one can provide an efficient implementation, can be higher. For arbitrarily large values of $n$, the presented results are mainly of theoretical interest.


\appendix

\renewcommand*{\thesection}{\Alph{section}}

\section{Mapping Bernstein polynomials to the ordinary power basis: revisited}
\label{sec:Bernstein_to_monomials}
\normalsize
In what follows, we show through direct calculations that the substitution of derivatives
(\ref{eq:monomial_derivative})--(\ref{eq:Bernstein_derivative}) into formulas (\ref{eq:first_half})--(\ref{eq:last_half}) leads to the classical transformation matrix (\ref{eq:Bernstein_to_monomials}) that maps the Bernstein polynomials of degree $n$ to the ordinary monomials. In order to perform these direct calculations, we will use the following two lemmas. (Naturally, by using simple algebraic manipulations, one can obtain this basis transformation in a much easier way. However, we thought it would be interesting to revisit this property in the general context of Theorem \ref{thm:basis_transformation}.)

\begin{lemma}
	[{Sum of multinomial coefficients over proper partitions; \cite[Theorem 2.2]{KaoZetterberg1957}}]For all $p\geq q$ we have that%
	\begin{equation}
	\sum_{s_{q,p}\in S_{q,p}}\frac{p!}{\prod_{z=1}^{q}k_{z}!}=\sum_{\gamma
		=0}^{q-1}\left(  -1\right)  ^{\gamma}\binom{q}{\gamma}\left(  q-\gamma\right)
	^{p}, \label{eq:sum_of_mulitnomial_coefficients}%
	\end{equation}
	where $s_{q,p}$ is a proper $q$-fold partition of $p$, i.e., $s_{q,p}$ is an
	ordered set of non-negative integers $\left\{  k_{z}\right\}  _{z=1}^{q}$ such
	that %
	$
	\sum_{z=1}^{q}k_{z}=p
	$
	and %
	$
	k_{\ell}\geq1,~\forall\ell=1,2,\ldots,q.
	$
	($S_{q,p}$ denotes the set of all proper $q$-fold partitions of $p$.)
\end{lemma}

\begin{lemma}
	For all orders $p\geq1$, one has that%
	\begin{equation}
	\left.  \frac{\text{\emph{d}}^{p}}{\text{\emph{d}}u^{p}}\frac{\left(
		1-e^{u}\right)  ^{p+1}-\left(  1-e^{u}\right)  ^{2}}{e^{u}}\right\vert
	_{u=0}=\left\{
	\begin{array}
	[c]{rl}%
	-2, & p\left(  \operatorname{mod}2\right)  =0,\\
	0, & p\left(  \operatorname{mod}2\right)  =1
	\end{array}
	\right.  =\frac{1}{\left(  -1\right)  ^{p+1}}-1\text{.}
	\label{eq:exp_derivative}%
	\end{equation}
	
\end{lemma}

From hereon, a number in paranthesis above the equality sign indicates that we
apply the corresponding identity. For example, when $2\leq j\leq\left\lfloor
\frac{n}{2}\right\rfloor $ and $i=1,2,\ldots,j-1$, one has that%
\begin{footnotesize}
\begin{align*}
\def\arraystretch{1.75}
t_{i,j}^{n}
\underset{\text{(\ref{eq:monomial_derivative})}}{\overset{\text{(\ref{eq:first_half})}}{=}} &  ~-\frac{1}{b_{n,j}^{\left(
		j\right)  }\left(  0\right)  }\frac{\varphi_{n,i}^{\left(  i\right)  }\left(
	0\right)  }{b_{n,i}^{\left(  i\right)  }\left(  0\right)  }\left(
b_{n,i}^{\left(  j\right)  }\left(  0\right)  +\sum_{\ell=1}^{j-i-1}\left(
-1\right)  ^{\ell}\sum_{i<k_{1}<k_{2}<\ldots<k_{\ell}<j}\frac{b_{n,i}^{\left(  k_{1}\right)  }\left(  0\right)
	b_{n,k_{1}}^{\left(  k_{2}\right)  }\left(  0\right)  b_{n,k_{2}}^{\left(
		k_{3}\right)  }\left(  0\right)  \ldots b_{n,k_{\ell-1}}^{\left(  k_{\ell
		}\right)  }\left(  0\right)  b_{n,k_{\ell}}^{\left(  j\right)  }\left(
	0\right)  }{b_{n,k_{1}}^{\left(  k_{1}\right)  }\left(  0\right)  b_{n,k_{2}%
	}^{\left(  k_{2}\right)  }\left(  0\right)  \ldots b_{n,k_{\ell}}^{\left(
	k_{\ell}\right)  }\left(  0\right)  }\right)  \\
\underset{\text{(\ref{eq:Bernstein_derivative})}}{\overset
	{\text{(\ref{eq:monomial_derivative})}}{=}} &  -\frac{1}{j!\binom{n}{j}}\frac
{i!}{i!\binom{n}{i}}\left(  \left(  -1\right)  ^{j-i}j!\binom{n}{j}\binom
{j}{i}
+\left(  -1\right)  ^{j-i}j!\binom{n}{j}\sum_{\ell=1}%
^{j-i-1}\left(  -1\right)  ^{\ell}\sum_{i<k_{1}<k_{2}<\ldots<k_{\ell}<j}%
\binom{k_{1}}{i}\binom{k_{2}}{k_{1}}\ldots\binom{k_{\ell}}{k_{\ell-1}}%
\binom{j}{k_{\ell}}\right)  \\
= &  ~\frac{\left(  -1\right)  ^{j-i+1}}{\binom{n}{i}}\left(  \binom{j}%
{i}+\sum_{\ell=1}^{j-i-1}\left(  -1\right)  ^{\ell}\sum_{i<k_{1}<k_{2}%
	<\ldots<k_{\ell}<j}\binom{k_{1}}{i}\binom{k_{2}}{k_{1}}\ldots\binom{k_{\ell}%
}{k_{\ell-1}}\binom{j}{k_{\ell}}\right)  \\
= &  ~\frac{\left(  -1\right)  ^{j-i+1}}{\binom{n}{i}}\left(  \binom{j}%
{i}+\sum_{\ell=1}^{j-i-1}\left(  -1\right)  ^{\ell}\sum_{i<k_{1}<k_{2}%
	<\ldots<k_{\ell}<j}\frac{j!}{i!\left(  k_{1}-i\right)  !\left(  k_{2}%
	-k_{1}\right)  !\ldots\left(  j-k_{\ell}\right)  !}\right)  \\
= &  ~\frac{\left(  -1\right)  ^{j-i+1}}{\binom{n}{i}}\binom{j}{i}\left(
1+\sum_{\ell=1}^{j-i-1}\left(  -1\right)  ^{\ell}\sum_{i<k_{1}<k_{2}%
	<\ldots<k_{\ell}<j}\frac{\left(  j-i\right)  !}{\left(  k_{1}-i\right)
	!\left(  k_{2}-k_{1}\right)  !\ldots\left(  j-k_{\ell}\right)  !}\right)  \\
\underset{q=\ell+1,~p=j-i}{\overset
	{\text{(\ref{eq:sum_of_mulitnomial_coefficients})}}{=}} &  ~\frac{\left(
	-1\right)  ^{j-i+1}}{\binom{n}{i}}\binom{j}{i}\left(  1+\sum_{\ell=1}%
^{j-i-1}\left(  -1\right)  ^{\ell}\sum_{\gamma=0}^{\ell}\left(  -1\right)
^{\gamma}\binom{\ell+1}{\gamma}\left(  \ell+1-\gamma\right)  ^{j-i}\right)  \\
= &  ~\frac{\left(  -1\right)  ^{j-i+1}}{\binom{n}{i}}\binom{j}{i}\left(
1+\sum_{\ell=1}^{j-i-1}\left(  -1\right)  ^{\ell}\sum_{\gamma=0}^{\ell}\left(
-1\right)  ^{\gamma}\binom{\ell+1}{\gamma}\left.  \frac{\text{d}^{j-i}%
}{\text{d}u^{j-i}}e^{\left(  \ell+1-\gamma\right)  u}\right\vert
_{u=0}\right)  \\
= &  ~\frac{\left(  -1\right)  ^{j-i+1}}{\binom{n}{i}}\binom{j}{i}\left(
1+\frac{\text{d}^{j-i}%
}{\text{d}u^{j-i}}\left(\sum_{\ell=1}^{j-i-1}\left(  -1\right)  ^{\ell}\left.  \sum_{\gamma=0}^{\ell}\left(  -1\right)  ^{\gamma}%
\binom{\ell+1}{\gamma}e^{\left(  \ell+1-\gamma\right)  u}\right)\right\vert
_{u=0}\right)  \\
= &  ~\frac{\left(  -1\right)  ^{j-i+1}}{\binom{n}{i}}\binom{j}{i}\left(
1+\left.  \frac{\text{d}^{j-i}}{\text{d}u^{j-i}}\left(\sum_{\ell=1}^{j-i-1}\left(
-1\right)  ^{\ell}\left(  e^{u}-1\right)  ^{\ell+1}\right)\right\vert _{u=0}\right)
\\
= &  ~\frac{\left(  -1\right)  ^{j-i+1}}{\binom{n}{i}}\binom{j}{i}\left(
1+\left.  \frac{\text{d}^{j-i}}{\text{d}u^{j-i}}\left(  \left(  e^{u}%
-1\right)  \sum_{\ell=1}^{j-i-1}\left(  1-e^{u}\right)  ^{\ell}\right)
\right\vert _{u=0}\right)  \\
= &  ~\frac{\left(  -1\right)  ^{j-i+1}}{\binom{n}{i}}\binom{j}{i}\left(
1+\left.  \frac{\text{d}^{j-i}}{\text{d}u^{j-i}}\left(  \left(  e^{u}%
-1\right)  \frac{\left(  1-e^{u}\right)  -\left(  1-e^{u}\right)  ^{j-i}%
}{1-\left(  1-e^{u}\right)  }\right)  \right\vert _{u=0}\right)  \\
= &  ~\frac{\left(  -1\right)  ^{j-i+1}}{\binom{n}{i}}\binom{j}{i}\left(
1+\left.  \frac{\text{d}^{j-i}}{\text{d}u^{j-i}}\frac{\left(  1-e^{u}\right)
	^{j-i+1}-\left(  1-e^{u}\right)  ^{2}}{e^{u}}\right\vert _{u=0}\right)  \\
\underset{p=j-i}{\overset{\text{(\ref{eq:exp_derivative})}}{=}} &
~\frac{\left(  -1\right)  ^{j-i+1}}{\binom{n}{i}}\binom{j}{i}\left(
1+\frac{1}{\left(  -1\right)  ^{j-i+1}}-1\right)  \\
= &  ~\frac{\binom{j}{i}}{\binom{n}{i}}.
\def\arraystretch{1.0}
\end{align*}
\end{footnotesize}

The remaining cases can be handled in a similar way.

\section{A simple Matlab example}
\label{sec:implementation_details}

In order to ease the reviewing process, Listing \ref{lst:ExponentialTrigonometricSpace} provides implementation details, by means of which one can verify the numerical values listed in Example \ref{exmp:exponential_trigonometric_space}. Note that, by using effective multi-threaded object oriented C++ programming and OpenGL rendering techniques,  even general cases can similarly be treated. With this illustrative Matlab code we opted for simplicity. Listing \ref{lst:ExponentialTrigonometricSpace} is based also on Doolittle's $LU$ decomposition which is implemented in Listing \ref{lst:Doolittle} for the sake of convenience.

\mnumberedcode{1}{Doolittle.m}{lst:Doolittle}{$\mathbf{LU}$ factorization of a regular real square matrix by means of Doolittle's algorithm}

\mnumberedcode{1}{ExponentialTrigonometricSpace_v3.m}{lst:ExponentialTrigonometricSpace}{Transforming the exponential trigonometric normalized B-basis (\ref{eq:5th_order_exp_trig_normalized_B_basis}) to the ordinary basis (\ref{eq:5th_order_exp_trig_ordinary_basis})}

\begin{flushleft}
	\includegraphics[scale = 1.0]{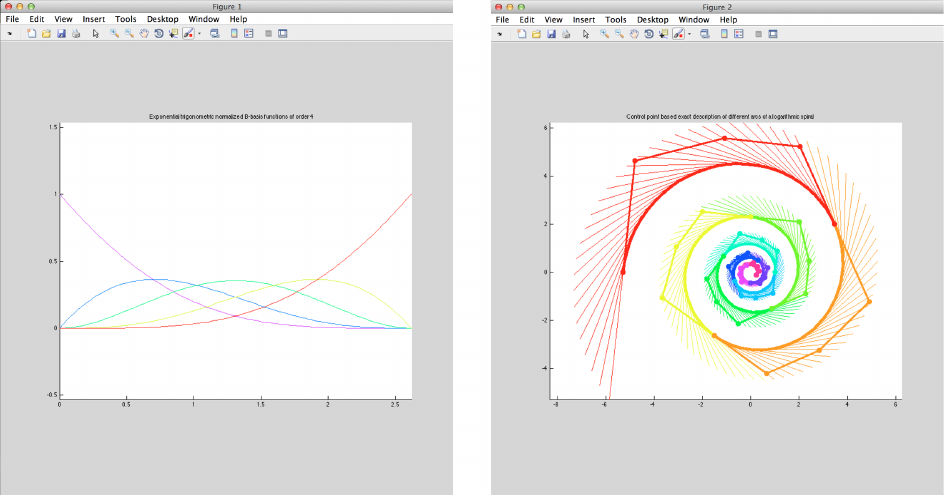}
\end{flushleft}


\begin{thebibliography}{99}                                               
	\small
	{                              %
		
		\bibitem{AitHadou2014}Ait-Hadou, R., 2014. \textit{Dimension elevation in M\"untz spaces: a new emergence of the M\"untz condition}, Journal of Approximation Theory, \textbf{181}:6--17.
		
		\bibitem{BarrioPena2004}Barrio, R., Pe\~{n}a, J.-M., 2004. \textit{Basis conversions among univariate polynomial representations}, Compte rendus de l'Academie des Sciences de Paris -- Series I,  \textbf{339}(4):293-298.
		
		\bibitem {deBoor1990}de Boor, C., 1990. \textit{Cutting corners always works,} Computer Aided Geometric Design, \textbf{4}(1--2):125--131.
		
		\bibitem{CargoShisha1966}Cargo, G.T., Shisha, O., 1966. \textit{The Bernstein form of a polynomial}. Journal of Research of the National Bureau of Standards, \textbf{70B}(1):79--81.
		
		\bibitem{CarnicerMainarPena2003}Carnicer, J.M., Mainar, E., Pe\~{n}a, J.M., 2003. \textit{Representing circles with five control points}. Computer Aided Geometric Design, \textbf{20}(8--9):501--511.
		
		\bibitem {CarnicerMainarPena2004}Carnicer, J.-M., Mainar, E., Pe\~{n}a, J.-M.,
		2004. \textit{Critical length for design purposes and extended Chebyshev
		spaces}. Constructive Approximation, \textbf{20}(1):55--71.
		
		\bibitem{CarnicerMainarPena2006}Carnicer, J.M., Mainar, E., Pe\~{n}a, J.M., 2006. \textit{Optimal bases of spaces with trigonometric functions}.  Pre-publicaciones del Seminario Matem\'atico ``Garc\'ia de Galdeano'', No. 29, 12 pages (in English),\\ \url{http://www.unizar.es/galdeano/preprints/2006/preprint29.pdf}.

		\bibitem {CarnicerMainarPena2007}Carnicer, J.-M., Mainar, E., Pe\~{n}a, J.-M., 2007. \textit{Shape preservation regions for six-dimensional spaces}.
		Advances in Computational Mathematics, \textbf{26}(1--3):121--136.
		

		\bibitem {Carnicer1993}Carnicer, J.-M., Pe\~{n}a, J.-M., 1993. \textit{Shape\ preserving representations and optimality of the Bernstein basis}. Advances in Computational Mathematics, \textbf{1}(2):173--196.
		
		\bibitem {CarnicerPena1994}Carnicer, J.-M., Pe\~{n}a, J.-M., 1994. \textit{Totally positive bases for shape preserving curve design and optimality of B-splines}. Computer Aided Geometric Design, \textbf{11}(6):633--654.
		
		\bibitem {CarnicerPena1995}Carnicer, J.-M., Pe\~{n}a, J.-M., 1995. \textit{On
			transforming a Tchebycheff system into a strictly totally positive system}.
		Journal of Approximation Theory, \textbf{81}(2):274--295.
		
		\bibitem{CoppersmithWinograd1990}Coppersmith, D., Winograd, S., 1990. \emph{Matrix multiplications via arithmetic progressions}. Journal of Symbolic Computation, \textbf{9}(3):251-280.
		
		\bibitem{CostantiniLycheManni2005}Costantini, P., Lyche, T., Manni, C., 2005. \textit{On a class of weak Tchebycheff systems}. Numerische Mathematik, \textbf{101}(2):333--354.	
		
		\bibitem {KaoZetterberg1957}Kao, R.C., Zetterberg, L.H., 1957.
		\textit{An identity for the sum  of multinomial coefficients}.
		American Mathematical Monthly, \textbf{64}(2):96--100.
		
		\bibitem {KarlinStudden1966}Karlin, S., Studden, W., 1966. \textit{Tchebycheff
			systems: with applications in analysis and statistics}. Wiley, New York.
		
		\bibitem{LuWangYang2002}L\"u, Y., Wang, G., Yang, X., 2002. \textit{Uniform hyperbolic polynomial B-spline curves}. Computer Aided Geometric Design, \textbf{19}(6):379--393.
		
		\bibitem{Lyche1985}Lyche, T., 1985. \textit{A recurrence relation for Chebyshevian B-splines}. Constructive Approximation, \textbf{1}(1):155--173.
		
		\bibitem{MainarPenaSanchez2001}Mainar, E., Pe\~{n}a, J.M., S\'anchez-Reyes, J., 2001. \textit{Shape preserving alternatives to the rational B\'ezier model}. Computer Aided Geometric Design, \textbf{18}(1):37--60.
		
		\bibitem{MainarPena2004}Mainar, E., Pe\~{n}a, J.M., 2004. \textit{Quadratic-cycloidal curves}. Advances in Computational Mathematics, \textbf{20}(1--3):161--175.
		
		\bibitem {MainarPena2010}Mainar, E., Pe\~{n}a, J.M., 2010. \textit{Optimal
			bases for a class of mixed spaces and their associated spline spaces}.
		Computers and Mathematics with Applications, \textbf{59}(4):1509--1523.
		
		\bibitem{ManniPelosiSampoli2011}Manni, C., Pelosi, F., Sampoli, M.L., 2011. \textit{Generalized B-splines as a tool in isogeometric analysis}. Computer Methods in Applied Mechanics and Engineering, \textbf{200}(5--8):867--881.
		
		\bibitem {Mazure1999}Mazure, M.-L., 1999. \textit{Chebyshev--Bernstein bases}. Computer
		Aided Geometric Design, \textbf{16}(7):649--669.
		
		\bibitem{Mazure2001}Mazure, M.-L., 2001. \textit{Chebyshev splines beyond total positivity}. Advances in Computational Mathematics, \textbf{14}(2):129--156.
		
		\bibitem {Pena1999}Pe\~{n}a, J.M., 1999. \textit{Shape Preserving
			Representations in Computer-Aided Geometric Design}. Nova Science Publishers,
		Commack NY.
		
		\bibitem{PottmannWagner1994}Pottmann, H., Wagner M.G., 1994. \textit{Helix splines as an example of affine Tchebycheffian splines}. Advances in Computational Mathematics, \textbf{2}(1):123--142.
		
		\bibitem{Press2007}Press, W.H., Teukolsky S.A., Vetterling W.T., Flannery, B.P., 2007. \emph{Numerical recipes. The art of scientific computing (3rd ed.)}. Cambridge University Press.
		
		\bibitem{RomaniSainiAlbrecht2014}Romani, L., Saini, L., Albrecht, G., 2014. \textit{Algebraic-trigonometric Pythagorean-hodograph curves and their use for Hermite interpolation}. Advances in Computational Mathematics, \textbf{40}(5--6):977--1010.
				
		\bibitem{Roth2015}R\'oth, \'A., 2015. \textit{Control point based exact description of trigonometric/hyperbolic curves, surfaces and volumes}. Journal of Computational and Applied Mathematics, \textbf{290}(C):74--91.
		
		\bibitem {Sanchez1998}S\'{a}nchez-Reyes, J., 1998. \textit{Harmonic rational
			B\'{e}zier curves, p-B\'{e}zier curves and trigonometric polynomials}.
		Computer Aided Geometric Design, \textbf{15}(9):909--923.
		
		\bibitem{Sanchez1999}S\'{a}nchez-Reyes, J., 1999. \emph{B\'ezier representation of epitrochoids and hypotrochoids}. Computer-Aided Design, \textbf{31}:747--750.
		
		\bibitem{Sanchez2002}S\'{a}nchez-Reyes, J., 1999. \emph{$p$-B\'ezier curves, spirals, and sectrix curves}. Computer Aided Geometric Design, \textbf{19}(6):445--464.
		
		\bibitem {ShenWang2005}Shen, W.-Q., Wang G.-Z., 2005. \textit{A class of
			B\'{e}zier curves based on hyperbolic polynomials}. Journal of Zhejiang
		University SCIENCE, 6A(Suppl. I), 116--123.
		
		\bibitem{Strassen1969}Strassen, V., 1969. \emph{Gaussian elimination is not optimal.} Numerische Mathematik, \textbf{14}(3):354--356.
		
		\bibitem{Vassilevska2012}Vassilevska Williams, V., 2012. \emph{Multiplying matrices faster than Coppersmith-Winograd}. In STOC'12 Proceedings of the 44th annual ACM Symposium on Theory of Computing, ACM New York, NY, USA, ISBN 978-1-4503-1245-5, pp. 887--898.
	}
	
	\bibitem{Zhang1996}Zhang, J., 1996. \textit{C-curves: an-extension of cubic curves}. Computer Aided Geometric Design, \textbf{13}(3):199--217.
	
\end{thebibliography}
\end{document}